\documentclass[10pt,a4paper,twoside,reqno]{amsart}

\usepackage[a4paper,
left=1in,right=1in,top=1.2in,bottom=1.2in,%
footskip=.25in]{geometry}

\usepackage{bm}
\usepackage{latexsym, amsmath, amstext, amssymb, amsfonts, amscd, bm, array, multirow, amsbsy, mathrsfs}
\usepackage{amsthm}
\usepackage{t1enc}
\usepackage[mathscr]{eucal}
\usepackage{indentfirst}
\usepackage{pb-diagram}
\usepackage{graphicx}
\usepackage{fancyhdr}
\usepackage{fancybox}
\usepackage{color}
\usepackage{tikz-cd}
\usetikzlibrary{arrows}
\usepackage[all]{xy}
\usepackage{hyperref}
\usepackage{tikz}
\usepackage{xparse}
\hypersetup{colorlinks=false,pdfborderstyle={/S/U/W 0}}
\usetikzlibrary{matrix}
\usepackage{upgreek}
\usepackage[utf8]{inputenc}
\usepackage[T1]{fontenc}
\usepackage{bbm}
\usepackage{enumitem}
\setlist[enumerate]{leftmargin=*,noitemsep}
\setlist[itemize]{leftmargin=*,noitemsep}
\usepackage{float}

\newcommand{\doi}[1]{\url{https://doi.org/#1}}
\usepackage[noadjust]{cite}

\makeatletter
\def\@defaultbiblabelstyle#1{[#1]}
\makeatother

\usepackage{url}
\theoremstyle{plain}
\newtheorem{thm}{Theorem}[section]

\newtheorem{prop}[thm]{Proposition}

\newtheorem{lemma}[thm]{Lemma}
\newtheorem{cor}[thm]{Corollary}

\theoremstyle{definition}
\newtheorem{definition}[thm]{Definition}

\theoremstyle{remark}
\newtheorem{remark}[thm]{Remark}
\newtheorem{example}[thm]{Example}

\newtheorem*{ack}{Acknowledgements}

\newcommand{\End}{\mathrm{End}}
\newcommand{\Aut}{\mathrm{Aut}}

\newcommand{\Mat}{\mathrm{Mat}}

\DeclareFontFamily{U}{rsf}{}
\DeclareFontShape{U}{rsf}{m}{n}{<5> <6> rsfs5 <7> <8> <9> rsfs7 <10-> rsfs10}{}
\DeclareMathAlphabet\Scr{U}{rsf}{m}{n}

\def\Z{\mathbb{Z}}
\def\C{\mathbb{C}}
\def\R{\mathbb{R}}

\def\rk{{\rm rk}}

\def\dd{\mathrm{d}}

\def\Ad{\mathrm{Ad}}

\def\Id{\mathrm{Id}}

\def\frq{\mathfrak{q}}

\def\frg{\mathfrak{g}}

\newcommand{\be}{\begin{equation*}}
\newcommand{\ee}{\end{equation*}}
\newcommand{\ben}{\begin{equation}}
\newcommand{\een}{\end{equation}}
\newcommand{\beqa}{\begin{eqnarray*}}
\newcommand{\eeqa}{\end{eqnarray*}}
\newcommand{\beqan}{\begin{eqnarray}}
\newcommand{\eeqan}{\end{eqnarray}}

\newcommand{\Tr}{\mathrm{Tr}}

\def\cC{{\mathcal C}}

\def\scB{\Scr B}
\def\scS{\Scr S}

\def\Cl{\mathrm{Cl}}

\def\Spin{\mathrm{Spin}}

\def\Spin{\mathrm{Spin}}

\def\SO{\mathrm{SO}}

\def\U{\mathrm{U}}

\def\cA{\mathcal{A}}
\def\cE{\mathcal{E}}
\def\hcE{\widehat{\mathcal{E}}}
\def\cI{\mathcal{I}}
\def\cP{\mathcal{P}}

\def\cT{\mathcal{T}}
\def\cP{\mathcal{P}}
\def\cF{\mathcal{F}}
\def\cC{\mathcal{C}}
\def\SU{\mathrm{SU}}

\def\G_2{\mathrm{G_2}}

\def\cV{\mathcal{V}}

\newcommand{\Hom}{{\rm Hom}}

\def\Aut{\mathrm{Aut}}

\def\Re{\mathrm{Re}}
\def\Im{\mathrm{Im}}

\def\G{\mathrm{G}}

\def\R{\mathbb{R}}

\def\dd{\mathrm{d}}
\newcommand{\abs}[1]{\left\lvert#1\right\rvert}

\newcommand{\escal}[1]{\langle#1\rangle}

\def\Span{\mathrm{Span}}

\def\frc{\mathfrak{F}}
\def\frc{\mathfrak{c}}

\newcolumntype{P}[1]{>{\centering\arraybackslash}p{#1}}

\usepackage[math]{anttor}
\usepackage{orcidlink}

\usepackage{mlmodern}

\begin{document}


\title[The algebraic square of an irreducible complex spinor]{The algebraic square of an irreducible complex spinor}

\author[Alejandro Gil-García]{Alejandro Gil-García \orcidlink{0000-0002-9370-241X}}
\address{Beijing Institute of Mathematical Sciences and Applications, Beijing, China}
\email{alejandrogilgarcia@bimsa.cn}

\author[C. S. Shahbazi]{C. S. Shahbazi \orcidlink{0000-0003-1185-9569}}
\address{Departamento de Matem\'aticas, Universidad UNED - Madrid, Reino de Espa\~na}
\email{cshahbazi@mat.uned.es}  
 
\begin{abstract}
We characterize, in every dimension and signature, the \emph{algebraic squares} of an irreducible complex spinor as a pair of exterior forms satisfying a prescribed system of algebraic relations that we present in terms of the geometric product of the underlying quadratic vector space. As a result, we obtain a general correspondence between irreducible complex spinors and algebraically constrained exterior forms, which clarifies the subtle relationship between spinors and exterior forms and contributes towards the understanding of spinors as the \emph{square root} of geometry. We use this formalism to construct the squares of an irreducible complex spinor in Euclidean dimensions up to six, and also to construct the squares of a generic, possibly non-pure and non-unit, irreducible complex chiral spinor in eight Euclidean dimensions. Elaborating on this result, we consider a natural notion of \emph{spinorial instanton} that we study for connections on a principal bundle with a complex structure group as well as for curvings of a $\mathbb{C}^{\ast}$-bundle gerbe defined on a Lorentzian six-manifold.\bigskip

\medskip


\end{abstract}

\maketitle

\setcounter{tocdepth}{1} 

\tableofcontents


\section{Introduction}


The main purpose of this article is to construct and algebraically characterize the \emph{squares} of an irreducible complex spinor $\eta\in\Sigma$ for an irreducible complex Clifford module $(\Sigma,\gamma)$ associated to a real quadratic vector space $(V,h)$\footnote{Here $V$ is a real vector space of dimension $d$ and $h$ a non-degenerate metric of signature $(p,q)$ of $p$ \emph{pluses} and $q$ \emph{minuses}.}. By \emph{construct} we mean establishing the existence of a quadratic map, depending on a choice of an admissible pairing on $(\Sigma,\gamma)$, that to every irreducible complex spinor assigns a complex exterior form, that is:
\begin{equation*}
\Sigma \ni \eta \mapsto \alpha_{\eta}\in \wedge V^{\ast}_{\mathbb{C}}.
\end{equation*}

\noindent
This defines a map, a \emph{spinor square map}, whose image consists of a very special class of exterior forms that translates the geometric information contained in $\eta$ into the language of exterior forms. By \emph{completely characterizing} the square of an irreducible complex spinor, we mean characterizing the image of this map inside $\wedge V^{\ast}_{\mathbb{C}}$ as the solution set of a system of explicit algebraic conditions. We achieve this goal through a set of algebraic equations and a fundamental inequality, all of which are neatly packed in terms of the complex linear extension of the geometric product $\diamond$ in $(V,h)$. Spinor square maps sit at the core of the current understanding of spinors as the \emph{square root} of geometry, that is, as the square root of differential forms, as exemplified for instance in \cite[\textsection 14]{HarveyBook} and \cite[\textsection IV]{Spin89}. The study of the intimate relationship between irreducible spinors and differential forms has a long tradition that can be traced back to Cartan \cite{Cartan1938} and Chevalley \cite{Chev54}. It has been revisited in more modern times by Trautman \cite{Trautman1997} and Atiyah \cite{AtiyahSeminar}, and further developed in \cite{Bhoja:2023dgk,BryantParallel,CharltonThesis,HarveyBook,MeinrenkenBook}. In particular, it was already noted by Cartan that a pure spinor could be used to construct a decomposable exterior form of middle degree \cite{Cartan1938}. The main goal at that time was to understand spinors globally, a problem considered formidably difficult during the first half of the 20th century. Our primary motivation for delving into the relation between spinors and exterior forms is to develop a general geometric framework for studying parallel complex spinors in terms of an associated algebraic-differential equivalent system for their square, in the spirit of the formalism presented in \cite{CLS21,Sha24} for parallel irreducible real spinors of real type. In this article, we consider the \emph{algebraic side} of this problem, which we believe to be of interest by itself, and in a separate publication, we will consider the \emph{differential} side of the problem together with several of its applications in the study of complex parallel spinors. It should be remarked that the purpose of this formalism is not to obtain a framework for studying the stabilizers of irreducible complex spinors, although it could be useful also in this regard. \emph{Au contraire}, the goal of this formalism is to provide a general algebraic framework to work with spinors that does not require knowing either the orbit stratification of the spinor space or the stabilizer of a given spinor, and can therefore be applied generally. This is especially convenient for our ultimate goal of studying parallel spinors and their moduli in various dimensions and signatures.


\subsection*{Outline and main results}


Our initial observation is that there exist \emph{two} natural squares associated to a given irreducible complex spinor, which are defined as follows:
\begin{itemize}
\item A \emph{Hermitian square} obtained via a choice of an admissible Hermitian pairing $\scS$ on $(\Sigma,\gamma)$.
\item A \emph{complex-bilinear square} obtained via a choice of an admissible complex-bilinear pairing $\scB$ on $(\Sigma,\gamma)$.
\end{itemize}

\noindent
In the following, we will denote by $s\in \mathbb{Z}_2$ the adjoint type of the given admissible pairing, whereas we will denote by $\sigma\in \mathbb{Z}_2$ its symmetry type. We will generically refer to these squares as \emph{spinorial exterior forms}, or \emph{spinorial forms} for short. The inverse of the Hermitian square of an irreducible complex spinor recovers the original spinor modulo a multiplicative unitary complex number. On the other hand, the inverse of the complex-bilinear square of an irreducible complex spinor recovers the original spinor modulo a sign. Hence, the complex-bilinear square of a spinor contains essentially all information about the original spinor. In contrast, the Hermitian square loses some information, but in exchange, it is occasionally more convenient in applications, such as in the study of algebraic constraint equations for spinors. Our main results, which we summarize below, concern the algebraic characterization of both the Hermitian and the complex-bilinear squares of a complex irreducible spinor, together with the algebraic characterization of their mutual compatibility.\medskip

Section \ref{sec:vectorsasendos} is devoted to proving several algebraic results on the characterization of certain classes of rank one endomorphisms of a complex vector space, which are naturally associated with a Hermitian or complex-bilinear pairing. These results will be key to characterizing the squares of an irreducible complex spinor in the subsequent sections.\medskip

In Section \ref{section:even_forms} we begin our study of irreducible complex spinors, considering first the even-dimensional case. We first identify the complex Clifford algebra $\C\mathrm{l}(V^*,h^*)$ with the \emph{Kähler-Atiyah algebra} $(\wedge V^*_\C,\diamond)$, see Appendix \ref{app:KA} for details. Then, we fix an irreducible Clifford module $(\Sigma,\gamma)$ equipped with an admissible complex-bilinear pairing $\scB$ and a compatible admissible Hermitian pairing $\scS$, associated to a quadratic vector space $(V,h)$ of arbitrary signature and even dimension. This allows us to obtain the image of the spinor square maps inside $\wedge V^*_\C$. A streamlined version of our first theorem, namely Theorem \ref{thm:even_Hermitian_forms}, reads:

\begin{thm}
A complex exterior form $\alpha\in\wedge V^*_\C$ is the Hermitian square of an irreducible complex spinor in even dimension $d$ if and only if it satisfies the following algebraic system:   
\begin{equation*}
\alpha\diamond\alpha = 2^{\frac{d}{2}}\alpha^{(0)}\alpha,\qquad(\pi^{\frac{1-s}{2}}\circ\tau)(\bar\kappa\alpha) = \kappa \bar{\alpha}\, , \qquad \alpha \diamond \beta \diamond \alpha=2^{\frac{d}{2}}(\alpha\diamond\beta)^{(0)}\alpha
\end{equation*}

\noindent
for an exterior form $\beta\in\wedge V^*_\C$ satisfying $(\alpha\diamond\beta)^{(0)} \neq 0$ and a unit complex number $\kappa\in \U(1)$.
\end{thm}

\noindent
This theorem characterizes the image of the spinor square map associated with an admissible Hermitian pairing in terms of a system of algebraic equations together with the fundamental inequality $(\alpha\diamond\beta)^{(0)} \neq 0$, where $(\alpha\diamond\beta)^{(0)}\in\C$ denotes the degree zero component of the exterior form $\alpha\diamond\beta$. Furthermore, we show that if the underlying quadratic vector space is Euclidean\footnote{Namely, equipped with a positive-definite metric.}, then it is enough to take $\beta = 1$, and hence in this case the characterization simplifies notably. The analogous result for the complex-bilinear square of an irreducible complex spinor in even dimensions is given in Theorem \ref{thm:bilinearsquare}, whose streamlined version is as follows:

\begin{thm} 
A complex exterior form $\alpha\in\wedge V^*_\C$ is the complex-bilinear square of an irreducible complex spinor in even dimension $d$ if and only if it satisfies the following algebraic system:   
\begin{eqnarray*}
\alpha\diamond\alpha = 2^{\frac{d}{2}} \alpha^{(0)}\alpha\,, \qquad (\pi^{\frac{1-s}{2}}\circ\tau) (\alpha) = \sigma \alpha\, , \qquad \alpha \diamond \beta \diamond \alpha = 2^{\frac{d}{2}}(\alpha\diamond\beta)^{(0)}\alpha
\end{eqnarray*}

\noindent
for an exterior form $\beta\in\wedge V^*_\C$ satisfying $(\alpha\diamond\beta)^{(0)} \neq 0$.
\end{thm}

\noindent
These theorems utilize crucially the geometric product $\diamond$ defined by the underlying metric to express the algebraic equations neatly. Their simple appearance may be deceiving: as we will see below, the way these equations capture the geometric information contained in the corresponding spinor when expanded in terms of the wedge and inner products can be remarkably subtle, and becomes exponentially complicated as the dimension increases.\medskip

Section \ref{section:odd_forms} deals with the odd-dimensional case. To obtain the analogous results in odd dimensions, and following \cite{LBC13,LB13,LBC16}, we need to consider a modification of the geometric product $\diamond$ that is defined on a \emph{truncation} of $\wedge V^{\ast}_\C$ given by: 
\begin{equation*}
\wedge^<V^*_\C:=\bigoplus_{k=0}^{\frac{d-1}{2}}\wedge^kV^*_\C    
\end{equation*}

\noindent
and equipped with the product $\vee$ defined in \eqref{eq:vee_product}. The streamlined version of our next theorem, namely Theorem \ref{thm_Hermitian_square_odd}, is the following:

\begin{thm} 
A complex exterior form $\alpha\in\wedge^< V^*_\C$ is the Hermitian square of an irreducible complex spinor in odd dimension $d$ if and only if it satisfies the following algebraic system:    
\begin{eqnarray*}
\alpha\vee\alpha = 2^{\frac{d-1}{2}}\alpha^{(0)}\alpha, \qquad(\pi^{\frac{1-s}{2}}\circ\tau)(\bar\kappa\alpha) = \kappa \bar{\alpha}\, , \qquad\alpha\vee\beta\vee\alpha=2^{\frac{d-1}{2}}(\alpha\vee\beta)^{(0)}\alpha
\end{eqnarray*}

\noindent
for an exterior form $\beta \in \wedge^< V^*_\C$ satisfying $(\alpha\vee\beta)^{(0)}\neq0$ and a unit complex number $\kappa \in \U(1)$.  
\end{thm}

\noindent
For the complex-bilinear square of an irreducible complex spinor in odd dimensions, we obtain the analogous characterization in Theorem \ref{thm:odd_complex-bilinear_square}, whose streamlined version is as follows:

\begin{thm}
A complex exterior form $\alpha\in\wedge^< V^*_\C$ is the complex-bilinear square of an irreducible complex spinor in odd dimension $d$ if and only if it satisfies the following algebraic system:    
\begin{eqnarray*}
\alpha\vee\alpha=2^{\frac{d-1}{2}}\alpha^{(0)}\alpha,\qquad(\pi^{\frac{1-s}{2}}\circ\tau)(\alpha)=\sigma\alpha,\qquad\alpha\vee\beta\vee\alpha=2^{\frac{d-1}{2}}(\alpha\vee\beta)^{(0)}\alpha   
\end{eqnarray*}
for an exterior form $\beta\in\wedge^<V^*_\C$ satisfying $(\alpha\vee\beta)^{(0)}\neq 0$.
\end{thm}

\noindent
These four theorems contain the complete algebraic characterization of the square of an irreducible complex spinor in every dimension and signature. We supplement these theorems with several results, which are useful in practical examples and applications. More precisely:
\begin{itemize}
\item In Propositions \ref{prop:compatibilitysquares} and \ref{prop:compatibilitysquares_odd}, we obtain if and only if algebraic conditions for a Hermitian and a complex-bilinear square to be the square of the \emph{same} irreducible complex spinor. This is particularly relevant in applications, since in practical situations one is typically interested in computing the Hermitian and complex-bilinear squares of the same spinor in order to use them complementarily.

\item In Propositions \ref{prop:Hermitiansquareconjugate}, \ref{prop:Bilinearsquareconjugate}, \ref{prop:Hermitiansquareconjugate_odd}, and \ref{prop:Bilinearsquareconjugate_odd}, we obtain the square of a conjugate spinor from both the Hermitian and complex-bilinear squares of the original given spinor. The notion of conjugate spinor is naturally determined by a compatible pair $\scB$ and $\scS$ through the unique complex anti-linear map that relates them. 

\item In Lemmas \ref{lemma:constrainedspinoreven} and \ref{lemma:constrainedspinorodd}, we obtain if and only if algebraic conditions for a spinor to lie in the kernel of a given endomorphism in terms of any of its squares. This result is particularly beneficial to study natural gauge theoretic conditions that are defined in terms of constraints on spinors, as illustrated in the next subsection with the notion of \emph{spinorial instanton}. 
\end{itemize}

\noindent
In Section \ref{sec:lowdimensions} we explore the general framework of spinorial exterior forms associated to complex irreducible spinors through several applications and examples. First, we use the theorems above to explicitly construct the Hermitian and complex-bilinear squares of an irreducible complex spinor $\eta$, not necessarily \emph{normalized}, in Euclidean dimensions ranging from two to six. Our goal with these computations is to illustrate how efficiently we can systematically construct these squares and highlight the importance of distinguishing between the Hermitian and complex-bilinear squares and their compatibility conditions. In particular, in these dimensions, we recover and refine the result obtained by Wang in \cite{Wang89} regarding the Hermitian and complex-bilinear squares of $\eta$, providing an alternative proof that does not rely on group-theoretic arguments. Wang proved that the Hermitian square of $\eta$ is a sum of powers of a Kähler form, and that the complex-bilinear square of $\eta$ is a complex volume plus certain multiples of a Kähler form, which we show to vanish in the cases we consider. It should be remarked that the algebraic characterization we obtain allows us to compute the squares of an irreducible complex spinor not necessarily normalized. This seemingly minor point becomes very relevant when applying this theory to the study of parallel spinors and other spinorial equations, whose solutions are typically not preserved by conformal transformations of the spinor. We also note that the construction of the complex-bilinear square that we present in Section \ref{sec:lowdimensions}, which relies on extracting the Plücker relations from the algebraic relations of Theorem \ref{thm:bilinearsquare}, can be generalized to arbitrary dimensions, as we will discuss in a separate article \cite{PinoGilSha}.\medskip

\noindent
In Section \ref{sec:lowdimensions}, we apply the algebraic theory of spinorial exterior forms to study a natural notion of a \emph{special connection} determined in terms of a nowhere vanishing complex spinor. This notion of special connection, to which we generically refer as a possibly higher \emph{spinorial instanton}, applies to connections defined on principal bundles as well as their categorifications, such as curvings on bundle gerbes and other notions of connections on categorified principal bundles. Roughly speaking, a \emph{spinorial connection} is a pair $(\eta,\Phi)$ consisting of a nowhere vanishing irreducible complex spinor $\eta$ and a connection $\Phi$ on a given, possibly categorified, principal bundle, satisfying:
\begin{equation*}
\cF_{\Phi}\cdot \eta = 0,
\end{equation*}

\noindent
where $\cF_{\Phi}$ denotes the appropriate notion of curvature associated to $\Phi$, and \emph{dot} denotes Clifford multiplication. For principal bundles, the definition of a spinorial instanton is as follows. Let $P$ be a principal $G$-bundle on $(M,g)$. A \emph{spinorial instanton} on $(P,S,M,g)$ is a pair $(\eta,A)$ consisting of a nowhere vanishing spinor $\eta\in \Gamma(S)$ on $(M,g)$ and a connection $A$ on $P$ satisfying:
\begin{equation*}
F_A \cdot \eta = 0,
\end{equation*}

\noindent
where $F_A \in \Omega^2(M,\frg_P)$ is the curvature of $A$, $\frg_P$ is the adjoint bundle of $P$, and \emph{dot} denotes Clifford multiplication of the two-form part of $F_A$. The following table summarizes these conditions in low Euclidean dimensions\footnote{The symbol $\mu\in\Z_2$ denotes the chirality of the spinor $\eta$. The symbol $\ell\in\Z_2$ denotes the eigenvalue of the complex volume form in odd dimensions. The symbol $\triangle_k$ denotes the generalized product of $(V,h)$, see Appendix \ref{app:KA}.}.

\begin{table}[!ht]
\centering
\caption{Spinorial instanton condition in Euclidean dimensions two to six}
\label{tab:F_conditions_simple}
\renewcommand{\arraystretch}{2.5} 
\begin{tabular}{|c|c|}
\hline
\emph{Dimension} & \emph{Spinorial instanton condition} \\
\hline
$d=2$ & 
$F_A = 0$
\\
\hline
$d=3$ & 
$\sqrt{\langle\vartheta,\vartheta\rangle}\ast F_A = i\ell\,\iota_{\vartheta^\sharp}F_A$
\\
\hline
$d=4$ & 
$F_A\wedge\omega=0,\qquad    F_A+\mu \ast F_A = i \sqrt{\frac{2}{ \escal{\omega,\omega}}} F_A\triangle_1\omega$
\\
\hline
$d=5$ & 
$\sqrt{\frac{\escal{\omega,\omega}}{2}} F_A + \ell*(F_A\wedge\theta)-iF_A\triangle_1\omega=0,$\\
& $F_A\triangle_1\theta+i\ell*(F_A\wedge\omega)=0,\qquad\escal{F_A,\omega}=0$
\\
\hline
$d=6$ & 
$\mu\sqrt{\frac{\escal{\omega,\omega}}{3}} \ast F_A + F_A\wedge\omega = i\mu*(F_A\triangle_1\omega) ,\qquad \escal{F_A,\omega}=0$
\\
\hline
\end{tabular}
\end{table}

For \emph{real} principal bundles defined on a low-dimensional Riemannian manifold, the notion of spinorial connections reduces to that of a standard instanton for the geometry determined by an irreducible unit complex spinor \cite{Carrion}. However, the notion of spinorial instanton applies generally in every dimension and signature, and for every type of connection on a possibly categorified principal bundle. Even more, when $\eta$ is \emph{isotropic}, the notion of spinorial instanton introduced above can differ from the notion of instanton given in \cite{Carrion} on a principal bundle. Note that we are considering the Lie algebra $\frg$ of $G$ to be a complex vector space, whose complex structure \emph{interacts} with the complex character of the spinor $\eta$ itself. This is why we obtain more complicated conditions than expected: if we assume that $\frg$ is \emph{real}, then the conditions in the previous table reduce to the standard instanton conditions on the given dimension. Conveniently enough, the framework of spinorial exterior forms that we use here can be efficiently applied to the more general complex case without resorting to the underlying representation theory.\medskip

\noindent
We also consider spinorial instantons on an abelian bundle gerbe. Let $(M,g)$ be a pseudo-Riemannian manifold equipped with a bundle of irreducible and chiral complex Clifford modules $S$ over the bundle of Clifford algebras of $(M,g)$. Let $(\cP, Y, \cA)$ be a $\mathbb{C}^{\ast}$-bundle gerbe with connective structure, where $\cP\to Y\times_M Y$ is a principal $\C^*$-bundle defined on the fibered product of the smooth submersion $Y\to M$ with itself, and $\cA$ is a connection on $\cP$ satisfying adequate compatibility conditions \cite{Brylinski1993,Bunk2021,Giraud1971,Murray1996}. A \emph{spinorial curving} on $(\cP, Y, \cA,M,g)$ is a pair $(\eta,b)$ consisting of $\eta\in \Gamma(S)$ and $b$ satisfying:
\begin{equation*}
H_b\cdot \eta = 0.
\end{equation*}

\noindent
The following table summarizes our characterization of spinorial curvings in low Euclidean dimensions. 
\begin{table}[!ht]
\centering
\caption{Spinorial curving condition in Euclidean dimensions three to six.}
\label{tab:H_conditions_simple}
\renewcommand{\arraystretch}{2.5} 
\begin{tabular}{|c|c|}
\hline
\emph{Dimension} & \emph{Spinorial curving condition} \\
\hline
$d=3$ & 
$H_b = 0$
\\
\hline
$d=4$ & 
$H_b+i \sqrt{\frac{2}{\escal{\omega,\omega}}}  H_b\triangle_1\omega=0,\qquad \ast H_b = i  \mu \sqrt{\frac{2}{ \escal{\omega,\omega}}} H_b\triangle_2\omega$
\\
\hline
$d=5$ & 
$\sqrt{\frac{\escal{\omega,\omega}}{2}} \ast H_b + \ell H_b\triangle_1\theta-i(*H_b)\triangle_1\omega=0,$\\
&$*(H_b\wedge\theta)+i\ell H_b\triangle_2\omega=0,\quad \escal{*H_b,\omega}=0$
\\
\hline
$d=6$ & 
$(H_b+i\mu \ast H_b )\wedge\omega=0,$\\
&$\mu \sqrt{\frac{\escal{\omega,\omega}}{3}} \ast H_b + H_b\triangle_1\omega+i  ( \mu \ast( H_b\triangle_1\omega) - \sqrt{\frac{\escal{\omega,\omega}}{3}} H_b)=0$
\\
\hline
\end{tabular}
\end{table}

\noindent
In Section \ref{section:8dRiemannian} we consider the squares of an irreducible complex and chiral spinor in eight Euclidean dimensions. This case is particularly interesting because eight is the smallest dimension for which \emph{impure}, also called \emph{non-simple}, irreducible complex chiral spinors exist. We compute both the Hermitian and complex-bilinear squares of the most general impure complex and irreducible chiral spinor, not necessarily normalized, complementing and extending some of the results of \cite{Krasnov2024} in the case of eight Euclidean dimensions. The Hermitian square is of the form:
\begin{equation*}
\widehat{\alpha} = \sqrt{\frac{2\escal{\omega,\omega}+\escal{\Theta,\Theta}}{14}} + i\omega + \Theta-i\mu*\omega +  \sqrt{\frac{2\escal{\omega,\omega}+\escal{\Theta,\Theta}}{14 \escal{\Theta,\Theta}^2 }}\Theta \wedge \Theta
\end{equation*}

\noindent
for a certain two-form $\omega$ and four-form $\Theta$ satisfying the conditions given in Corollary \ref{cor:Hermitian8dimpure}. On the other hand, the complex-bilinear square is of the form:
\begin{equation*}
\alpha =  \frac{\langle \Omega , \Omega \rangle^{\frac{1}{2}}}{\sqrt{14}} + \Omega + \frac{\Omega\wedge\Omega}{\sqrt{14} \langle \Omega , \Omega \rangle^{\frac{1}{2}}}
\end{equation*}

\noindent
for a certain complex four-form $\Omega$ satisfying the following algebraic conditions:
\begin{equation*}
\ast \Omega = \mu \Omega\, , \qquad \sqrt{14}\Omega\triangle_2\Omega + 12  \langle \Omega , \Omega \rangle^{\frac{1}{2}} \Omega = 0,
\end{equation*} 
 
\noindent
where $\mu\in\mathbb{Z}_2$ is the chirality of the corresponding spinor. When such a spinor is a complex multiple of a real and irreducible spinor, we recover the intrinsic algebraic conditions obtained in \cite{LS24_Spin7} to characterize Cayley forms. We apply the construction of the Hermitian square of an irreducible and chiral complex spinor to obtain the algebraic characterization of both a spinorial instanton on a principal bundle and a spinorial curving on an abelian bundle gerbe, similarly to how we proceeded in Section \ref{sec:lowdimensions} in dimensions ranging from two to six. Equations \eqref{eq:algebraicF8d} give the following result.

\begin{cor}
Let $P$ be a principal bundle with complex structure group $G$ defined on an eight-dimensional Riemannian manifold. A pair $(\eta,A)$ is a spinorial instanton on $(P,S,M,g)$ if and only if:
\begin{gather*}
\langle F_A , \omega \rangle = 0,\qquad F_A\wedge \omega  + i F_A\triangle_1 \Theta  +  \mu \ast (F_A\wedge \omega)   = 0,\\
14^{-\frac{1}{2}}\sqrt{ 2\escal{\omega,\omega}+\escal{\Theta,\Theta}} \,\ast F_A = \mu F_A\wedge \Theta + i \ast (F_A\triangle_1\omega),
\end{gather*}

\noindent
where:
\begin{equation*}
\widehat{\alpha} = \sqrt{\frac{2\escal{\omega,\omega}+\escal{\Theta,\Theta}}{14}} + i\omega + \Theta-i\mu*\omega +  \sqrt{\frac{2\escal{\omega,\omega}+\escal{\Theta,\Theta}}{14 \escal{\Theta,\Theta}^2 }}\Theta \wedge \Theta,\quad \omega\in\Omega^2(M),\,\Theta\in\Omega^4(M),
\end{equation*}

\noindent
is the Hermitian square of $\eta\in\Gamma(S)$.
\end{cor}

\noindent
These equations hold for a general irreducible and chiral complex spinor, which interacts with the complex structure of the Lie algebra of $G$. If we assume that $\eta$ is a complex multiple of a real and irreducible spinor, and hence stabilized by $\Spin(7)$, then $\omega = 0$ and after normalizing $\eta$ these equations reduce to:
\begin{eqnarray*}
\ast F_A = \mu F_A\wedge \Theta   \, , \qquad F_A\triangle_1 \Theta = 0.
\end{eqnarray*}

\noindent
The first equation recovers the well-known condition for a connection on a principal bundle to be a $\Spin(7)$-instanton on a \emph{topological} $\Spin(7)$ manifold \cite{Carrion}, as expected. Interestingly enough, we also obtain an extra condition, namely  $F_A\triangle_1 \Theta = 0$, which is not considered in the literature on $\Spin(7)$-instantons. However, a finer inspection of the algebraic identities satisfied by a $\Spin(7)$-structure reveals that $F_A\triangle_1 \Theta = 0$ is an identity and therefore can be omitted. This identity is certainly non-trivial, which illustrates the refined and deep manners in which the algebraic conditions contained in the previous general theorems can manifest in explicit examples.  \medskip

We also consider the algebraic characterization of a \emph{spinorial curving} on a $\C^{\ast}$-bundle gerbe defined on a Riemannian eight-manifold $(M,g)$. Equations \eqref{eq:algebraicH8d} imply the following result.

\begin{cor}
A pair $(\eta,b)$ is a spinorial curving on $(\cP, Y, \cA,M,g)$ if and only if:
\begin{equation*}
i H_b \triangle_2 \omega + H_b \triangle_3 \Theta = 0\, , \qquad 14^{-\frac{1}{2}}\sqrt{ 2\escal{\omega,\omega}+\escal{\Theta,\Theta}} H_b = H_b\triangle_2 \Theta  -  i H_b \triangle_1 \omega    -  i \mu \ast (H_b\wedge \omega),
\end{equation*}

\noindent
where:
\begin{equation*}
\widehat{\alpha} = \sqrt{\frac{2\escal{\omega,\omega}+\escal{\Theta,\Theta}}{14}} + i\omega + \Theta-i\mu*\omega +  \sqrt{\frac{2\escal{\omega,\omega}+\escal{\Theta,\Theta}}{14 \escal{\Theta,\Theta}^2 }}\Theta \wedge \Theta,\quad \omega\in\Omega^2(M),\,\Theta\in\Omega^4(M),
\end{equation*}

\noindent
is the Hermitian square of $\eta\in\Gamma(S)$.
\end{cor}

\noindent
We believe that these equations define natural conditions in abelian higher gauge theory and together with their non-abelian generalization may lead to interesting geometric problems.\medskip

In Section \ref{section:SDGerbe} we apply the theory of spinorial exterior forms to first construct both the Hermitian and complex-bilinear squares of an irreducible and chiral complex spinor in six Lorentzian dimensions and then use the result to characterize the spinorial curvings of a $\mathbb{C}^{\ast}$-bundle gerbe defined on a Lorentzian six-manifold. Due to the deep geometric relation between the self-duality condition for a three-form in six Lorentzian dimensions and the spinorial curving condition, this is a particularly interesting case to consider. The main result of this section is as follows.

\begin{thm}
A pair $(\eta,b)$ is a spinorial curving on $(\cC,M,g)$ if and only if there exists an isotropic one-form $v\in \Omega^1(M)$ conjugate to the Dirac current $u$ of $\eta$ such that the following conditions are satisfied:
\begin{gather*}
H_b(u^{\sharp},v^{\sharp}) = -\mu \ast_{uv} H_b^\perp - iH_b^\perp\triangle_2^\perp\omega - i\omega(H_b(u^{\sharp},v^{\sharp})^\sharp),\\
(H_b(u^\sharp,v^\sharp)\wedge u+H_b(u^\sharp))\wedge\omega=0,\\
\mu \ast_{uv} (H_b(u^\sharp,v^\sharp)\wedge u+H_b(u^\sharp)) = H_b(u^\sharp,v^\sharp)\wedge u+H_b(u^\sharp) - i (H_b(u^\sharp,v^\sharp)\wedge u+H_b(u^\sharp)) \triangle^{\perp}_1 \omega,
\end{gather*}

\noindent
where $\ast_{uv}$ denotes the Hodge dual induced by $h$ on the orthogonal complement $V_{uv}$ to $u$ and $v$ and the Hermitian square of $\eta$ is given by $\widehat{\alpha} = u + i u\wedge \omega - \mu \ast u$. 
\end{thm}

\noindent
The reader is referred to Theorem \ref{thm:spinorialcurving} for more details. Several corollaries can be extracted from this result, of which we reproduce the following, which is particularly simple in its form.

\begin{cor}
Let $b$ be a curving whose curvature satisfies $\ast H_b = \mu H_b$ on an oriented Lorentzian six-dimensional manifold $(M,g)$. Then, the pair $(\eta,b)$ is a spinorial curving on $(\cC,M,g)$ for every chiral irreducible complex spinor $\eta$ of chirality $\mu\in \mathbb{Z}_2$ on $(M,g)$.
\end{cor}

\noindent
By the previous corollary, the notion of spinorial curving on a six-dimensional Lorentzian manifold is a natural generalization of the notion of self-duality for curvings.

\begin{ack}
The work of AGG is supported by the Beijing Institute of Mathematical Sciences and Applications (BIMSA). The work of CSS was partially supported by the Leonardo grant LEO22-2-2155 of the BBVA Foundation and the research grant PID2023-152822NB-I00 of the Ministry of Science of the government of Spain. We would like to thank I.\ Chrysikos, D.\ Conti, and K.\ Krasnov for interesting suggestions and comments. 
\end{ack}

\begin{center}
\begin{tikzpicture}
    \node[
        draw,                  
        minimum width=12cm,     
        minimum height=2cm,    
        align=center,          
        line width=1pt,        
        fill=white!10           
    ] 
    at (0,0) 
    {\quad \emph{Achtung}: we define Clifford algebras using the plus convention $v^2 = h(v,v)$!\quad}; 
\end{tikzpicture}
\end{center}


\section{Complex vectors as endomorphisms}
\label{sec:vectorsasendos}


In this section, we algebraically characterize a class of rank one endomorphisms that is naturally associated with the choice of either a Hermitian or a complex-bilinear pairing on a complex vector space. We will use this algebraic characterization in the remainder of the article to study the squares of an irreducible complex spinor. The results of Subsection \ref{subsec:complexbilinearvectors} can be considered as the complex analog of the results presented in \cite[Section 2]{CLS21}, where the algebraic structure of rank one real endomorphisms of a real \emph{paired} vector space is examined in detail. 


\subsection{Hermitian pairings}


Let $\Sigma$ be a complex vector space equipped with a non-degenerate \emph{Hermitian} pairing $\scS\colon\Sigma\times\Sigma\to\C$, namely a non-degenerate sesquilinear pairing satisfying: $$\scS(\xi_1 , \xi_2) = \overline{\scS(\xi_2 , \xi_1)}$$ for every $\xi_1 , \xi_2 \in \Sigma$. We follow the convention that $\scS$ is complex anti-linear in the second entry, that is:
\begin{equation*}
\scS(\xi_1,c\xi_2)=\bar{c}\scS(\xi_1,\xi_2)
\end{equation*}

\noindent
for all $c\in\C$. Note that we do not assume $\scS$ to be necessarily positive-definite. We will refer to $(\Sigma,\scS)$ as a \emph{Hermitian vector space}. Let $(\End(\Sigma) , \circ)$ be the unital associative complex algebra of complex linear endomorphisms of $\Sigma$, where $\circ$ denotes the standard composition of linear maps. Given $E\in\End(\Sigma)$, we denote by $E^\dagger$ the adjoint of $E$ with respect to $\scS$, which is uniquely determined by the condition: 
\begin{equation*}
\scS(E^\dagger\xi_1,\xi_2)=\scS(\xi_1,E\xi_2)    
\end{equation*}

\noindent
for every $\xi_1 , \xi_2 \in \Sigma$.

\begin{definition}
The \emph{Hermitian square maps} of a Hermitian vector space $(\Sigma,\scS)$ are the following quadratic maps:
\begin{equation*}
\hcE_\kappa\colon\Sigma\to\End(\Sigma)\, , \qquad \xi\mapsto \hcE_\kappa(\xi):=\kappa\,\xi\otimes\xi^{\ast}\, , \qquad \kappa\in \U(1),
\end{equation*}

\noindent
parametrized by $\U(1)$, where $\xi^* := \scS(-,\xi)\in\Sigma^{\ast} = \Hom(\Sigma,\C)$.
\end{definition}

\noindent
Our main goal in this subsection is to obtain the algebraic characterization of the endomorphisms in the image $\Im(\hcE_\kappa)\subset \End(\Sigma)$ of the Hermitian square maps. We will extensively use such characterization in later sections to describe the \emph{Hermitian square} of an irreducible complex spinor, see Definition \ref{def:squares}. The Hermitian square maps satisfy the following general identities, which are crucial for our purposes.

\begin{lemma}
\label{lemma:properties_cE}
Let $(\Sigma,\scS)$ be a Hermitian vector space. Then:
\begin{equation*}
\hcE_\kappa(\xi)\circ A \circ \hcE_\kappa(\xi) =\Tr\big(\hcE_\kappa(\xi)\circ A \big)\hcE_\kappa(\xi)\, , \qquad (\bar\kappa \hcE_\kappa(\xi))^\dagger = \bar\kappa \hcE_\kappa(\xi)
\end{equation*}

\noindent
for every $A\in \End(\Sigma)$ and $\kappa \in \U(1)$.
\end{lemma}

\begin{proof}
Let $\chi\in\Sigma$. Then: 
\begin{align*}
\hcE_\kappa(\xi)\circ A \circ \hcE_\kappa(\xi) \chi &= \hcE_\kappa(\xi)\big(\kappa\scS(\chi,\xi) A(\xi) \big) = \kappa^2 \scS(\chi,\xi)\scS(A(\xi),\xi)\xi\\
&=\Tr\big(\hcE_\kappa(\xi) \circ A\big)\kappa\scS(\chi,\xi)\xi=\Tr\big(\hcE_\kappa(\xi) \circ A \big)\hcE_\kappa(\xi)\chi,
\end{align*}

\noindent
where we have used that $\Tr\big(\hcE_\kappa(\xi) \circ A\big)=\kappa\scS(A(\xi),\xi)$. This gives the first equation in the statement of the lemma. For the second equation, let $\eta_1 , \eta_2\in\Sigma$ and set $E=\hcE_\kappa(\xi)$ for some $\kappa\in\U(1)$ and $\xi\in\Sigma$. We compute: 
\begin{align*}
\scS(E^\dagger \eta_1,\eta_2) & = \scS(\eta_1,E\eta_2) = \scS(\eta_1,\kappa\scS(\eta_2,\xi)\xi) = \bar{\kappa} \overline{\scS(\eta_2,\xi)} \scS(\eta_1,\xi) = \bar{\kappa} \scS(\xi,\eta_2)\scS(\eta_1,\xi)\\
&=\scS(\bar{\kappa} \scS(\eta_1,\xi)\xi,\eta_2)=\scS(\tfrac{\bar\kappa}{\kappa} \kappa \scS(\eta_1,\xi)\xi,\eta_2) = \scS(\tfrac{\bar\kappa}{\kappa}E\eta_1,\eta_2).
\end{align*}

\noindent
Then $E^\dagger=\frac{\bar\kappa}{\kappa}E$. This can be equivalently written as $(\bar\kappa \hcE_\kappa(\xi))^\dagger = \bar\kappa \hcE_\kappa(\xi)$ and hence we conclude.
\end{proof}

\begin{lemma}
\label{lemma:Z=calE(Sigma)}
Let $\kappa_1 , \kappa_2 \in \U(1)$. We have $\hcE_{\kappa_1}(\Sigma)\cap \hcE_{\kappa_2}(\Sigma)=\{0\}$ if and only if $\kappa_1\neq\kappa_2$. 
\end{lemma}

\begin{proof}
If $\hcE_{\kappa_1}(\Sigma)\cap \hcE_{\kappa_2}(\Sigma)=\{0\}$ then clearly $\kappa_1\neq\kappa_2$ since otherwise $\hcE_{\kappa_1}(\Sigma) = \hcE_{\kappa_2}(\Sigma) \neq \{0\}$. For the converse, assume that there exists a non-zero element in $\hcE_{\kappa_1}(\Sigma)\cap \hcE_{\kappa_2}(\Sigma)$ with $\kappa_1\neq\kappa_2$. Then, there exist non-zero elements $\xi_1 , \xi_2 \in \Sigma$ satisfying: 
\begin{equation*}
\kappa_1\xi_1\otimes\xi_1^* = \kappa_2\xi_2\otimes\xi_2^*.
\end{equation*}

\noindent
Choose an element $\chi\in\Sigma$ such that $\xi^{\ast}_1(\chi)\neq 0$ and $\xi^{\ast}_2(\chi)\neq 0$, which exists since otherwise $\xi_1 = \xi_2 = 0$. Evaluating the previous equation on such $\chi$ we obtain:
\begin{equation*}
 \xi_1  = \frac{\kappa_2}{\kappa_1}\frac{\xi_2^{\ast}(\chi)}{\xi_1^{\ast}(\chi)}\xi_2,
\end{equation*}

\noindent
which plugged back into the previous equation gives:
\begin{equation*}
\left(\kappa_2 - \kappa_1 \left\vert\frac{\xi_2^{\ast}(\chi)}{\xi_1^{\ast}(\chi)}\right\vert^2\right)\xi_2\otimes\xi_2^* = 0,
\end{equation*}

\noindent
which immediately implies $\kappa_1 = \kappa_2$ and $\vert \xi_1^{\ast}(\chi) \vert^2 = \vert \xi_2^{\ast}(\chi)\vert^2$ since $\kappa_1,\kappa_2\in\mathrm{U}(1)$. This in turn implies $\xi_1 = \kappa^{\prime} \xi_2$ for a phase $\kappa^{\prime}\in \U(1)$ and hence we conclude.
\end{proof}

\noindent
Elaborating on the previous lemmas, we obtain the desired characterization of those endomorphisms that belong to $\Im(\hcE_\kappa)\subset\End(\Sigma)$ for a given $\kappa \in \U(1)$.
\begin{prop}
\label{prop:algebraichHermitian}
Let $(\Sigma,\scS)$ be a Hermitian vector space and let $E\in \End(\Sigma)$ be a non-zero endomorphism. The following statements are equivalent:   

\begin{enumerate}
\item[\normalfont(a)] There exists an element $\xi\in \Sigma$ and a phase $\kappa \in \U(1)$ such that $E = \hcE_{\kappa}(\xi) = \kappa\, \xi \otimes \xi^{\ast}$. That is, $E\in \Im (\hcE_{\kappa})$ for a uniquely determined $\kappa \in \U(1)$.
\item[\normalfont(b)]  There exists a phase $\kappa\in \U(1)$ such that the following equations are satisfied:
\begin{equation*}
E \circ A \circ E=\Tr(E\circ A)E\, , \qquad (\bar\kappa E)^\dagger = \bar\kappa E
\end{equation*}

\noindent
for every endomorphism $A\in \End(\Sigma)$.

\item[\normalfont(c)]  There exists an endomorphism $A\in\End(\Sigma)$ such that $\Tr(E\circ A)\neq 0$ and the following equations are satisfied:
\begin{equation*}
E\circ E=\Tr(E)E\, , \qquad E \circ A \circ E=\Tr(E\circ A)E\, , \qquad (\bar\kappa E)^\dagger = \bar\kappa E
\end{equation*}

\noindent
for a phase $\kappa\in \U(1)$.
\end{enumerate}
\end{prop}

\begin{proof}
The implication $(a) \Rightarrow (b)$ follows from Lemma \ref{lemma:properties_cE}. The implication $(b) \Rightarrow (c)$ follows from the fact that the trace defines a complex non-degenerate bilinear pairing on $\End(\Sigma)$. To get $E^2=\Tr(E)E$ we just set $A=\Id$. We prove then that $(c) \Rightarrow (a)$. Hence, we assume that $A\in\End(\Sigma)$ satisfies $E\circ A\circ E=\Tr(E\circ A)E$ with $\Tr(E\circ A)\neq0$. Assume first that $\Tr(E) = 0$. Define:
\begin{equation*}
A_\varepsilon:=\Id+\frac{\varepsilon}{\Tr(E\circ A)}A,
\end{equation*}

\noindent
where $\varepsilon\in\R_{>0}$. For $\varepsilon>0$ small enough, the endomorphism $A_\varepsilon$ is invertible. Define $E_\varepsilon:=E\circ A_\varepsilon$. Then $\Tr(E_\varepsilon)=\varepsilon$ and $E_\varepsilon^2=\varepsilon E_\varepsilon$, where we have used that $E^2=\Tr(E)E=0$. Hence $P:=\frac{1}{\varepsilon}E_\varepsilon$ satisfies $P^2=P$ and $\Tr(P)=1$. Consequently, $\rk(E_\varepsilon)=\rk(P)=\Tr(P)=1$. Since $A_\varepsilon$ is invertible, this implies $\rk(E)=1$. Since $E$ is of rank one, there exists a non-zero vector $\xi\in\Sigma$ and a non-zero linear functional $\beta\in\Sigma^*$ such that $E=\xi\otimes\beta$. Furthermore, since the Hermitian pairing $\scS$ is non-degenerate, there exists a unique non-zero $\xi_0\in\Sigma$ such that $\beta = \scS(-,\xi_0) = \xi_0^*$. Condition $(\bar\kappa E)^\dagger = \bar\kappa E$ is now equivalent to:
\begin{equation*}
\bar\kappa\scS(-,\xi_0)\xi=\kappa\scS(-,\xi)\xi_0.
\end{equation*}

\noindent
Since $\scS$ is non-degenerate, there exists $\chi \in\Sigma$ such that $\scS(\chi,\xi)=1$. Then, evaluating the equation above on $\chi$ we obtain: 
\begin{equation*}
\xi_0=\frac{\bar\kappa}{\kappa}\scS(\chi,\xi_0)\xi\quad\text{and}\quad E=\frac{\kappa}{\bar\kappa}\overline{\scS(\chi,\xi_0)}\xi\otimes\xi^*.
\end{equation*}

\noindent
Let $c:=\frac{\kappa}{\bar\kappa}\overline{\scS(\chi,\xi_0)}\in\C\setminus\{0\}$ and $\hat\xi:=\abs{c}^{1/2}\xi$. Then: 
\begin{equation*}
E=c\, \xi\otimes\xi^{\ast}=\frac{c}{\abs{c}}\hat\xi\otimes\hat\xi^*.
\end{equation*}

\noindent
Therefore $E=\hcE_{c\abs{c}^{-1}}(\hat\xi)$ and $\hat\xi\in\Sigma$. This proves $(c) \Rightarrow (a)$ in the case $\Tr(E) = 0$. If $\Tr(E) \neq 0$ we simply define $P = \Tr(E)^{-1} E$ which directly satisfies $P^2=P$ and hence we conclude.
\end{proof}

\noindent
Define:
\begin{equation*}
\hcE(\Sigma) := \bigcup_{\kappa\in\U(1)}{\hcE}_\kappa(\Sigma) \subset \End(\Sigma)
\end{equation*}

\noindent
as the disjoint union of the images of the Hermitian square maps for all $\kappa\in\U(1)$. We have:
\begin{equation*}
\hcE (\Sigma) = \{E\in\End(\Sigma) \, \mid \rk(E) \leq 1\, , \,(\bar\kappa E)^\dagger= \bar\kappa E\, , \,  \kappa \in \U(1)\}.
\end{equation*}

\noindent
Hence, $\Im(\hcE)\subset \End(\Sigma)$ is the union of a family of \emph{real} linear subspaces, parametrized by $\kappa \in \U(1)$, of the determinantal variety of rank one complex endomorphisms of $\Sigma$. If $\scS$ is a Hermitian inner product on $\Sigma$, namely if $\scS$ is positive-definite, then Proposition \ref{prop:algebraichHermitian} simplifies as follows.

\begin{cor} 
Let $(\Sigma,\scS)$ be a Hermitian vector space with $\scS$ positive-definite, and let $E\in \End(\Sigma)$ be a non-zero endomorphism. The following statements are equivalent:   

\begin{enumerate}
\item[\normalfont(a)] There exists an element $\xi\in \Sigma$ and a phase $\kappa \in \U(1)$ such that $E = \hcE_{\kappa}(\xi) = \kappa\, \xi \otimes \xi^{\ast}$. That is, $E\in \Im (\hcE_{\kappa})$ for a uniquely determined $\kappa \in \U(1)$.

\item[\normalfont(b)]  The following equations are satisfied:
\begin{equation*}
E \circ  E=\Tr(E)E\, , \qquad (\bar\kappa E)^\dagger = \bar\kappa E
\end{equation*}

\noindent
for a phase $\kappa\in \U(1)$.
\end{enumerate}
\end{cor}

\begin{proof}
By the proof of Proposition \ref{prop:algebraichHermitian}, it is enough to prove that if a non-zero endomorphism $E\in \End(\Sigma)$ satisfies (b), then its trace is non-vanishing. Since $E^2=\Tr(E)E$ it follows that:
\begin{equation*}
(\bar\kappa E)^2=\Tr(\bar\kappa E)\bar\kappa E.
\end{equation*}

\noindent
Equation $(\bar\kappa E)^\dagger = \bar\kappa E$ implies that $\bar\kappa E$ is \emph{diagonalizable} with real eigenvalues $r_i$, with $i = 1 , \hdots , \dim_\C\Sigma$. Hence, taking the trace of the previous equation, we obtain:
\begin{equation*}
\Tr(\bar\kappa E)^2=\Tr((\bar\kappa E)^2)=\sum_{i = 1}^{\dim_\C\Sigma}r_i^2.
\end{equation*}

\noindent
Since the right-hand side is a sum of positive squares, it vanishes if and only if all $r_i$ vanish, that is, if and only if $  E=0$, and thus we conclude.
\end{proof}


\subsection{Complex-bilinear pairings}
\label{subsec:complexbilinearvectors}


Let $\Sigma$ be a complex vector space equipped with a non-degenerate complex-bilinear pairing $\scB\colon\Sigma\times\Sigma\to\C$. We say that $(\Sigma,\scB)$ is a \emph{paired vector space}. We assume $\scB$ to be either symmetric or skew-symmetric and we say that $\scB$, or $(\Sigma,\scB)$, has symmetry type $\sigma\in\Z_2$ if:
\begin{equation*}
\scB(\xi_1 ,\xi_2)=\sigma\scB(\xi_2,\xi_1).
\end{equation*}

\noindent
We denote by $E^t$ the adjoint of an endomorphism $E\in\End(\Sigma)$ with respect to the complex-bilinear pairing $\scB$, which is determined by the condition: 
\begin{equation*}
\scB(E^t\xi_1,\xi_2)=\scB(\xi_1,E\xi_2)\, , \qquad \forall \,\, \xi_1 , \xi_2 \in \Sigma.
\end{equation*}

\noindent
We define the \emph{complex-bilinear square map} $\cE\colon\Sigma\to\End(\Sigma)$ of a paired vector space $(\Sigma,\scB)$ by:
\begin{equation*}
\cE(\xi):= \xi\otimes\xi^{\ast},
\end{equation*}

\noindent
where $\xi^*:=\scB(-,\xi)\in\Sigma^*$. 

\begin{remark}
In contrast with the Hermitian case considered in the previous subsection, any multiplicative factor included in the definition of $\cE$ can be reabsorbed by a \emph{complex} homothety of $\xi$ and therefore is irrelevant for our purposes.
\end{remark}

\noindent
Similarly to the Hermitian case considered in the previous subsection, we have:
\begin{align*}
    \cE(\xi) \circ A \circ \cE(\xi)\chi  &=  \cE(\xi) \big(\scB(\chi,\xi) A\xi\big) = \scB(\chi,\xi)\scB(A\xi,\xi) \xi\\
    & = \Tr\big(\cE(\xi) \circ A \big) \scB(\chi,\xi)\xi=\Tr\big(\cE(\xi) \circ A\big)\cE(\xi) \chi
\end{align*}

\noindent
for every $A\in \End(\Sigma)$ and $\xi, \chi\in\Sigma$. Furthermore:  
\begin{align*}
    \scB(\cE(\xi)^t \eta_1,\eta_2) &= \scB(\eta_1,\cE(\xi)\eta_2)=\scB(\eta_1, \scB(\eta_2,\xi)\xi)= \scB(\eta_2,\xi)\scB(\eta_1,\xi)\\
    & = \sigma\scB(\xi,\eta_2)\scB(\eta_1,\xi)=\scB(\sigma \scB(\eta_1,\xi)\xi,\eta_2)=\scB(\sigma \cE(\xi)\eta_1,\eta_2)
\end{align*}

\noindent
for every $\eta_1, \eta_2\in\Sigma$. This proves the following analog of Lemma \ref{lemma:properties_cE}.
\begin{lemma}
\label{lemma:propertiesbilinearI}
Let $(\Sigma,\scB)$ be a paired vector space of symmetry type $\sigma\in\Z_2$. Then:
\begin{equation*}
\cE (\xi)\circ A \circ \cE (\xi) =\Tr\big(\cE (\xi)\circ A \big)\cE (\xi)\, , \qquad  \cE(\xi)^t = \sigma \cE (\xi)
\end{equation*}

\noindent
for every $A\in \End(\Sigma)$.
\end{lemma}

\noindent
The proof of the following proposition, which we will use to characterize the complex-bilinear square of an irreducible complex spinor, see Definition \ref{def:squares}, is completely analogous to that of Proposition \ref{prop:algebraichHermitian} and is therefore left to the reader.
\begin{prop}
\label{prop:algebraichbilinear}
Let $(\Sigma,\scB)$ be a paired vector space of symmetry type $\sigma\in\Z_2$ and let $E\in \End(\Sigma)$ be a non-zero endomorphism. The following statements are equivalent:   

\begin{enumerate}
\item[\normalfont(a)] There exists an element $\xi\in \Sigma$ such that $E = \cE(\xi) = \xi \otimes \xi^{\ast}$. That is, $E\in \Im (\cE)$.

\item[\normalfont(b)]  The following equations are satisfied:
\begin{equation*}
E \circ A \circ E=\Tr(E\circ A)E\, , \qquad E^t=\sigma E
\end{equation*}

\noindent
for every endomorphism $A\in \End(\Sigma)$.

\item[\normalfont(c)]  The following equations are satisfied:
\begin{equation*}
E\circ E=\Tr(E)E\, , \qquad E \circ A \circ E=\Tr(E\circ A)E\, , \qquad E^t=\sigma E
\end{equation*}

\noindent
for an endomorphism $A\in\End(\Sigma)$ such that $\Tr(E\circ A)\neq 0$.
\end{enumerate}
\end{prop}
 

\subsection{The square of a conjugate element}
\label{sec:conjugatevector}


Let $\Sigma$ be a complex vector space equipped with both a Hermitian pairing $\scS$ and a non-degenerate complex-bilinear pairing $\scB$. We will say that $\scS$ and $\scB$ are \emph{compatible} if for every $\eta_1 , \eta_2 \in \Sigma$ the following relation holds:
\begin{equation*}
\scS(\eta_1 , \eta_2) = \scB(\eta_1 , K\eta_2) 
\end{equation*}

\noindent
for a unique complex anti-linear map $K\colon \Sigma \to \Sigma$ satisfying $K^2 = \epsilon\, \Id$ with either $\epsilon = +1$ or $\epsilon = -1$. This set-up will occur in later sections in our study of the square of an irreducible complex spinor. Note that if $\scS$ and $\scB$ are compatible, then it follows that:
\begin{equation*}
\scS(K\eta_1 , K\eta_2) = \epsilon \sigma  \overline{\scS(\eta_1 , \eta_2)}\, ,\qquad  \eta_1 , \eta_2 \in \Sigma,
\end{equation*}

\noindent
where $\sigma\in \mathbb{Z}_2$ denotes the symmetry type of $\scB$. 
\begin{definition}
\label{def:conjugatevector}
Let $\Sigma$ be a complex vector space equipped with compatible pairings $\scS$ and $\scB$. The \emph{conjugate} of $\xi \in \Sigma$ is $K(\xi) \in \Sigma$. 
\end{definition}

\noindent
With these provisos in mind, it is natural to wonder about the algebraic characterization of $\hcE_{\kappa}(K\xi)$ or $\cE(K\xi)$ as elements in $\End(\Sigma)$ assuming that we know the algebraic characterization of $\hcE_{\kappa}(\xi) \in \End(\Sigma)$ or $\cE(\xi) \in \End(\Sigma)$. For any endomorphism $A\in \End(\Sigma)$ and any non-zero vector $\xi\in \Sigma \setminus \{ 0 \}$ we obtain the following identities:
\begin{align}
\hcE_{\kappa}(\xi) \circ A\circ \hcE_{\kappa}(K\xi) &=   \kappa^2 \overline{\scB(A^{\dagger}(\xi),\xi)}\, \cE(\xi), \label{eq:conjugatemix1}\\
\cE(\xi) \circ A\circ \cE(K\xi) &= \bar{\kappa} \sigma  \scS( A^t(\xi) , \xi) \, \hcE_\kappa(\xi). \label{eq:conjugatemix2}
\end{align}

\noindent
This computation motivates the following result. 

\begin{prop}
\label{prop:conjugatevectorsquaresesqui}
Let $\xi \in \Sigma\setminus \{ 0\}$ and $T\in \Im(\hcE_{\kappa})\subset \End(\Sigma)$. We have $T = \hcE_{\kappa}(K \xi)$ if and only if:
\begin{equation}
\label{eq:conjugatevectorsquaresesqui}
\hcE_{\kappa}(\xi) \circ A\circ T = \kappa^2 \overline{\Tr(\cE(\xi) \circ A^{\dagger})}\, \cE(\xi)
\end{equation}

\noindent
for an endomorphism $A\in \End(\Sigma)$ satisfying $\Tr(\cE(\xi) \circ A^{\dagger}) \neq 0$.
\end{prop}

\begin{proof}
The \emph{only if} direction follows from Equation \eqref{eq:conjugatemix1} together with the non-degeneracy of the trace operator. For the converse, we set $T = \hcE_{\kappa}(\eta)$ for $\eta \in \Sigma\setminus \{ 0\}$ and plug it into Equation \eqref{eq:conjugatevectorsquaresesqui}, obtaining:
\begin{equation*}
\hcE_{\kappa}(\xi) \circ A\circ T = \kappa^2 \scS(A(\eta) , \xi) \, \scS( - , \eta) \xi =  \kappa^2 \overline{\scB(A^{\dagger}(\xi),\xi)}\, \scB( - , \xi) \xi,
\end{equation*}

\noindent
which implies:
\begin{equation*}
\overline{\scS(A^{\dagger}(\xi) , \eta)} \, \scS( - , \eta)   = \overline{\scB(A^{\dagger}(\xi),\xi)}\, \scB( - , \xi) 
\end{equation*}

\noindent
or, equivalently:
\begin{equation*}
\overline{\scS(A^{\dagger}(\xi) , \eta)} \, K(\eta)  = \overline{\scB(A^{\dagger}(\xi),\xi)}\, \xi.
\end{equation*}

\noindent
Since $\scB(A^{\dagger}(\xi),\xi) \neq 0$ by assumption, it can be seen that $\eta\in  \Sigma\setminus \{ 0\}$ is a solution of the previous equation if and only if $\eta \in  \{ \mu\, K(\xi) \mid \mu \in \U(1) \}$ and hence we conclude.
\end{proof}

\noindent
Proposition \ref{prop:conjugatevectorsquaresesqui} determines the Hermitian square $\hcE(K\xi)$ of the conjugate of $\xi\in \Sigma$ in terms of the Hermitian $\hcE(\xi)$ and complex-bilinear $\cE(\xi)$ squares of the latter. We proceed similarly to determine the complex-bilinear square $\cE(K\xi)$ of a conjugate element in terms of its Hermitian square $\hcE(\xi)$ and its complex-bilinear square $\cE(\xi)$.

\begin{prop}
\label{prop:conjugatevectorsquare}
Let $\xi \in \Sigma\setminus \{ 0\}$ and $T\in \Im(\cE)$. We have $T = \cE(K\xi)$ if and only if:
\begin{equation*}
\cE(\xi) \circ A\circ T = \sigma \bar{\kappa}^{2}\Tr(\hcE_\kappa(\xi) \circ A^{t}) \, \hcE_\kappa(\xi) 
\end{equation*}

\noindent
for an endomorphism $A\in \End(\Sigma)$ satisfying $\Tr(\hcE_\kappa(\xi) \circ A^{t}) \neq 0$.
\end{prop}

\begin{proof}
The \emph{only if} direction follows from Equation \eqref{eq:conjugatemix2} together with the non-degeneracy of the trace operator. For the converse, we set $T = \cE(\eta)$ for $\eta \in \Sigma\setminus \{ 0\}$ and we plug it into Equation \eqref{eq:conjugatemix2}, obtaining:
\begin{equation*}
\cE(\xi) \circ A\circ T = \scB(A(\eta) , \xi) \scB(- , \eta) \xi  =  \sigma  \scS( A^t(\xi) , \xi) \, \scS( - , \xi) \xi,
\end{equation*}

\noindent
which implies:
\begin{equation*}
\scB(A(\eta) , \xi) \eta  = \sigma \scS( A^t(\xi) , \xi)  K(\xi) = \scB(A(K(\xi)) , \xi) K(\xi).
\end{equation*}

\noindent
Since $\scB(A(K(\xi)) , \xi) \neq 0$ by assumption, the only solutions to this equation are  $K(\xi)$ and $ - K(\xi)$ and thus we conclude.  
\end{proof}

\noindent
We will use these propositions in later sections to characterize the \emph{square} of a conjugate spinor in terms of its complex-bilinear and Hermitian squares.


\subsection{Compatibility of square maps}


We have introduced two types of \emph{square maps}, namely:
\begin{equation*}
\hcE_{\kappa}\colon \Sigma \to \End(\Sigma), \qquad \cE\colon \Sigma \to \End(\Sigma),
\end{equation*}

\noindent
respectively associated with a choice of Hermitian $\scS$ and complex-bilinear $\scB$ non-degenerate pairings on $\Sigma$. Propositions \ref{prop:algebraichHermitian} and \ref{prop:algebraichbilinear} guarantee that a given endomorphism belongs to the image of $\hcE_{\kappa}$ or $\cE$, respectively. However, even if an endomorphism satisfies the conditions in both propositions, there is no guarantee that they are the image of the \emph{same} element in $\Sigma$, an issue that we examine in this subsection. We will assume that $\scS$ and $\scB$ are compatible in the sense introduced in the previous subsection. For every $\xi\in\Sigma$ we have:
\begin{equation*}
\hcE_\kappa^{-1}(\hcE_\kappa(\xi)) = \{ \mu \xi \mid \mu\in \U(1) \}\, , \qquad \cE^{-1}(\cE(\xi))= \{ \xi , - \xi\}
\end{equation*}

\noindent
and therefore an  element in $\Im(\hcE_\kappa) \subset \End(\Sigma)$ determines a vector in $\Sigma$ modulo a unitary phase, while an element in $\Im(\cE)$ determines an element in $\Sigma$ modulo a sign. The Hermitian and complex-bilinear squares of the \emph{same} element $\xi \in \Sigma$ satisfy natural compatibility identities. Indeed, a direct computation shows that: 
\begin{equation}
\label{eq:comp_sq_maps}
\hcE_\kappa(\xi)\circ A \circ \cE(\xi)=\Tr (\hcE_\kappa(\xi) \circ A )\cE(\xi)\, , \qquad \cE(\xi)\circ A \circ \hcE_\kappa(\xi)=\Tr (\cE(\xi)\circ A )\hcE_\kappa(\xi) 
\end{equation}

\noindent
as well as:
\begin{eqnarray}
\label{eq:comp_sq_mapsK}
\hcE_\kappa(\xi)\circ A \circ \hcE_\kappa(K\xi)= \kappa^2 \overline{\Tr\big(\cE(\xi) \circ A^{\dagger}\big)}\cE(\xi)\, , \qquad \cE(\xi) \circ A\circ \cE(K(\xi)) = \sigma \bar{\kappa}^{2} \Tr(\hcE_{\kappa}(\xi) \circ A^{t}) \, \hcE_{\kappa}(\xi)
\end{eqnarray}

\noindent
for every endomorphism $A\in\End(\Sigma)$. Our next goal is to determine the Hermitian square of an element $\xi\in \Sigma$, assuming we know the complex-bilinear squares of $\xi$ and its conjugate $K(\xi)$.

\begin{prop}
\label{prop:endoquadraticbi}
Let $\xi \in \Sigma \setminus \{ 0\}$. An endomorphism $\widehat{E}\in \Im(\hcE_\kappa)$ is the Hermitian square of $\xi$ if and only if:
\begin{equation} 
\label{eq:comp_sq_mapsII}
\widehat{E} \circ A \circ \cE(\xi)  =\Tr (\widehat{E} \circ A) \cE(\xi) \, , \qquad   \cE(\xi) \circ A\circ \cE(K(\xi)) = \sigma \bar{\kappa}^{2}\Tr(\widehat{E} \circ A^{t}) \, \widehat{E}
\end{equation}

\noindent
for an endomorphism $A\in \End(\Sigma)$ such that $\Tr (\widehat{E} \circ A ) \neq 0$.
\end{prop}
 
\begin{proof}
The only if direction follows from the first equation in \eqref{eq:comp_sq_maps} and the second equation in \eqref{eq:comp_sq_mapsK} together with the non-degeneracy of the trace operator. For the converse, suppose that we set $\widehat{E} = \hcE_\kappa(\eta)$ for a non-zero element $\eta \in \Sigma\setminus\{0\}$. Plugging this expression into the first equation in \eqref{eq:comp_sq_mapsII} we obtain:
\begin{equation*}
\scS(A \eta , \eta)\xi = \scS(A \xi , \eta) \eta.
\end{equation*}

\noindent
Since $\Tr (\widehat{E} \circ A ) = \kappa\scS(A\eta,\eta) \neq 0$, we conclude that $\eta = c \xi$ for a non-vanishing complex number $c\in\C\setminus\{0\}$. Substituting $\eta = c \xi$  into the second equation in \eqref{eq:comp_sq_mapsII} we obtain $\abs{c}^2 = 1$ and thus $c \in \U(1)$ is of unit norm. Since $\hcE_\kappa(\xi) = \hcE_\kappa(\mu\xi)$ for every $\mu\in\U(1)$, we obtain that $\hcE_\kappa(\xi) = \hcE_\kappa(\eta)$ and thus we conclude. 
\end{proof}

\noindent
The approach of Proposition \ref{prop:endoquadraticbi} cannot a priori be applied to determine the complex-bilinear square of an element $\xi\in\Sigma$ from the Hermitian squares of $\xi$ and its conjugate $K(\xi)$. This is because the Hermitian square of an element in $\Sigma$ only determines it modulo a unitary number rather than a sign, as it happens with the complex-bilinear square of an element in $\Sigma$.


\subsection{Algebraically constrained squares}
\label{subsec:constrainedsquares}


With a view towards understanding \emph{constrained} differential spinors and the spinorial notion of instanton that we will consider in later sections, we are interested in studying the kernel of a fixed endomorphism of $\Sigma$ in terms of the squares of the elements in $\Sigma$, as stated in the following result. 

\begin{lemma}
\label{lemma:constrainedvector}
Let $Q\in \End(\Sigma)$. Then, $\eta\in \Sigma$ satisfies $Q(\eta) = 0$ if and only if $Q\circ \cE(\eta) = 0$, if and only if $Q\circ \hcE_{\kappa}(\eta) = 0$.
\end{lemma}
 
\begin{proof}
It is enough to observe that $Q\circ \cE(\eta) = 0$ is equivalent to:
\begin{equation*}
\scB( \chi ,\eta) Q(\eta) = 0
\end{equation*}

\noindent
for every $\chi\in \Sigma$. Since $\scB$ is non-degenerate, choosing $\chi\in \Sigma$ such that $\scB( \chi ,\eta) \neq 0$, we conclude that $Q(\eta) = 0$ if and only if $Q\circ \cE(\eta) = 0$. The equivalence with $Q\circ \hcE_{\kappa}(\eta) = 0$ is proven analogously.
\end{proof}

\noindent
Lemma \ref{lemma:constrainedvector} shows that the algebraic equation $Q(\eta) = 0$ for $\eta \in \Sigma$ can be equivalently studied in terms of the Hermitian squares of the elements in $\Sigma$, even though these do not encode all the \emph{information} contained in $\eta$. This is a consequence of the simple fact that condition $Q(\eta) = 0$ implies $Q(c \eta) = 0$ for every $c \in \C$.


\section{Spinorial forms in even dimensions}
\label{section:even_forms}


In this section, we introduce the \emph{square maps} for complex irreducible spinors in even dimension and characterize their image in terms of a certain natural system of algebraic equations together with a key inequality. These quadratic maps are valued in the exterior algebra of the underlying quadratic vector space and define our notion of \emph{square} of a spinor. 


\subsection{The spinor square maps}


Let $(V,h)$ be a quadratic vector space of even dimension $d$. We will use the \emph{Kähler-Atiyah model} $(\wedge V^*,\diamond)$ for the universal Clifford algebra $\Cl(V^*,h^*)$ of the quadratic vector space $(V^*,h^*)$ dual to $(V,h)$. We briefly review this model in Appendix \ref{app:KA}. Let: 
\begin{equation*}
\mathrm{Cl}(V^*,h^*)\to\End(\Sigma)   
\end{equation*}

\noindent
be an irreducible complex representation of $\mathrm{Cl}(V^*,h^*)$ on a complex vector space $\Sigma$, whence $\dim_\C\Sigma=2^{\frac{d}{2}}$. Consider the complexification of the real Clifford algebra $\mathrm{Cl}(V^*,h^*)$: 
\begin{equation*}
\C\mathrm{l}(V^*,h^*):=\mathrm{Cl}(V^*,h^*)\otimes_\R\C\cong\mathrm{Cl}(V^*_\C,h^*_\C),
\end{equation*}

\noindent
where $(V^*_\C,h^*_\C)$ denotes the complexification of the quadratic vector space $(V^*,h^*)$. Then, since $d$ is even, the map:
\begin{equation*}
\gamma\colon\C\mathrm{l}(V^*,h^*)\to\End(\Sigma)
\end{equation*}

\noindent
is an isomorphism of unital and associative algebras.

\begin{remark}
Every element $z\in\C\mathrm{l}(V^*,h^*)$ is of the form $z=\sum_j x_j\otimes c_j$ where $x_j\in\mathrm{Cl}(V^*,h^*)$ and $c_j\in\C$. The conjugate of $z$ is well-defined and given by $\bar{z}=\sum_jx_j\otimes\bar{c}_j$.
\end{remark}

\noindent
Composing $\gamma$ with the complexification of the Chevalley-Riesz isomorphism \eqref{eq:ChevRiesz_iso}, which we denote by $\Psi:=\Psi_0\otimes\C$, gives an isomorphism of unital associative algebras which we denote by: 
\begin{equation*}
    \Psi_\gamma:=\gamma\circ\Psi\colon(\wedge V^*_\C,\diamond)\to(\End(\Sigma),\circ).
\end{equation*}

\noindent
Using the isomorphism $\Psi_\gamma$ we can define a trace in the Kähler-Atiyah algebra.

\begin{definition}
The \emph{Kähler-Atiyah trace} is the linear functional given by: 
\begin{equation*}
\mathcal{T}\colon \wedge V^*_\C\to\C,\quad\alpha\mapsto\Tr(\Psi_\gamma(\alpha)).
\end{equation*}
\end{definition}

\noindent
Since $\Psi_\gamma\colon\wedge V^*_\C\to\End(\Sigma)$ is a unital isomorphism of algebras, we have that:
\begin{equation*}
 \mathcal{T}(1)=\Tr(\Psi_\gamma(1))=\Tr(\Id_\Sigma)=\dim_\C\Sigma=2^{\frac{d}{2}}\quad\text{and}\quad\mathcal{T}(\alpha_1\diamond\alpha_2)=\mathcal{T}(\alpha_2\diamond\alpha_1),
\end{equation*}

\noindent
where $1\in\C=\wedge^0V^*_\C$.

\begin{prop}
For any $\alpha\in\wedge V^*_\C$, the Kähler-Atiyah trace is given by: 
\begin{equation*}
 \mathcal{T}(\alpha)=2^{\frac{d}{2}}\alpha^{(0)},
\end{equation*}

\noindent
where $\alpha^{(0)}$ denotes the degree zero component of the exterior form $\alpha$.
\end{prop}

\begin{proof}
Let $\{e^1, \hdots, e^d\}$ be an orthonormal basis of $(V^{\ast},h^{\ast})$. For $i,j = 1 , \hdots, n$, $i\neq j$, we have $e^i\diamond e^j = - e^j\diamond e^i$ and hence $(e^i)^{-1}\diamond e^j\diamond e^i = - e^j$. Let $0 < k \leq d$ and $1\leq i_1 < \cdots < i_k \leq d$. If $k$ is even we have:
\begin{equation*}
\cT(e^{i_1}\diamond \dots \diamond e^{i_{k}}) = \cT(e^{i_k}\diamond
e^{i_1}\diamond \dots \diamond e^{i_{k-1}}) = (-1)^{k-1}
\cT(e^{i_1}\diamond \dots \diamond e^{i_{k}})
\end{equation*}

\noindent
and hence $\cT(e^{i_1}\diamond \dots \diamond e^{i_{k}}) = 0$. Here we used the cyclicity of the Kähler-Atiyah trace and the fact that $e^{i_{k}}$ anti-commutes with $e^{i_1}, \dots, e^{i_{k-1}}$. If $k$ is odd, let $j\in\{1,\dots,d\}$ be such that $j\not\in\{i_1,\dots,i_k\}$. Note that such a $j$ exists since $k<d$. We have:
\begin{equation*}
\cT(e^{i_1}\diamond \dots \diamond e^{i_{k}}) = \cT(( e^j)^{-1} \diamond e^{i_1}\diamond \dots 
\diamond e^{i_{k}}\diamond e^j) = -  \cT( e^{i_1}\diamond \dots \diamond e^{i_{k}}),
\end{equation*}

\noindent
which implies that $\mathcal{T}(e^{i_1}\diamond \dots \diamond e^{i_k})=0$ and hence we conclude.
\end{proof}

\noindent
To define the spinor square maps we combine the isomorphism $\Psi_\gamma\colon\wedge V^*_\C\to\End(\Sigma)$ that we introduced above together with the squaring maps $\cE\colon\Sigma\to\End(\Sigma)$ and $\hcE_\kappa\colon\Sigma\to\End(\Sigma)$ introduced in the previous section as quadratic maps associated to a certain choice of non-degenerate complex-bilinear pairing $\scB$ and a certain choice of Hermitian pairing $\scS$ on $\Sigma$, respectively. For our purposes, we cannot use an arbitrary choice of non-degenerate complex-bilinear or Hermitian pairings on $\Sigma$. Instead, it is convenient to work with non-degenerate pairings on $\Sigma$ that are adapted to its structure as a complex Clifford module. This leads to the notion of \emph{admissible pairing}, introduced in \cite{AC97,ACDVP05}. In our case, we must distinguish between Hermitian and complex-bilinear admissible pairings.

\begin{definition}\label{def:admissible_pairing}
Let $(\Sigma,\gamma)$ be an irreducible complex Clifford module. An \emph{admissible Hermitian pairing} on $(\Sigma,\gamma)$ is a non-degenerate Hermitian pairing $\scS$ on $\Sigma$ satisfying:
\begin{equation*}
\scS(\gamma(z)\xi,\chi)=\scS(\xi,\gamma(\overline{(\pi^{\frac{1-s}{2}}\circ\tau)(z)})\chi),\quad s\in\mathbb{Z}_2
\end{equation*}

\noindent
for all $z\in\C\mathrm{l}(V^*,h^*)$ and $\xi,\chi\in\Sigma$. Similarly, an \emph{admissible complex-bilinear pairing} on $(\Sigma,\gamma)$ is a non-degenerate complex-bilinear pairing $\scB$ on $\Sigma$ satisfying:
\begin{equation*}
\scB(\gamma(z)\xi,\chi)=\scB(\xi,\gamma((\pi^{\frac{1-s}{2}}\circ\tau)(z))\chi),\quad s\in\mathbb{Z}_2
\end{equation*}

\noindent
again for all $z\in\C\mathrm{l}(V^*,h^*)$ and $\xi,\chi\in\Sigma$. The sign factor $s$ is called the \emph{adjoint type} of the admissible pairing. We say that an admissible pairing is of \emph{positive adjoint type} if $s=+1$ and of \emph{negative adjoint type} if $s=-1$.
\end{definition}

\noindent
The existence of both Hermitian and complex-bilinear admissible pairings on every irreducible complex Clifford module is respectively proven in Proposition \ref{prop:Hermitian_admissible} and Proposition \ref{prop:complex_bilinear_admissible} using an idea that can be traced back to \cite{Baum81}, see also \cite{HarveyBook}. We will always assume that, given an irreducible complex Clifford module $(\Sigma,\gamma)$, the square maps $\cE\colon\Sigma\to\End(\Sigma)$ and $\hcE_\kappa\colon\Sigma\to\End(\Sigma)$ are constructed using an admissible pairing. We will refer to such triples $(\Sigma, \gamma ,\scS)$ and $(\Sigma, \gamma ,\scB)$ as \emph{Hermitian} and \emph{complex-bilinear paired Clifford modules}. We have the following diagram: 
\begin{equation*}
\begin{tikzcd}
{\mathbb{C}\mathrm{l}(V^*,h^*)} \arrow[r, "\gamma"]  & \mathrm{End}(\Sigma) & \Sigma \arrow[l, " \widehat{\mathcal{E}}_\kappa / \mathcal{E}"'] \\
\wedge V^*_{\mathbb{C}} \arrow[u, "\Psi"] \arrow[ru, "\Psi_\gamma=\gamma\circ\Psi"'] &                      &             
\end{tikzcd}
\end{equation*}

\noindent
Consequently, we define the \emph{complex-bilinear} and \emph{Hermitian square spinor maps} respectively as follows:
\begin{equation*}
\cE_\gamma:=\Psi_\gamma^{-1}\circ\cE\colon\Sigma\to\wedge V^*_\C,\qquad \hcE_\gamma^\kappa:=\Psi_\gamma^{-1}\circ\hcE_\kappa\colon\Sigma\to\wedge V^*_\C\, .
\end{equation*}

\noindent
We define the square of an even-dimensional irreducible complex spinor in terms of these square spinor maps. 

\begin{definition}
\label{def:squares}
The \emph{complex-bilinear square} of a complex irreducible spinor $\xi \in \Sigma$ in even dimensions is the complex exterior form $\cE_\gamma(\xi)\in \wedge V^*_{\mathbb{C}}$. The \emph{Hermitian square} of a complex irreducible spinor $\xi \in \Sigma$ in even dimensions is the complex exterior form $\hcE_\gamma^\kappa(\xi)\in \wedge V^*_{\mathbb{C}}$ for $\kappa\in\U(1)$.
\end{definition}

\begin{remark}
Note that an element in the image of $\hcE_\gamma^\kappa$ determines a complex spinor uniquely modulo a multiplicative unitary complex number. On the other hand, an element in the image of $\cE_\gamma$ determines a complex spinor uniquely modulo a sign. Hence, the complex-bilinear square of a complex spinor contains \emph{more} information than its Hermitian square and is equivalent to the spinor itself modulo a sign. This fact is key to developing the theory of \emph{complex spinorial forms} and its applications to the study of spinors parallel under a general connection on the spinor bundle.
\end{remark}

\noindent
We introduce the following terminology, which we will use to study spinors that lie in the kernel of a given endomorphism in terms of their Hermitian or complex-bilinear square.

\begin{definition}
Let $(\Sigma,\gamma)$ be an irreducible complex Clifford module. The \emph{dequantization} $\frq \in \wedge V^{\ast}_{\C}$ of an endomorphism $Q\in \End(\Sigma)$ is $\frq := \Psi_{\gamma}^{-1}(Q)$.
\end{definition}

\noindent
As an immediate consequence of Lemma \ref{lemma:constrainedvector}, we have:

\begin{lemma}
\label{lemma:constrainedspinoreven}
Let $Q\in \End(\Sigma)$. Then, $\eta\in \Sigma$ satisfies $Q(\eta) = 0$ if and only if $\mathfrak{q}\diamond \cE_{\gamma}(\eta) = 0$, if and only if $\mathfrak{q}\diamond \hcE^{\kappa}_{\gamma}(\eta) = 0$.
\end{lemma}

\noindent
By definition, complex-bilinear spinorial forms are elements in $\Im(\cE_\gamma)\subset \wedge V^*_{\mathbb{C}}$, whereas Hermitian spinorial forms are elements in $\Im(\hcE_\gamma^\kappa)\subset \wedge V^*_{\mathbb{C}}$. In the following subsections, we consider these two \emph{squares} separately.


\subsection{Hermitian spinorial forms}
\label{subsec:Hermitian}


In this subsection, we give the algebraic characterization of the complex spinorial forms associated to a Hermitian paired Clifford module $(\Sigma,\gamma,\scS)$. We begin by proving the existence of admissible Hermitian pairings on every irreducible complex Clifford module $(\Sigma,\gamma)$.  

\begin{prop}
\label{prop:Hermitian_admissible}
Every irreducible complex Clifford module $(\Sigma,\gamma)$ admits two non-degenerate Hermitian pairings $\scS_\pm \colon \Sigma \times \Sigma \to\C$ satisfying: 
\begin{equation}
\label{eq:Hermitian_admissible_Ss}
\scS_+(\gamma(z)\xi,\chi)=\scS_+(\xi,\gamma(\overline{\tau(z)})\chi)\quad\text{and}\quad\scS_-(\gamma(z)\xi,\chi)=\scS_-(\xi,\gamma(\overline{(\pi\circ\tau)(z)})\chi)
\end{equation} 

\noindent
for every $z\in\C\mathrm{l}(V^*,h^*)$ and $\xi,\chi\in\Sigma$. 
\end{prop}
 
\begin{proof}
Let $\{e^1,\ldots,e^{d}\}$ be an $h^*$-orthonormal basis of $V^*$, $d=\dim_\R V$, and let:
\begin{equation}\label{eq:group_units_K}
\mathfrak{K} := \{1\}\cup\{\pm e^{i_1}\cdots e^{i_k}\mid 1\leq i_1<\cdots<i_k\leq d,\,1\leq k\leq d\}
\end{equation}

\noindent
be the finite multiplicative subgroup of $\C\mathrm{l}(V^*,h^*)$ generated by $\pm e^i$. Note that $\C\mathrm{l}(V^*,h^*)=\Span_\C\{ \mathfrak{K} \}$. Let $\beta \colon \Sigma \times \Sigma \to \C$ be a non-degenerate positive-definite Hermitian pairing. Then we can construct a $\mathfrak{K}$-invariant Hermitian inner product on $\Sigma$ by averaging over $\mathfrak{K}$:

\begin{equation*}
\escal{\xi,\chi}:=\frac{1}{\abs{\mathfrak{K}}}\sum_{k\in \mathfrak{K}}\beta(\gamma(k)\xi,\gamma(k)\chi).
\end{equation*}

\noindent
This pairing satisfies $\escal{\gamma(k)\xi,\gamma(k)\chi}=\escal{\xi,\chi}$ for all $k\in \mathfrak{K}$. Write $V^*=V^*_+\oplus V^*_-$, where $V^*_+$ is a $p$-dimensional subspace on which $h^*$ is positive-definite and $V^*_-$ is a $q$-dimensional subspace on which $h^*$ is negative-definite. Fix volume forms $\nu_+\in\wedge^pV^*_+$ and $\nu_-\in\wedge^qV^*_-$ so that $\nu = \nu_+\wedge\nu_-$ for the pseudo-Riemannian volume form $\nu$ on $(V,h)$. Define the non-degenerate pairings $\scS_{+} \colon\Sigma\times\Sigma\to\C$ and $\scS_{-} \colon\Sigma\times\Sigma\to\C$ by:

\begin{align*}
\scS_{+}(\xi,\chi)&:=\begin{cases}
    i^{\frac{p(p-1)}{2}}\escal{\Psi_\gamma(\nu_{+})\xi,\chi}&\text{ if $p$ (and hence $q$) is odd},\\
    i^{\frac{q(q+1)}{2}} \escal{\Psi_\gamma(\nu_{-})\xi,\chi}&\text{ if $p$ (and hence $q$) is even},
    \end{cases}\\
\scS_{-}(\xi,\chi)&:=\begin{cases}
    i^{\frac{q(q+1)}{2}} \escal{\Psi_\gamma(\nu_{-})\xi,\chi}&\text{ if $p$ (and hence $q$) is odd},\\
    i^{\frac{p(p-1)}{2}}\escal{\Psi_\gamma(\nu_{+})\xi,\chi}&\text{ if $p$ (and hence $q$) is even}
    \end{cases}
\end{align*}

\noindent
for all $\xi,\chi\in\Sigma$. A direct computation shows that both $\scS_{+}$ and $\scS_{+}$ are non-degenerate Hermitian pairings on $\Sigma$. Furthermore, for every $e^i\in V^*$ and every $\xi , \chi \in \Sigma$ we have:
\begin{eqnarray*}
\scS_{+}(\Psi_\gamma(e^i)\xi , \chi) = \scS_{+}( \xi , \Psi_\gamma(e^i)\chi)\, , \qquad \scS_{-}(\Psi_\gamma(e^i)\xi , \chi) = - \scS_{-}( \xi , \Psi_\gamma(e^i)\chi)
\end{eqnarray*}

\noindent
and hence we conclude.
\end{proof}

\noindent
A direct computation proves the following result.

\begin{prop}
\label{prop:relationhermitianproducts}
The admissible Hermitian pairings $\scS_+$ and $\scS_-$ on $(\Sigma,\gamma)$ are related as follows: 
\begin{equation*}
\scS_-(\xi,\chi) = i^{\binom{p}{2}}i^{\binom{q+1}{2}}\scS_+(\Psi_\gamma(\nu)\xi,\chi)
\end{equation*}

\noindent
for every $\xi,\chi \in \Sigma$.
\end{prop} 

\begin{remark}
Relations \eqref{eq:Hermitian_admissible_Ss} can be uniformly written as:
\begin{equation}
\label{eq:gamma_dagger}
\gamma(z)^\dagger=\gamma(\overline{(\pi^{\frac{1-s}{2}}\circ\tau)(z)})
\end{equation} 
for all $z\in\C\mathrm{l}(V^*,h^*)$, where $\gamma(z)^\dagger$ denotes the adjoint with respect to the admissible Hermitian pairing $\scS_s$ of \emph{adjoint type} $s\in\Z_2$.
\end{remark}

\begin{lemma}
\label{lemma:exterior_form_dagger}
Let $\alpha\in\wedge V^*_\C$ and let $\scS_s$ be an admissible Hermitian pairing of adjoint type $s\in\Z_2$. Then:
\begin{equation*}
\Psi_\gamma(\alpha)^\dagger=\Psi_\gamma(\overline{(\pi^{\frac{1-s}{2}}\circ\tau)(\alpha)}).
\end{equation*}
\end{lemma}

\begin{proof}
Note that $\pi\circ\Psi=\Psi\circ\pi$ and $\tau\circ\Psi=\Psi\circ\tau$. Using relation \eqref{eq:gamma_dagger} we get: 
\begin{eqnarray*}
 \Psi_\gamma(\alpha)^\dagger=\gamma(\Psi(\alpha))^\dagger=\gamma(\overline{(\pi^{\frac{1-s}{2}}\circ\tau)(\Psi(\alpha)}))=\gamma(\Psi(\overline{(\pi^{\frac{1-s}{2}}\circ\tau)(\alpha)}))=\Psi_\gamma(\overline{(\pi^{\frac{1-s}{2}}\circ\tau)(\alpha)})   
\end{eqnarray*}

\noindent
for every $\alpha\in\wedge V^*_\C$.
\end{proof}

\begin{thm}
\label{thm:even_Hermitian_forms}
Let $(\Sigma,\gamma)$ be an irreducible complex Clifford module equipped with an admissible Hermitian pairing $\scS_s$ of adjoint type $s\in \mathbb{Z}_2$. Then the following statements are equivalent for a complex exterior form $\alpha\in\wedge V^*_\C$: 
\begin{enumerate}
\item[\normalfont(a)] $\alpha$ is the Hermitian square of a spinor $\xi\in\Sigma$, that is, $\alpha = \hcE_\gamma^\kappa(\xi)$ for some $\kappa\in\U(1)$.
\item[\normalfont(b)] $\alpha$ satisfies the following relations:
\begin{equation*}
\alpha\diamond\alpha = 2^{\frac{d}{2}}\alpha^{(0)}\alpha,\qquad(\pi^{\frac{1-s}{2}}\circ\tau)(\bar\kappa\alpha) = \kappa \bar{\alpha}\, , \qquad \alpha \diamond \beta \diamond \alpha=2^{\frac{d}{2}}(\alpha\diamond\beta)^{(0)}\alpha
\end{equation*}

\noindent
for a fixed exterior form $\beta\in\wedge V^*_\C$ satisfying $(\alpha\diamond\beta)^{(0)}\neq0$.
\item[\normalfont(c)] The following equations hold: 
\begin{equation*}
(\pi^{\frac{1-s}{2}}\circ\tau)(\bar\kappa\alpha) = \kappa \bar{\alpha}\, , \qquad \alpha \diamond \beta \diamond \alpha=2^{\frac{d}{2}}(\alpha\diamond\beta)^{(0)}\alpha
\end{equation*}

\noindent
for every exterior form $\beta\in\wedge V^*_\C$.
\end{enumerate}
\end{thm}

\begin{proof}
Since $\Psi_\gamma\colon\wedge V^*_\C\to\End(\Sigma)$ is a unital isomorphism of algebras, $\alpha$ satisfies the conditions in (c) if and only if: 
\begin{eqnarray*}
(\bar{\kappa}E)^\dagger = \bar{\kappa}E\quad\text{and}\quad E\circ A\circ E = \Tr(E\circ A) E \quad \forall A\in\End(\Sigma),
\end{eqnarray*}

\noindent
where $E:=\Psi_\gamma(\alpha)$ and $A:=\Psi_\gamma(\beta)$. The conclusion follows then by using Lemma \ref{lemma:exterior_form_dagger} together with the definition and properties of the Kähler-Atiyah trace, and Proposition \ref{prop:algebraichHermitian}.
\end{proof}

\noindent
By the previous theorem, the collection of all Hermitian squares of irreducible complex spinors is a semi-algebraic subset of $\wedge V^{\ast}_{\C}$ depending exclusively on $s\in \mathbb{Z}_2$ and $\kappa \in \U(1)$. Let $(\Sigma,\gamma,\scS)$ be an irreducible Hermitian paired Clifford module. Denote by $\dot{\Sigma}_{\scS} \subset \Sigma$ the open set of non-vanishing norm spinors. By Theorem \ref{thm:even_Hermitian_forms} it follows that:
\begin{equation*}
\hcE_\gamma^\kappa(\dot{\Sigma}_{\scS}) = \left\{ \alpha\in \wedge V^{\ast}_{\C} \,\, \vert\,\, \alpha^{(0)} \neq 0\, ,  \,\, \alpha\diamond\alpha = 2^{\frac{d}{2}}\alpha^{(0)}\alpha\, , \,\, (\pi^{\frac{1-s}{2}}\circ\tau)(\bar\kappa\alpha) = \kappa \bar{\alpha}\, , \,\, \kappa \in \U(1) \right\}.
\end{equation*}

\noindent
Hence, $\hcE_\gamma^\kappa(\dot{\Sigma}_{\scS})$ is a \emph{real} algebraic subset of a complex cone in an open subset of $\wedge V^{\ast}_{\C}$ determined by the condition $\alpha^{(0)} \neq 0$.\medskip

\noindent
Let $\{e^1, \hdots , e^d\}$ be an orthonormal basis of $(V^{\ast} , h^{\ast})$. For simplicity in the exposition, in the following we set:
\begin{eqnarray*}
\gamma^i := \Psi_{\gamma}(e^i)\, , \quad i= 1,\hdots , d\, , \qquad \gamma^{i_1\cdots i_k} := \gamma^{i_1} \cdots \gamma^{i_k} \, , \quad i_1, \hdots , i_k = 1, \hdots , d.
\end{eqnarray*}

\noindent
These are the celebrated \emph{gamma matrices} associated to the Clifford module $(\Sigma,\gamma)$ and the choice of orthonormal basis $\{e^1, \hdots , e^d\}$. For simplicity in the exposition, we set:
\begin{equation*}
e_i := h^{\ast}(e^i , e^i) e^i \, , \qquad \gamma_i := h^{\ast}(e^i , e^i) \gamma^i\, , \quad i= 1,\hdots , d.
\end{equation*}

\noindent
Since the map $\Psi_{\gamma} \colon \wedge V^{\ast}_{\C} \to \End(\Sigma)$ is an isomorphism of unital and associative complex algebras, the set of those endomorphisms of $\Sigma$ consisting of the identity together with all matrices of the form $\gamma^{i_1\cdots i_k}$ with $i_1 < \cdots < i_k$ and $k=1,\hdots , d$ defines a basis of $\End(\Sigma)$ that is orthogonal with respect to the natural complex-bilinear pairing on $\End(\Sigma)$ defined in terms of the complex trace. Set $\widehat{E}_{\eta} := \hcE_{\kappa}(\eta)$ and $\widehat{\alpha}_{\eta} := \hcE^{\kappa}_{\gamma}(\eta) \in \wedge V^{\ast}_{\C}$ for $\eta \in \Sigma\setminus\{ 0 \}$. Then:
\begin{eqnarray*}
\Tr(\widehat{E}_{\eta} \circ (\gamma^{i_1 \cdots i_k})^{-1}) = \kappa \, \scS((\gamma^{i_1 \cdots i_k})^{-1}\eta , \eta)\, , \qquad  i_1 < \cdots < i_k \, , \qquad k=1,\hdots , d,
\end{eqnarray*}

\noindent
which immediately implies:
\begin{eqnarray*}
\widehat{E}_{\eta} = \frac{\kappa}{2^{\frac{d}{2}}} \scS(\eta,\eta)\, \Id + \frac{\kappa}{2^{\frac{d}{2}}} \sum_{k=1}^d \sum_{i_1 < \cdots <i_k} \scS((\gamma^{i_1 \cdots i_k})^{-1}\eta , \eta)\gamma^{i_1 \cdots i_k}.
\end{eqnarray*}

\noindent
Applying $\Psi^{-1}_{\gamma} \colon \End(\Sigma) \to \wedge V^{\ast}_{\C}$ to the previous expression, and using that it is an isomorphism of unital and associative algebras, we obtain:
\begin{eqnarray}
\label{eq:explicitHermitian}
\widehat{\alpha}_{\eta} = \frac{\kappa}{2^{\frac{d}{2}}} \scS(\eta,\eta) + \frac{\kappa}{2^{\frac{d}{2}}} \sum_{k=1}^d \sum_{i_1 < \cdots < i_k} \scS((\gamma^{i_1 \cdots i_k})^{-1}\eta , \eta)\, e^{i_1}\wedge  \cdots \wedge e^{i_k}\, .
\end{eqnarray}

\noindent
This gives the explicit form of the Hermitian square of an irreducible complex spinor $\eta \in \Sigma$ in terms of an explicit orthonormal basis of $(V^{\ast} , h^{\ast})$. Since the admissible Hermitian pairings $\scS_{+}$ and $\scS_{-}$ are related as stated in Proposition \ref{prop:relationhermitianproducts}, it is natural to expect that the Hermitian square of $\eta\in\Sigma$ relative to $\scS_{+}$ is related to the Hermitian square of $\eta\in \Sigma$ relative to $\scS_{-}$. This is indeed the case. Denote the former by $\widehat{\alpha}_{\eta}^{+} \in \wedge V^{\ast}_{\C}$ and the latter by $\widehat{\alpha}_{\eta}^{-} \in \wedge V^{\ast}_{\C}$. A direct computation shows that:
\begin{equation*}
(\widehat{\alpha}_\eta^-)^{(k)} = i^{\binom{p}{2}}i^{\binom{q+1}{2}}\frac{k!}{(d-k)!}*\tau((\widehat{\alpha}_\eta^+)^{(d-k)}).
\end{equation*}

\noindent
Therefore, the choice of an admissible Hermitian pairing to construct the Hermitian square of a spinor is a matter of taste and computational convenience. 


\subsection{Equivariance properties of the Hermitian square map}
\label{subsec:Hermitianequivariance}


We consider now the equivariance properties of the Hermitian square map. Let us denote by:
\begin{equation*}
\Spin_o^c (p,q) := \Spin_o(p,q)\cdot \U(1) = \big(\Spin_o(p,q)\times \U(1)\big)/\Z_2 \subset \mathbb{C}\mathrm{l}(V^*,h^*)    
\end{equation*}

\noindent
the identity component of the spin$^c$ group of $(V,h)$, where $\Spin_o(p,q)$ denotes the identity component of the spin group of $(V,h)$. Hence, we can consider elements in $\Spin_o^c(p,q)$ as equivalence classes $[u,z]$ with $u\in\Spin_o(p,q)$ and $z\in\U(1)$. Note that $\scS$ is invariant under $\Spin_o^c(p,q)$. For every $\eta\in \Sigma$ and $[u,z] \in \Spin_o^c(p,q)$, we compute:
\begin{eqnarray*}
\hcE_\gamma^\kappa(\gamma([u,z])\eta) = \hcE_\gamma^\kappa(z\gamma(u)\eta) =\hcE_\gamma^\kappa( \gamma(u)\eta) = \kappa \Psi_{\gamma}^{-1}(\scS( - , \gamma(u)\eta) \gamma(u)\eta) = \Ad_u(\hcE_\gamma^\kappa(\eta)),
\end{eqnarray*}

\noindent
where $\Ad\colon \Spin_o^c(p,q) \to \SO_o(p,q)$ is the natural covering of the identity component $\SO_o(p,q)$ of the special orthogonal group of $(V,h)$ induced by its double cover $\Spin_o(p,q)$. Here we have used Equation \eqref{eq:explicitHermitian} together with the following identities:
\begin{align*}
\scS((\gamma^{i_1 \cdots i_k})^{-1}\gamma(u)\eta , \gamma(u)\eta)\, e^{i_1}\wedge  \cdots \wedge e^{i_k} &= \scS((\gamma(u) \gamma^{i_1 \cdots i_k} \gamma(u)^{-1})^{-1}\eta , \eta)\, e^{i_1}\wedge  \cdots \wedge e^{i_k}\\
&= \scS(\gamma(u \Psi(e^{i_1}) \cdots \Psi(e^{i_k}) u^{-1})^{-1}\eta , \eta)\, e^{i_1}\wedge  \cdots \wedge e^{i_k}\\
&= \scS(\gamma(u \Psi(e^{i_1}) u^{-1} \cdots u \Psi(e^{i_k}) u^{-1})^{-1}\eta , \eta)\, e^{i_1}\wedge  \cdots \wedge e^{i_k}.
\end{align*}

\noindent
Hence, we conclude that the map $\hcE_\gamma^\kappa \colon \Sigma \to \wedge V^*_\C$ is equivariant with respect to the adjoint representation $\Ad\colon \Spin_o^c(p,q) \to \SO_o(p,q)$ of $\Spin_o^c(p,q)$ on $\wedge V^{\ast}_\C$.


\subsection{Complex-bilinear spinorial forms}
\label{subsec:bilinear}


In this subsection, we obtain the algebraic characterization of the complex spinorial forms associated to a complex-bilinear paired Clifford module $(\Sigma,\gamma,\scB)$. We begin by proving the existence of admissible complex-bilinear pairings on every irreducible complex Clifford module $(\Sigma,\gamma)$ associated to an even-dimensional quadratic vector space $(V,h)$. In order to do so, we need to distinguish between complex Clifford modules of \emph{real} and \emph{quaternionic} type.

\begin{definition}
Let $\Sigma$ be a complex vector space. A \emph{real structure} on $\Sigma$ is a complex anti-linear map $\frc \colon \Sigma \to \Sigma$ such that $\frc^2 = \Id$. A \emph{quaternionic structure} on $\Sigma$ is a complex anti-linear map $J\colon\Sigma\to\Sigma$ satisfying $J^2=-\Id$. 
\end{definition}

\begin{definition}
Let $(\Sigma,\gamma)$ be a complex Clifford module for $\Cl(V^*,h^*)$. A \emph{real structure} on $(\Sigma,\gamma)$ is a real structure $\frc$ on $\Sigma$ such that $\gamma(z) \circ \frc = \frc\circ \gamma(\bar z)$ for all $z\in\C\mathrm{l}(V^*,h^*)$. Similarly, a \emph{quaternionic structure} on $(\Sigma,\gamma)$ is a quaternionic structure $J$ on $\Sigma$ such that $\gamma(z) \circ J = J \circ \gamma(\bar z)$. 
\end{definition}

\noindent
Every irreducible complex Clifford module $(\Sigma,\gamma)$ in even dimensions admits either a real structure, in which case we say it is of \emph{real type}, or a quaternionic structure, in which case we say it is of \emph{quaternionic type} \cite{HarveyBook}.

\begin{lemma}\label{lemma:real_or_quaternionic_type}
Let $(\Sigma,\gamma)$ be an irreducible complex Clifford module for $\Cl(V^*,h^*)$ with $(V,h)$ even-dimensional of signature $(p,q)$. Then:
\begin{itemize}
    \item If $p-q\equiv_8 0,2$ then $(\Sigma,\gamma)$ is of real type.
    \item If $p-q\equiv_8 4,6$ then $(\Sigma,\gamma)$ is of quaternionic type.
\end{itemize}
\end{lemma}

\begin{proof}
When $p-q\equiv_8 0,2$, the irreducible complex representation $\gamma\colon\mathbb{C}\mathrm{l}(V^*,h^*) \to \End(\Sigma)$ is the complexification of the unique, modulo isomorphism, real irreducible Clifford module for $\Cl(V^*,h^*)$, a fact that immediately implies the existence of a real structure on $(\Sigma,\gamma)$. If $p-q\equiv_8 4,6$ then the corresponding irreducible real Clifford module is of quaternionic type, in the sense that its Schur algebra is of quaternionic type, and hence contains two anti-commuting complex structures $J_1$ and $J_2$. Choosing either of them, say $J_1$, to complexify the aforementioned irreducible real Clifford module, we obtain, modulo isomorphism, the irreducible complex Clifford module $(\Sigma,\gamma)$. Then, the other complex structure $J_2$ defines a quaternionic structure on $(\Sigma,\gamma)$.
\end{proof}

\noindent
Proposition \ref{prop:Hermitian_admissible} establishes the existence of a pair of admissible Hermitian pairings on $(\Sigma,\gamma)$. We establish now the existence of admissible Hermitian pairings that are \emph{compatible} with any given choice of real or quaternionic structure on $(\Sigma,\gamma)$.  

\begin{lemma}
\label{lemma:compatible_admissible_pairings}
Let $(\Sigma,\gamma)$ be an irreducible complex Clifford module equipped with a real or quaternionic structure $K\colon \Sigma \to \Sigma$ and let $s\in\mathbb{Z}_2$. Then $(\Sigma,\gamma)$ can be equipped with an admissible Hermitian pairing $\scS_s$ of adjoint type $s\in\mathbb{Z}_2$ such that:
\begin{eqnarray*}
\scS_{s}(K \xi , K \chi ) = (-1)^{\binom{p}{2}+\frac{d}{4}(1+s(-1)^p)}\, \overline{\scS_{s}( \xi , \chi )} = (-1)^{\frac{1}{4}(2p(p-1) + (1+s(-1)^p) d)}\, \overline{\scS_{s}( \xi , \chi )}
\end{eqnarray*}
 for every $\xi , \chi \in \Sigma$.
 \end{lemma}
 
\begin{proof}
The proof is similar to that of Proposition \ref{prop:Hermitian_admissible}. We start with a non-degenerate positive-definite Hermitian pairing on $\Sigma$ and we construct another Hermitian inner product $\escal{\cdot,\cdot}'$ on $\Sigma$ invariant under the action of the group \eqref{eq:group_units_K}. Using the invariant Hermitian inner product $\escal{\cdot,\cdot}'$ we construct a third Hermitian inner product as follows: 
\begin{equation*}
 \escal{\xi,\chi}=\frac{1}{2}(\escal{\xi,\chi}'+\overline{\escal{K\xi,K\chi}'}) \, , \qquad \forall\,\, \xi,\chi\in\Sigma.
 \end{equation*}

\noindent
This third Hermitian inner product satisfies $\escal{K\xi,K\chi}=\overline{\escal{\xi,\chi}}$ for all $\xi,\chi\in \Sigma$. Using $\escal{\cdot,\cdot}$ we construct the Hermitian pairings $\scS_\pm$ as in the proof of Proposition \ref{prop:Hermitian_admissible} and then a direct computation shows that the claimed properties are satisfied.
\end{proof}

\noindent
By Lemma \ref{lemma:compatible_admissible_pairings}, given a choice of real or quaternionic structure $K\colon \Sigma\to \Sigma$ on $(\Sigma,\gamma)$, we can always find an admissible Hermitian pairing $\scS_s$ that is either \emph{invariant}, in the sense that:
\begin{equation*}
\scS_s(K(\xi),K(\chi))=\overline{\scS_s(\xi,\chi)}  \, , \qquad \forall\,\, \xi, \chi \in \Sigma,
\end{equation*}

\noindent
or \emph{anti-invariant}, in the sense that:
\begin{equation*}
\scS_s(K(\xi),K(\chi))= - \overline{\scS_s(\xi,\chi)}  \, , \qquad \forall\,\, \xi, \chi \in \Sigma.
\end{equation*}

\noindent
We will refer to such Hermitian pairings simply as \emph{compatible}. Note that if $(V,h)$ is Euclidean, then $\scS_{+}$ can always be chosen to be invariant with respect to $K$. On the other hand, if $h$ is negative-definite, then it is $\scS_{-}$ that can always be chosen to be invariant with respect to $K$.\medskip

\noindent
Using the compatible admissible Hermitian pairings $\scS_{+}$ and $\scS_{-}$ constructed in Lemma \ref{lemma:compatible_admissible_pairings} we obtain admissible complex-bilinear pairings on $(\Sigma,\gamma)$ as stated in the following proposition.

\begin{prop}
\label{prop:complex_bilinear_admissible}
Every irreducible complex Clifford module $(\Sigma,\gamma)$ in even dimensions admits two non-degenerate complex-bilinear pairings $\scB_{\pm} \colon \Sigma \times \Sigma \to\C$ such that:
\begin{equation}
\label{eq:complex_bilinear_admissible_Bs}
\scB_+(\gamma(z)\xi,\chi)=\scB_+(\xi,\gamma(\tau(z))\chi)\quad\text{and}\quad\scB_-(\gamma(z)\xi,\chi)=\scB_-(\xi,\gamma((\pi\circ\tau)(z))\chi)
\end{equation} 

\noindent
for all $z\in\C\mathrm{l}(V^*,h^*)$ and $\xi,\chi\in\Sigma$. 
\end{prop}

\begin{proof}
 Let $(\Sigma,\gamma)$ be an irreducible complex Clifford module. If $(\Sigma,\gamma)$ is of real type, we choose a real structure $\frc$ and a compatible admissible Hermitian pairing $\scS$ on $(\Sigma,\gamma)$ and define:
 \begin{equation*}
\scB_{s}(\xi,\chi) := \scS_{s}(\xi,\frc(\chi)) \qquad \forall\,\, \xi, \chi \in \Sigma\, .
 \end{equation*}

\noindent
Then $\scB_{s}$ has the same adjoint type as $\scS_{s}$ and has the following symmetry type:
\begin{equation*}
\scB_{s}(\xi,\chi) = (-1)^{\binom{p}{2}+\frac{d}{4}(1+s(-1)^p)}\, \scB_{s}(\chi,\xi),
\end{equation*}

\noindent
where $s\in\mathbb{Z}_2$. Similarly, if $(\Sigma,\gamma)$ is of quaternionic type, we choose a quaternionic structure $J$ and a compatible admissible Hermitian pairing $\scS_{s}$ on $(\Sigma,\gamma)$ and define:
\begin{equation*}
\scB_{s}(\xi,\chi) := \scS_{s}(\xi,J(\chi)) 
\end{equation*}

\noindent
for every $\xi, \chi \in \Sigma$. Then it can be readily checked that $\scB_{s}$ is an admissible complex-bilinear pairing on $(\Sigma,\gamma)$, whose adjoint type is the same as that of $\scS_{s}$ and whose symmetry type is:
\begin{equation*}
\scB_{s}(\xi,\chi)  =  - (-1)^{\binom{p}{2}+\frac{d}{4}(1+s(-1)^p)}\, \scB_{s}(\chi,\xi)  
\end{equation*}

\noindent
for all $\xi, \chi \in \Sigma$.
\end{proof}

\begin{remark}
From the construction of the admissible complex-bilinear pairings $\scB_{+}$ and $\scB_{-}$ in terms of $\scS_{+}$ and $\scS_{-}$, it immediately follows that:
\begin{equation*}
\scB_-(\xi,\chi) = i^{\binom{p}{2}}i^{\binom{q+1}{2}}\scB_+(\Psi_\gamma(\nu)\xi,\chi)\, , \qquad \xi , \chi \in \Sigma,
\end{equation*}

\noindent
which is the same relation as in Proposition \ref{prop:relationhermitianproducts}.
\end{remark}

\noindent
We summarize in Table \ref{tab:symmetry_grid_singlereal_modified} and Table \ref{tab:symmetry_grid_singleq_modified} the symmetry properties of $\scB_s$, which depend on the \emph{modulo 4} value of $p$ and $d$.\medskip

\begin{table}
\centering
\caption{Symmetry type of $\scB_{s}$ when $p-q \equiv_8 0,2$}
\label{tab:symmetry_grid_singlereal_modified}
\renewcommand{\arraystretch}{1.3} 

\begin{tabular}{|c|c|c|c|}
\hline 
\textbf{$p$ mod 4} & \textbf{$d$ mod 4} & \textbf{Symmetry of $\scB_{+}$} & \textbf{Symmetry of $\scB_{-}$} \\ \hline 
\multirow{2}{*}{\textbf{0}} & \textbf{0} & Symmetric & Symmetric \\ \cline{2-4}
& \textbf{2} & Skew-symmetric & Symmetric \\ \hline
\multirow{2}{*}{\textbf{1}} & \textbf{0} & Symmetric & Symmetric \\ \cline{2-4}
& \textbf{2} & Symmetric & Skew-symmetric \\ \hline
\multirow{2}{*}{\textbf{2}} & \textbf{0} & Skew-symmetric & Skew-symmetric \\ \cline{2-4}
& \textbf{2} & Symmetric & Skew-symmetric \\ \hline
\multirow{2}{*}{\textbf{3}} & \textbf{0} & Skew-symmetric & Skew-symmetric \\ \cline{2-4}
& \textbf{2} & Skew-symmetric & Symmetric \\ \hline 
\end{tabular}
\end{table}

\begin{table}
\centering
\caption{Symmetry type of $\scB_{s}$ when $p-q \equiv_8 4,6$}
\label{tab:symmetry_grid_singleq_modified}
\renewcommand{\arraystretch}{1.3} 

\begin{tabular}{|c|c|c|c|}
\hline 
\textbf{$p$ mod 4} & \textbf{$d$ mod 4} & \textbf{Symmetry of $\scB_{+}$} & \textbf{Symmetry of $\scB_{-}$} \\ \hline 
\multirow{2}{*}{\textbf{0}} & \textbf{0} & Skew-symmetric & Skew-symmetric \\ \cline{2-4}
& \textbf{2} & Symmetric & Skew-symmetric \\ \hline
\multirow{2}{*}{\textbf{1}} & \textbf{0} & Skew-symmetric & Skew-symmetric \\ \cline{2-4}
& \textbf{2} & Skew-symmetric & Symmetric \\ \hline
\multirow{2}{*}{\textbf{2}} & \textbf{0} & Symmetric & Symmetric \\ \cline{2-4}
& \textbf{2} & Skew-symmetric & Symmetric \\ \hline
\multirow{2}{*}{\textbf{3}} & \textbf{0} & Symmetric & Symmetric \\ \cline{2-4}
& \textbf{2} & Symmetric & Skew-symmetric \\ \hline 
\end{tabular}
\end{table}

\noindent
In particular, if $p$ is odd, that is if $p \equiv_4 1,3$, the symmetry type of $\scB_{+}$ depends only on $p$ and is independent of $d$. On the other hand, if $p$ is even, that is if $p \equiv_4 0,2$, it is the symmetry of $\scB_{-}$ that only depends on $p$ and is independent of $d$. 

\begin{remark}
Note that the relations given in \eqref{eq:complex_bilinear_admissible_Bs} can be uniformly written as follows: 
\begin{eqnarray}
\label{eq:gamma_t}
\gamma(z)^t=\gamma((\pi^{\frac{1-s}{2}}\circ\tau)(z))    
\end{eqnarray}

\noindent
for all $z\in\C\mathrm{l}(V^*,h^*)$, where $\gamma(z)^t$ denotes the adjoint with respect to the admissible complex-bilinear pairing $\scB_s$ of adjoint type $s\in\Z_2$.
\end{remark}

\begin{lemma}
\label{lemma:exterior_form_t}
Let $\alpha\in\wedge V^*_\C$ and let $\scB_s$ be an admissible complex-bilinear pairing of adjoint type $s\in\Z_2$. Then:
\begin{equation*}
\Psi_\gamma(\alpha)^t=\Psi_\gamma((\pi^{\frac{1-s}{2}}\circ\tau)(\alpha)).
\end{equation*}
\end{lemma}

\begin{proof}
Note that $\pi\circ\Psi=\Psi\circ\pi$ and $\tau\circ\Psi=\Psi\circ\tau$. Using relation \eqref{eq:gamma_t} we get: 
\begin{eqnarray*}
\Psi_\gamma(\alpha)^t = \gamma(\Psi(\alpha))^t = \gamma((\pi^{\frac{1-s}{2}}\circ\tau)(\Psi(\alpha)))=\gamma(\Psi((\pi^{\frac{1-s}{2}}\circ\tau)(\alpha)))=\Psi_\gamma((\pi^{\frac{1-s}{2}}\circ\tau)(\alpha))   
\end{eqnarray*}

\noindent
for every $\alpha\in\wedge V^*_\C$.
\end{proof}

\begin{thm}
\label{thm:bilinearsquare}
Let $(\Sigma,\gamma)$ be an irreducible complex Clifford module equipped with an admissible complex-bilinear pairing $\scB_s$ of adjoint type $s\in \mathbb{Z}_2$ and symmetry type $\sigma\in \mathbb{Z}_2$. Then the following statements are equivalent for a complex exterior form $\alpha\in\wedge V^*_\C$: 
\begin{enumerate}
\item[\normalfont(a)] $\alpha$ is the complex-bilinear square of a spinor $\xi\in\Sigma$, that is, $\alpha=\cE_\gamma(\xi)$.
\item[\normalfont(b)] $\alpha$ satisfies the following relations: 
\begin{eqnarray*}
\alpha\diamond\alpha = 2^{\frac{d}{2}} \alpha^{(0)}\alpha\,, \qquad (\pi^{\frac{1-s}{2}}\circ\tau) (\alpha) = \sigma \alpha\, , \qquad \alpha \diamond \beta \diamond \alpha = 2^{\frac{d}{2}}(\alpha\diamond\beta)^{(0)}\alpha
\end{eqnarray*}

\noindent
for a fixed exterior form $\beta\in\wedge V^*_\C$ satisfying $(\alpha\diamond\beta)^{(0)} \neq 0$.
\item[\normalfont(c)] The following relations hold: 
\begin{equation*}
(\pi^{\frac{1-s}{2}}\circ\tau)(\alpha)=\sigma\alpha\,, \qquad \alpha \diamond \beta \diamond \alpha = 2^{\frac{d}{2}}(\alpha\diamond\beta)^{(0)}\alpha
\end{equation*}

\noindent
for every exterior form $\beta\in\wedge V^*_\C$.
\end{enumerate}
\end{thm}

\begin{proof}
Since $\Psi_\gamma\colon\wedge V^*_\C\to\End(\Sigma)$ is a unital isomorphism of algebras, $\alpha$ satisfies the conditions in (c) if and only if: 
\begin{eqnarray*}
E^t = E\quad\text{and}\quad E\circ A\circ E = \Tr(E\circ A) E \quad \forall A\in\End(\Sigma),
\end{eqnarray*}

\noindent
where $E:=\Psi_\gamma(\alpha)$ and $A:=\Psi_\gamma(\beta)$. The conclusion follows then by using Lemma \ref{lemma:exterior_form_t} together with the definition and properties of the Kähler-Atiyah trace and Proposition \ref{prop:algebraichbilinear}.
\end{proof}

\noindent
By the previous theorem, the collection of all Hermitian squares of irreducible complex spinors is a semi-algebraic subset of $\wedge V^{\ast}_{\C}$ depending exclusively on $s\in \mathbb{Z}_2$ and $\sigma \in \mathbb{Z}_2$. Let $(\Sigma,\gamma,\scB)$ be an irreducible complex-bilinear paired Clifford module. Denote by $\dot{\Sigma}_{\scB} \subset \Sigma$ the open set of non-vanishing norm spinors with respect to $\scB$. By Theorem \ref{thm:bilinearsquare} it follows that:
\begin{equation*}
\cE_\gamma(\dot{\Sigma}_{\scB}) = \left\{ \alpha\in \wedge V^{\ast}_{\C} \,\, \vert\,\, \alpha^{(0)} \neq 0\, ,  \,\, \alpha\diamond\alpha = 2^{\frac{d}{2}} \alpha^{(0)}\alpha\, , \,\, (\pi^{\frac{1-s}{2}}\circ\tau) (\alpha) = \sigma \alpha  \right\}.
\end{equation*}

\noindent
Hence, $\cE_\gamma(\dot{\Sigma}_{\scB})$ is a linear complex subspace of a complex cone in the open subset of $\wedge V^{\ast}_{\C}$ determined by the condition $\alpha^{(0)} \neq 0$.\medskip

\noindent
Similarly to the Hermitian square of a spinor considered in the previous subsection, for the complex-bilinear square of a spinor $\eta\in \Sigma$, we have: 
\begin{eqnarray}
\label{eq:explicitbilinear}
\alpha_{\eta} = \frac{1}{2^{\frac{d}{2}}} \scB(\eta,\eta) + \frac{1}{2^{\frac{d}{2}}} \sum_{k=1}^d \sum_{i_1 < \cdots < i_k} \scB((\gamma^{i_1 \cdots i_k})^{-1}\eta , \eta)\, e^{i_1}\wedge  \cdots \wedge e^{i_k}.
\end{eqnarray}

\noindent
This gives the explicit form of the complex-bilinear square of an irreducible complex spinor $\eta \in \Sigma$ in terms of the gamma matrices associated to an explicit orthonormal basis of $(V^{\ast} , h^{\ast})$. A direct computation shows that:
\begin{equation*}
(\alpha_\eta^-)^{(k)}=i^{\binom{p}{2}}i^{\binom{q+1}{2}}\frac{k!}{(d-k)!}*\tau((\alpha_\eta^+)^{(d-k)}),
\end{equation*}

\noindent
where $\alpha_{\eta}^{+}$ is complex-bilinear square of $\eta$ relative to $\scB_{+}$ and $\alpha_{\eta}^{-}$ is complex-bilinear square of $\eta$ relative to $\scB_{-}$. As in the Hermitian case, the choice of an admissible complex-bilinear pairing to construct the complex-bilinear square of a spinor becomes a matter of taste and computational convenience.  


\subsection{Equivariance properties of the complex-bilinear square map}
\label{subsec:Bilinearequivariance}


We consider now the equivariance properties of the complex-bilinear square map. Note that $\scB$ is invariant under $\Spin_o(p,q)$ but not $\Spin_o^c(p,q)$. We introduce the following morphism of groups:
\begin{eqnarray*}
\Ad^c\colon \Spin_o^c(p,q) \to \SO_o(p,q) \times \U(1) \, , \qquad [u,z]\mapsto \Ad^c_{[u,z]} := (\Ad_u,z^2),
\end{eqnarray*}

\noindent
where $\Ad\colon \Spin_o(p,q) \to \SO_o(p,q)$ is the standard double cover of $\Spin_o(p,q)$. Hence, the map $\Ad^c\colon \Spin_o^c(p,q) \to \SO_o(p,q) \times \U(1)$ is a double cover of $\SO_o(p,q) \times \U(1)$. For every $\eta\in \Sigma$ and $[u,z] \in \Spin_o^c(p,q)$, we compute:
\begin{eqnarray*}
\cE_\gamma(\gamma([u,z])\eta) = \cE_\gamma (z\gamma(u)\eta) = z^2 \cE_\gamma ( \gamma(u)\eta) = z^2\Psi_{\gamma}^{-1}(\scB( - , \gamma(u)\eta) \gamma(u)\eta) = \Ad^c_{[u,z]}(\cE_\gamma (\eta)),
\end{eqnarray*}

\noindent
where $\Ad^c_{[u,z]} = (\Ad_u , z^2)$ acts on $\wedge V^{\ast}_\C$ via the standard orthogonal action of $\Ad_u$ and multiplication by $z^2$. Hence, we conclude that $\cE_\gamma \colon \Sigma \to \wedge V^*_\C$ is equivariant with respect to the  \emph{double cover representation} $\Ad^c \colon \Spin_o^c(p,q) \to \SO_o(p,q)\times \U(1)$ and its natural action on $\wedge V^*_\C$. This implies in particular that the Lie algebra of the stabilizer of $\cE_\gamma(\eta)$ in $\SO_o(p,q)\times \U(1)$ is equal to the Lie algebra of the stabilizer of $\eta$ in $\Spin^c_o(p,q)$.


\subsection{The square of a complex chiral spinor}


Let $V$ be an oriented real vector space of even dimension $d=p+q$ equipped with a non-degenerate metric $h$ of signature $(p,q)$. Consider the complexified Clifford algebra $\C\mathrm{l}(V^*,h^*)$ associated to $(V^*,h^*)$ and let:
\begin{equation*}
\nu_\C:=i^{q+\frac{d}{2}}\nu\in\C\mathrm{l}(V^*,h^*)    
\end{equation*}

\noindent
be its pseudo-Riemannian complex volume form. It follows that $\nu_\C$ satisfies $\nu_\C^2=1$ and it lies in the center of the even Clifford algebra $\C\mathrm{l}^{\mathrm{ev}}(V^*,h^*)$. Therefore we can decompose $\C\mathrm{l}^{\mathrm{ev}}(V^*,h^*)$ into the $\pm1$-eigenspaces of $\nu_\C$ as:
\begin{equation*}
\C\mathrm{l}^{\mathrm{ev}}(V^*,h^*)=\C\mathrm{l}^{\mathrm{ev}}_+(V^*,h^*)\oplus\C\mathrm{l}^{\mathrm{ev}}_-(V^*,h^*)\, ,    
\end{equation*}

\noindent
where:
\begin{equation*}
\C\mathrm{l}^{\mathrm{ev}}_{\pm}(V^*,h^*)=\{\alpha\in\C\mathrm{l}^{\mathrm{ev}}(V^*,h^*)\mid\nu_\C\alpha=\pm\alpha\}=\frac{1}{2}(1\pm\nu_\C)\C\mathrm{l}^{\mathrm{ev}}(V^*,h^*)\, .
\end{equation*}

\noindent
We decompose the complex representation space $\Sigma$ accordingly as:
\begin{equation*}
\Sigma=\Sigma^+ \oplus \Sigma^-\, ,
\end{equation*}

\noindent
where:
\begin{equation*}
 \Sigma^\pm=\{\eta \in \Sigma\mid\gamma(\nu_\C)\xi=\pm\xi\}=\frac{1}{2}\big(\Id\pm\gamma(\nu_\C)\big)\Sigma\, .   
\end{equation*}

\noindent
The subspaces $\Sigma^\pm\subset\Sigma$ are preserved by the restriction of $\gamma$ to $\C\mathrm{l}^{\mathrm{ev}}(V^*,h^*)$. Hence, the restriction of $\gamma$ to $\C\mathrm{l}^{\mathrm{ev}}(V^*,h^*)$ decomposes as a sum of two irreducible representations:
\begin{equation*}
\gamma^\pm\colon\C\mathrm{l}^{\mathrm{ev}}(V^*,h^*)\to\End(\Sigma^\pm)    
\end{equation*}

\noindent
distinguished by the value they take on the complex volume form $\nu_\C \in \C \mathrm{l}^{\mathrm{ev}}(V^*,h^*)$, namely: 
\begin{equation*}
\gamma^\pm(\nu_\C)=\pm\Id\, .    
\end{equation*}
 
\noindent
A spinor $\eta\in\Sigma$ is called \emph{chiral of chirality $\mu\in\mathbb{Z}_2$} if it belongs to $\Sigma^\mu\subset\Sigma$. Let $\alpha\in\wedge V^*_\C$ be either the Hermitian or complex-bilinear square of the spinor $\eta\in\Sigma$. By Lemma \ref{lemma:constrainedspinoreven}, $\eta$ being chiral is equivalent to:
\begin{equation*}
 i^{q + \frac{d}{2}}\nu\diamond\alpha=\nu_\C\diamond\alpha=\mu\alpha,
\end{equation*}

\noindent
which, by Lemma \ref{lemma:product_volume_form}, is equivalent to: 
\begin{equation*}
i^{q + \frac{d}{2}}*(\pi\circ\tau)(\alpha)=\mu\alpha.
\end{equation*}

\noindent
Combining this result with Theorems \ref{thm:even_Hermitian_forms} and \ref{thm:bilinearsquare}, we obtain the algebraic characterization of both the Hermitian and complex-bilinear squares of a chiral irreducible complex spinor, as stated in the following corollaries.

\begin{cor}
\label{cor:sq_S_chiral}
Let $(\Sigma,\gamma)$ be an irreducible complex Clifford module equipped with an admissible Hermitian pairing $\scS_s$ of adjoint type $s \in \mathbb{Z}_2$. Then the following statements are equivalent for a complex exterior form $\alpha\in\wedge V^*_\C$, where $\mu\in\mathbb{Z}_2$ is a fixed chirality type: 
\begin{enumerate}
\item[\normalfont(a)] $\alpha$ is the Hermitian square of a chiral spinor with chirality $\mu$, for some $\kappa\in\U(1)$.
\item[\normalfont(b)] $\alpha$ satisfies the following relations: 
\begin{equation*}
\alpha\diamond\alpha=2^{\frac{d}{2}}\alpha^{(0)}\alpha\, , \quad (\pi^{\frac{1-s}{2}}\circ\tau)(\bar\kappa\alpha)=  \kappa \bar{\alpha} , \quad \alpha \diamond \beta \diamond\alpha=2^{\frac{d}{2}}(\alpha\diamond\beta)^{(0)}\alpha\, ,\quad i^{q+\frac{d}{2}} \ast (\pi\circ\tau)(\alpha)=\mu\alpha
\end{equation*}

\noindent
for a fixed exterior form $\beta\in\wedge V^*_\C$ satisfying $(\alpha \diamond \beta)^{(0)} \neq 0$.
\item[\normalfont(c)] The following relations hold: 
\begin{equation*}
(\pi^{\frac{1-s}{2}}\circ\tau)(\bar\kappa\alpha)=  \kappa \bar{\alpha} , \quad \alpha \diamond \beta \diamond\alpha=2^{\frac{d}{2}}(\alpha\diamond\beta)^{(0)}\alpha\, ,\quad i^{q+\frac{d}{2}} \ast (\pi\circ\tau)(\alpha)=\mu\alpha
\end{equation*}

\noindent
for every exterior form $\beta\in\wedge V^*_\C$.
\end{enumerate}
\end{cor}

\begin{cor}
\label{cor:sq_B_chiral}
Let $(\Sigma,\gamma)$ be an irreducible complex Clifford module equipped with an admissible complex-bilinear pairing $\scB_s$ of adjoint type $s\in\mathbb{Z}_2$ and symmetry type $\sigma\in\mathbb{Z}_2$. Then the following statements are equivalent for a complex exterior form $\alpha\in\wedge V^*_\C$, where $\mu\in\mathbb{Z}_2$ is a fixed chirality type: 
\begin{enumerate}
\item[\normalfont(a)] $\alpha$ is the complex-bilinear square of a chiral spinor with chirality $\mu$.
\item[\normalfont(b)] $\alpha$ satisfies the following relations: 
\begin{equation*}
\alpha\diamond\alpha=2^{\frac{d}{2}} \alpha^{(0)}\alpha\, , \quad (\pi^{\frac{1-s}{2}}\circ\tau)(\alpha) = \sigma \alpha\, , \quad \alpha \diamond \beta\diamond\alpha=2^{\frac{d}{2}}(\alpha\diamond\beta)^{(0)}\alpha\, ,\quad i^{q+\frac{d}{2}}\ast (\pi\circ\tau)(\alpha)=\mu\alpha
\end{equation*}

\noindent
for a fixed exterior form $\beta\in\wedge V^*_\C$ satisfying $(\alpha\diamond\beta)^{(0)} \neq 0$.
\item[\normalfont(c)] The following relations hold: 
\begin{equation*}
(\pi^{\frac{1-s}{2}}\circ\tau)(\alpha) = \sigma \alpha\, , \quad \alpha \diamond \beta\diamond\alpha=2^{\frac{d}{2}}(\alpha\diamond\beta)^{(0)}\alpha\, ,\quad i^{q+\frac{d}{2}}\ast (\pi\circ\tau)(\alpha)=\mu\alpha
\end{equation*}

\noindent
for every exterior form $\beta\in\wedge V^*_\C$.
\end{enumerate}
\end{cor}

\begin{example}
In later sections, we will consider several other examples, both in Euclidean and Lorentzian signatures. For the moment, take $(V,h)$ to be the four-dimensional Minkowski space. Let $(\Sigma , \gamma)$ be a complex irreducible module of $\Cl(V^{\ast},h^{\ast})$, which is four-dimensional and of real type. Denote by: 
\begin{eqnarray*}
\Sigma = \Sigma^{+} \oplus \Sigma^{-}
\end{eqnarray*}

\noindent
the chiral decomposition of $\Sigma$ with respect to the complex volume form $\nu_{\C} = -i\nu$, which in this case is the Lorentzian volume form $\nu$ of $(V,h)$. Following Proposition \ref{prop:Hermitian_admissible}, we equip $(\Sigma,\gamma)$ with a Hermitian pairing $\scS$ of positive adjoint type. Note that the chiral splitting of $\Sigma$ is orthogonal with respect to $\scS$. By Corollary \ref{cor:sq_S_chiral}, an exterior form $\alpha \in \wedge V^{\ast}_{\C}$ is the Hermitian square of a non-zero element $\eta\in\Sigma^{\mu}$ of chirality $\mu\in\mathbb{Z}_2$ with $\kappa = 1$ if and only if:
\begin{eqnarray}
\label{eq:sesquilinear4d}
 \tau(\alpha)=   \bar{\alpha}\, ,\qquad   \ast (\pi\circ\tau)(\alpha) = i \mu\alpha\, , \qquad \alpha \diamond \beta \diamond\alpha= 4 (\alpha\diamond\beta)^{(0)}\alpha
\end{eqnarray}

\noindent
for a complex exterior form $\beta\in \wedge V^{\ast}_{\C}$ such that $(\alpha \diamond\beta)^{(0)} \neq 0$. The first two equations, which are linear in $\alpha$, are directly solved by:
\begin{equation*}
\alpha = u + i \mu \ast u,
\end{equation*}

\noindent
where $u\in V^{\ast}$ is a non-zero one-form uniquely determined by the spinor $\eta$. Regarding the third equation in \eqref{eq:sesquilinear4d}, we begin by considering it first for $\beta =1$. We obtain:
\begin{equation*}
\alpha\diamond \alpha = (u + i \mu \ast u)\diamond (u + i \mu \ast u) = 0
\end{equation*}

\noindent
and thus it is satisfied automatically for $\alpha = u + i \mu \ast u$. Taking $\beta = u$, we obtain $(\alpha\diamond \beta)^{(0)}=\escal{u,u}$ and the third equation in \eqref{eq:sesquilinear4d} reduces to $\escal{u,u}=0$. Hence $u\in V^*$ is a nowhere vanishing isotropic one-form. Choose a one-form $v\in V^*$ conjugate to $u$. That is, $v$ is a nowhere vanishing isotropic one-form satisfying $\escal{u,v}=1$. Setting $\beta = v$ in the third equation of \eqref{eq:sesquilinear4d}, we first compute:
\begin{equation*}
v\diamond \alpha = v \diamond (u + i \mu \ast u) = 1 + v\wedge u + i\mu ( \nu + \ast(u\wedge v) ).
\end{equation*}

\noindent
Hence $(v\diamond \alpha)^{(0)} = 1$ and a direct computation shows the third equation in \eqref{eq:sesquilinear4d} is automatically satisfied when $\beta = v$. We have thus proven the following result. 

\begin{prop}
A complex exterior form $\alpha\in \wedge V^{\ast}_{\mathbb{C}}$ is the Hermitian square of an irreducible complex spinor $\eta \in \Sigma^{\mu}$ of chirality $\mu$ in four-dimensional Lorentzian signature if and only if:
\begin{equation*}
\alpha = \kappa (u + i \mu \ast u)
\end{equation*}

\noindent
for a unitary number $\kappa \in \U(1)$ and a non-zero isotropic one-form $u\in V^{\ast}$.
\end{prop}

\noindent
Hence, the Hermitian square of $\eta$ is given in terms of an isotropic one-form, typically called its \emph{Dirac current}. This illustrates that the Hermitian square of a spinor does not contain all information about the spinor, since, by the results of \cite{CLS21}, a real and irreducible spinor in four Lorentzian dimensions is not uniquely determined by an isotropic vector, in contrast to the situation occurring in three Lorentzian dimensions \cite{Sha24}.
\end{example}


\subsection{The square of a conjugate spinor}
\label{subsec:conjugateeven}


Let $(\Sigma,\gamma)$ be an irreducible complex Clifford module equipped with compatible Hermitian $\scS_s$ and complex-bilinear $\scB_s$ pairings, where $s\in \mathbb{Z}_2$ denotes their common adjoint type. Hence, there exists a unique complex anti-linear map $K\colon \Sigma \to \Sigma$ such that:
\begin{equation*}
\scS_s( \xi , \chi) = \scB_s ( \xi , K \chi)
\end{equation*}

\noindent
and such that $K^2 = \epsilon \Id$, with $\epsilon = 1$ if $(\Sigma,\gamma)$ is of real type and $\epsilon = -1$ if $(\Sigma,\gamma)$ is of quaternionic type. Hence, we are in the situation considered in Subsection \ref{sec:conjugatevector}. Mimicking Definition \ref{def:conjugatevector}, we introduce the notion of \emph{conjugate spinor}.

\begin{definition}
The \emph{conjugate} of the complex spinor $\eta\in \Sigma$ is $K\eta \in \Sigma$.
\end{definition}

\noindent
The explicit expressions given in equations \eqref{eq:explicitHermitian} and \eqref{eq:explicitbilinear} for the Hermitian and complex-bilinear squares of a complex spinor, together with the invariance properties of $\scS_s$ stated in Lemma \ref{lemma:compatible_admissible_pairings}, lead to the following result.

\begin{prop}
\label{prop:Hermitiansquareconjugate}
Let $\widehat{\alpha}_{\eta}$ and $\widehat{\alpha}_{K\eta}$ respectively denote the Hermitian squares of an irreducible complex spinor $\eta\in \Sigma$ and its conjugate $K\eta \in \Sigma$. Then:
\begin{equation*}
\widehat{\alpha}_{K\eta} = \tfrac{\kappa}{\bar\kappa}(-1)^{\binom{p}{2}+\frac{d}{4}(1+s(-1)^p)}\, \overline{\widehat{\alpha}}_{\eta},
\end{equation*}

\noindent
where $d$ is the dimension of $V$ and $p$ is the number of pluses in the signature of $h$.
\end{prop}

\noindent
We obtain a similar relation for the complex-bilinear squares of a $\eta$ and $K\eta$.

\begin{prop}
\label{prop:Bilinearsquareconjugate}
Let $\alpha_{\eta}$ and $\alpha_{K\eta}$ respectively denote the complex-bilinear squares of an irreducible complex spinor $\eta\in \Sigma$ and its conjugate $K\eta \in \Sigma$. Then:
\begin{equation*}
\alpha_{K\eta} = (-1)^{\binom{p}{2}+\frac{d}{4}(1+s(-1)^p)} \, \overline{\alpha}_{\eta},
\end{equation*}

\noindent
where $d$ is the dimension of $V$ and $p$ is the number of pluses in the signature of $h$.
\end{prop}

\begin{proof}
The result follows from the invariance properties of $\scB_s$, which as a consequence of the invariance properties of $\scS_s$. For every $\xi , \chi\in \Sigma$ we have:
\begin{equation*}
\scB_s (K\xi , K\chi) =  \scS_s (K\xi , \chi) = \epsilon (-1)^{\binom{p}{2}+\frac{d}{4}(1+s(-1)^p)} \, \overline{\scS_s (\xi , K\chi)} = (-1)^{\binom{p}{2}+\frac{d}{4}(1+s(-1)^p)} \, \overline{\scB_s (\xi , \chi)} 
\end{equation*}

\noindent
and hence we obtain the same sign as in the Hermitian case.
\end{proof}

\noindent
Hence, the Hermitian square of a conjugate complex spinor is, modulo a global phase, the complex conjugate of the Hermitian square of the given spinor, whereas the complex-bilinear square of a conjugate complex spinor is the complex conjugate of the complex-bilinear square of the given spinor, modulo a global sign.


\subsection{Compatibility of the Hermitian and complex-bilinear squares}
\label{subsec:compatibilitysquares}


In the previous subsections, we have characterized the complex-bilinear and Hermitian squares of complex irreducible spinors in even dimensions as complex exterior forms satisfying several algebraic constraints in $(\wedge V^{\ast}_{\C},\diamond)$. However, for the moment we have no criteria to guarantee that a pair of elements $\widehat{\alpha} \in \Im(\hcE)$ and $\alpha \in \Im(\cE)$ are the \emph{squares} of the \emph{same} spinor in $\Sigma$. Having such criteria is convenient for applications, since in applications we are interested in computing the complex-bilinear and Hermitian squares of the \emph{same} complex spinor in order to use them in combination. Below we obtain necessary and sufficient conditions for a pair of squares to be the square of the \emph{same} spinor in $\Sigma$. As in the previous subsection, we fix an irreducible complex Clifford module $(\Sigma,\gamma)$ equipped with compatible Hermitian $\scS_s$ and complex-bilinear $\scB_s$ pairings related by the anti-linear isomorphism $K\colon \Sigma \to \Sigma$.

\begin{prop} 
\label{prop:compatibilitysquares}
Let $\alpha_{\eta} , \alpha_{K\eta} \in \wedge V^{\ast}_{\C}$ respectively denote the complex-bilinear squares of $\eta\in \Sigma$ and its conjugate $K\eta\in \Sigma$. A complex exterior form $\widehat{\alpha} \in \wedge V^{\ast}_{\C}$ is the Hermitian square of $\eta$, namely $\widehat{\alpha} = \hcE^\kappa_{\gamma}(\eta)$, if and only if:
\begin{equation*}  
\widehat{\alpha} \diamond \beta \diamond \alpha_{\eta}  = 2^{\frac{d}{2}} (\widehat{\alpha} \diamond \beta)^{(0)} \alpha_{\eta} \, , \qquad   \alpha_{\eta} \diamond \beta \diamond \alpha_{K\eta} = 2^{\frac{d}{2}} \sigma \bar{\kappa}^{2} (\widehat{\alpha} \diamond (\pi^{\frac{1-s}{2}}\circ \tau)(\beta))^{(0)} \widehat{\alpha}
\end{equation*}

\noindent
for a complex exterior form $\beta \in \wedge V^{\ast}_{\C}$ such that $(\widehat{\alpha} \diamond \beta)^{(0)} \neq 0$.
\end{prop}

\begin{proof}
Follows by applying the isomorphism of unital and associative algebras $\Psi_{\gamma}^{-1}\colon \End(\Sigma) \to \wedge V^{\ast}_{\C}$ to Equation \eqref{eq:comp_sq_mapsII} in Proposition \ref{prop:endoquadraticbi}, together with Equation \eqref{eq:gamma_t} to compute the \emph{transpose} of $\beta \in \wedge V_{\C}^{\ast}$.
\end{proof}


\section{Spinorial forms in odd dimensions}
\label{section:odd_forms}


Let $V$ be an oriented real vector space of odd dimension $d=2m+1$ equipped with a non-degenerate metric of signature $(p,q)$. Let $(V^*,h^*)$ be the quadratic vector space dual to $(V,h)$ and denote by $\mathrm{Cl}(V^*,h^*)$ the Clifford algebra of $(V^*,h^*)$ with the convention $\theta^2=h^*(\theta,\theta)$ for all $\theta\in V^*$. Consider the complexification $\C\mathrm{l}(V^*,h^*) := \mathrm{Cl}(V^*,h^*) \otimes_\R\C$. Denoting by $\nu$ the pseudo-Riemannian volume form of $\mathrm{Cl}(V^*,h^*)$, we introduce the \emph{complex volume form} in $\C\mathrm{l}(V^*,h^*)$ as follows:
\begin{equation*}
\nu_\C:=i^{q+m}\nu\in\C\mathrm{l}(V^*,h^*).
\end{equation*}

\noindent
It can be easily seen that the complex volume form satisfies $\nu_\C^2=1$ and lies in the center of $\C\mathrm{l}(V^*,h^*)$. Therefore, $\C\mathrm{l}(V^*,h^*)$ is non-simple and splits as a direct sum of unital and associative algebras: 
\begin{equation}
\label{eq:splitting_Clifford}
\C\mathrm{l}(V^*,h^*)=\C\mathrm{l}_+(V^*,h^*)\oplus\C\mathrm{l}_-(V^*,h^*),
\end{equation}

where: 
\begin{equation*}
\C\mathrm{l}_\ell(V^*,h^*)=\{z\in\C\mathrm{l}(V^*,h^*)\mid\nu_\C z=\ell z\}=\frac{1}{2}(1+\ell\nu_\C)\C\mathrm{l}(V^*,h^*),\quad\ell\in\mathbb{Z}_2.
\end{equation*}

\noindent
The algebra $\C\mathrm{l}_\ell(V^*,h^*)$ is simple and isomorphic to $\Mat(2^m,\C)$, where $m=\frac{1}{2}(d-1)$. Therefore, $\C\mathrm{l}(V^*,h^*)$ admits two irreducible left Clifford modules of complex dimension $2^m$: 
\begin{equation*}
\gamma_\ell\colon\C\mathrm{l}(V^*,h^*)\to\End(\Sigma)   
\end{equation*}

\noindent
that correspond to the projection of $\C\mathrm{l}(V^*,h^*)$ to the factor $\C\mathrm{l}_\ell(V^*,h^*)$ composed with an isomorphism of the later to the algebra $\End(\Sigma)$ of endomorphisms of a complex vector space $\Sigma$ of dimension $2^m$. These two Clifford modules are distinguished by the value they take at the complex volume form $\nu_\C\in\C\mathrm{l}(V^*,h^*)$, namely: 
\begin{equation*}
\gamma_\ell(\nu_\C)=\ell\,\Id\in\End(\Sigma) .
\end{equation*}

\noindent
We transport the splitting \eqref{eq:splitting_Clifford} of $\C\mathrm{l}(V^*,h^*)$ to $(\wedge V^*_\C,\diamond)$ through the complexification of the Chevalley-Riesz isomorphism \eqref{eq:ChevRiesz_iso}. This gives: 
\begin{equation*}
(\wedge V^*_\C,\diamond)=(\wedge_+V^*_\C,\diamond)\oplus(\wedge_- V^*_\C,\diamond)    
\end{equation*}

\noindent
where:
\begin{equation*}
\wedge_\ell V^*_\C=\{\alpha\in\wedge V^*_\C\mid\nu_\C\diamond\alpha=\ell\alpha\}=\{\alpha\in\wedge V^*_\C\mid i^{q+m}*\tau(\alpha)=\ell\alpha\},\quad\ell\in\mathbb{Z}_2 .    
\end{equation*}

\noindent
Here we have used the identity: 
\begin{equation}
\label{eq:mult_volume_form_odd}
\alpha\diamond\nu_\C=\nu_\C\diamond\alpha=i^{q+m}*\tau(\alpha)
\end{equation} 

\noindent
given by Lemma \ref{lemma:product_volume_form}. By composing the Chevalley-Riesz isomorphism $\Psi\colon(\wedge V^*_\C,\diamond)\to\C\mathrm{l}(V^*,h^*)$ with the irreducible complex representation $\gamma_\ell\colon\C\mathrm{l}(V^*,h^*)\to\End(\Sigma)$ we obtain a surjective morphism of unital associative complex algebras that we denote by: 
\begin{equation*}
 \Psi_\ell:=\gamma_\ell\circ\Psi\colon(\wedge V^*_\C,\diamond)\to\End(\Sigma) .   
\end{equation*}

\noindent
Let $\cP_\ell\colon\wedge V^*_\C\to\wedge_\ell V^*_\C$ be the natural projection given explicitly by: \begin{equation}\label{eq:projection_l}
    \cP_\ell(\alpha)=\frac{1}{2}(1+\ell\nu_\C)\diamond\alpha=\frac{1}{2}(\alpha+i^{q+m}\ell*\tau(\alpha)),
\end{equation} where we have used the identity \eqref{eq:mult_volume_form_odd}. Consider the canonical linear inclusion $\iota_\ell\colon\wedge_\ell V_\C^*\hookrightarrow\wedge V_\C^*$, which is a right inverse to $\cP_\ell$. In particular: \begin{equation*}
    \Psi_\ell\circ\iota_\ell\colon(\wedge_\ell V^*_\C,\diamond)\to\End(\Sigma)
\end{equation*} is an isomorphism of unital and associative algebras. Following \cite{LBC13,LB13,LBC16}, we define: \begin{equation}\label{eq:vee_product}
    \wedge^<V^*_\C:=\bigoplus_{k=0}^m\wedge^kV^*_\C.
\end{equation}

Note that $\wedge V^*_\C=\wedge^<V_\C^*\oplus*\wedge^<V^*_\C$. Then, restricting $\cP_\ell\colon\wedge V^*_\C\to\wedge_\ell V^*_\C$ to $\wedge^<V^*_\C$, we obtain an isomorphism of vector spaces that we use to transport the algebra product in $(\wedge_\ell V^*_\C,\diamond)$ to $\wedge^<V^*_\C$. For every pair $\alpha,\beta\in\wedge^<V^*_\C$ we define:
\begin{equation*}
    \alpha\vee\beta:=2\cP_<(\cP_\ell(\alpha\diamond\beta)),
\end{equation*} where $\cP_<\colon\wedge V^*_\C\to\wedge^<V^*_\C$ is the natural projection. By construction, $(\wedge_\ell V^*_\C,\diamond)$ and $(\wedge^<V^*_\C,\vee)$ are naturally isomorphic as unital and associative complex algebras. For further reference we introduce the following linear map: 
\begin{equation}\label{eq:truncated_iso}
\Psi^<_\ell:=\Psi_\ell\circ\iota_\ell\circ\cP_\ell\vert_{\wedge^<V^*_\C}\colon(\wedge^<V^*_\C,\vee)\to\End(\Sigma),
\end{equation} 

\noindent
which, by the previous discussion, is an isomorphism of unital and associative algebras. All together, we obtain the following commutative diagram of unital and associative algebras:
$$\begin{tikzcd}
                                                                                                                                                        &  & {\mathbb{C}\mathrm{l}_\ell(V^*,h^*)} \arrow[d]                                                                                                                                   &  &                                                                                                                                           \\
{\mathbb{C}\mathrm{l}(V^*,h^*)} \arrow[rru, "\mathrm{pr}_\ell"] \arrow[rr, "\gamma_\ell"]                                                  &  & \mathrm{End}(\Sigma)                                                                                                                                                             &  & {(\wedge^<V^*_{\mathbb{C}},\vee)} \arrow[ll, "\Psi^<_\ell"'] \arrow[lldd, "\mathcal{P}_\ell|_{\wedge^<V^*_{\mathbb{C}}}"', shift right] \\
                                                                                                                                                        &  &                                                                                                                                                                                  &  &                                                                                                                                           \\
{(\wedge V^*_{\mathbb{C}},\diamond)} \arrow[uu, "\Psi"] \arrow[rruu, "\Psi_\ell"] \arrow[rr, "\mathcal{P}_\ell", shift left] &  & {(\wedge_\ell V^*_{\mathbb{C}},\diamond)} \arrow[uu, "\Psi_\ell\circ\iota_\ell"] \arrow[ll, "\iota_\ell", shift left] \arrow[rruu, "2\mathcal{P}_<"', shift right] &  &                                                                                                                                          
\end{tikzcd}$$

Through the isomorphism \eqref{eq:truncated_iso} the trace on $\End(\Sigma)$ transfers to the \emph{truncated Kähler-Atiyah algebra} $(\wedge^<V^*_\C,\vee)$, defining the \emph{truncated Kähler-Atiyah trace} $\mathcal{T}_\ell$, explicitly given by the following linear map: 
\begin{eqnarray*}
\mathcal{T}_\ell\colon\wedge^<V^*_\C\to\C,\quad\alpha\mapsto\Tr(\Psi^<_\ell(\alpha))\, .    
\end{eqnarray*}

\noindent 
Since $\Psi^<_\ell$ is a unital morphism of algebras, we have: 
\begin{eqnarray*}
\mathcal{T}_\ell(1) = 2^{\frac{d-1}{2}}  = 2^m \quad \text{and}\quad\mathcal{T}_\ell(\alpha_1\vee\alpha_2)=\mathcal{T}_\ell(\alpha_2\vee\alpha_1)\qquad\forall \,\, \alpha_1,\alpha_2\in\wedge^<V^*_\C,
\end{eqnarray*}

\noindent
where $1\in\wedge^0V^*_\C=\C$ is the identity of $(\wedge^<V^*_\C,\vee)$.
\begin{prop}
We have that $\mathcal{T}_\ell(\alpha)=2^{\frac{d-1}{2}}\alpha^{(0)}$ for all $\alpha\in\wedge^<V^*_\C$.
\end{prop}

\begin{proof}
    Let $\{e^1,\ldots,e^d\}$ be an orthonormal basis of $(V^*,h^*)$. For $i\neq j$ we have $e^i\vee e^j=-e^j\vee e^i$ and hence $(e^i)^{-1}\vee e^j\vee e^i=-e^j$. Let $1\leq k\leq\frac{d-1}{2}$ and $1\leq i_1<\cdots<i_k\leq d$. If $k$ is even, then: $$\mathcal{T}_\ell(e^{i_1}\vee\cdots\vee e^{e_k})=\mathcal{T}_\ell(e^{i_k}\vee e^{i_1}\vee\cdots\vee e^{i_{k-1}})=-\mathcal{T}_\ell(e^{i_1}\vee\cdots\vee e^{e_k})$$ and hence $\mathcal{T}_\ell(e^{i_1}\vee\cdots\vee e^{e_k})=0$. Here we have used the cyclicity of the truncated Kähler-Atiyah trace and the fact that $e^{i_k}$ anti-commutes with $e^{i_1},\ldots,e^{i_{k-1}}$. If $k$ is odd, let $j\in\{1,\ldots,d\}$ be such that $j\not\in\{i_1,\ldots,i_k\}$, which exists since $k<d$. We have: $$\mathcal{T}_\ell(e^{i_1}\vee\cdots\vee e^{i_k})=-\mathcal{T}_\ell((e^j)^{-1}\vee e^{i_1}\vee\cdots\vee e^{i_k}\vee e^j)=-\mathcal{T}_\ell(e^{i_1}\vee\cdots\vee e^{i_k}),$$ which implies that $\mathcal{T}_\ell(e^{i_1}\vee\dots\vee e^{i_k})=0$ and hence we conclude.
\end{proof}

Fix $\ell\in\mathbb{Z}_2$ and let $(\Sigma,\gamma_\ell)$ be an irreducible complex Clifford module equipped with a complex-bilinear pairing $\scB$ and a Hermitian pairing $\scS$. We define the \emph{complex-bilinear} and \emph{Hermitian} square spinor maps respectively as follows:
\begin{equation*}
\cE_\ell:=(\Psi_\ell^<)^{-1}\circ\cE\colon\Sigma\to\wedge^<V^*_\C,\qquad\hcE^\kappa_\ell:=(\Psi_\ell^<)^{-1}\circ\hcE_\kappa\colon\Sigma\to\wedge^<V^*_\C.
\end{equation*}

\noindent
Recall that, similarly to the even-dimensional case considered in the previous section, a Hermitian pairing $\scS\colon\Sigma\times\Sigma\to\C$ is called admissible if it satisfies:
\begin{equation}
\label{eq:adjoint_type}
\scS(\gamma_\ell(z)\xi,\chi)=\scS(\xi,\gamma_\ell(\overline{(\pi^{\frac{1-s}{2}}\circ\tau)(z)})\chi),\quad s\in\mathbb{Z}_2,
\end{equation} 

\noindent
whereas a complex-bilinear pairing $\scB\colon\Sigma\times\Sigma\to\C$  is called admissible if it satisfies:
\begin{equation*} 
\scB(\gamma_\ell(z)\xi,\chi)=\scB(\xi,\gamma_\ell((\pi^{\frac{1-s}{2}}\circ\tau)(z))\chi),\quad s\in\mathbb{Z}_2
\end{equation*} 

\noindent
for all $z\in\C\mathrm{l}(V^*,h^*)$ and $\xi,\chi\in\Sigma$.\medskip

As in the even-dimensional case considered in the previous section, we have the following characterization of \emph{constrained} spinors.

\begin{lemma}
\label{lemma:constrainedspinorodd}
Let $Q\in \End(\Sigma)$. Then, $\eta\in \Sigma$ satisfies $Q(\eta) = 0$ if and only if $\mathfrak{q}_\ell\vee \cE_{\ell}(\eta) = 0$, if and only if $\mathfrak{q}_\ell\vee \hcE^{\kappa}_{\ell}(\eta) = 0$, where $\mathfrak{q}_\ell:=(\Psi_\ell^<)^{-1}(Q)$ is the \emph{dequantization} of $Q\in \End(\Sigma)$.
\end{lemma}


\subsection{Hermitian spinorial forms}


In this subsection, we give the algebraic characterization of the complex spinorial forms associated to an irreducible complex Clifford module $(\Sigma,\gamma_\ell)$ equipped with an admissible Hermitian pairing $\scS$ in odd dimensions. We begin by proving the existence of admissible Hermitian pairings on every irreducible complex Clifford module $(\Sigma,\gamma_\ell)$ associated to an odd-dimensional quadratic vector space $(V,h)$.

\begin{prop}
\label{prop:Hermitian_admissibleodd}
Every irreducible complex Clifford module $(\Sigma,\gamma_\ell)$ admits a non-degenerate Hermitian pairing $\scS$ whose adjoint type is given in the following table:
\begin{table}[H]
\centering
\begin{tabular}{l|llll}
\hline
$\scS $&$p\equiv_40$&$p\equiv_41$&$p\equiv_42$&$p\equiv_43$\\
\hline
Adjoint type&Negative&Positive&Negative&Positive\\
\hline
\end{tabular}
\end{table}

\noindent
Here $p$ denotes the number of pluses in the signature of $h$.
\end{prop}

\begin{proof}
Let $\{e^1,\ldots,e^{d}\}$ be an $h^*$-orthonormal basis of $V^*$, $d=\dim_\R V$, and let: 
\begin{equation*}
\mathfrak{K} :=\{1\}\cup\{\pm e^{i_1}\cdots e^{i_k}\mid 1\leq i_1<\cdots<i_k\leq d,\,1\leq k\leq d\}
\end{equation*} 
be the finite multiplicative subgroup of $\C\mathrm{l}(V^*,h^*)$ generated by $\pm e^i$. Note that $\C\mathrm{l}(V^*,h^*)=\Span_\C\{ \mathfrak{K} \}$. Let $\beta\colon\Sigma\times\Sigma\to\C$ be a non-degenerate positive-definite Hermitian pairing. Then we can construct a $\mathfrak{K} $-invariant non-degenerate Hermitian pairing by averaging over $\mathfrak{K}$: 
\begin{equation*}
\escal{\xi,\chi}:=\frac{1}{\abs{\mathfrak{K}}}\sum_{k\in  \mathfrak{K}}\beta(\gamma_\ell(k)\xi,\gamma_\ell(k)\chi).
\end{equation*}
Then this pairing satisfies $\escal{\gamma_\ell(k)\xi,\gamma_\ell(k)\chi}=\escal{\xi,\chi}$ for all $k\in \mathfrak{K}$. Write $V^*=V^*_+\oplus V^*_-$, where $V^*_+$ is a $p$-dimensional subspace on which $h^*$ is positive-definite and $V^*_-$ is a $q$-dimensional subspace on which $h^*$ is negative-definite. Fix volume forms $\nu_+\in\wedge^pV^*_+$ and $\nu_-\in\wedge^qV^*_-$ so that $\nu = \nu_+\wedge\nu_-$ for the pseudo-Riemannian volume form $\nu$ on $(V,h)$. Define the non-degenerate Hermitian pairing $\scS \colon\Sigma\times\Sigma\to\C$ by:
\begin{eqnarray*}
\scS(\xi,\chi):= i^{\binom{p}{2}} \escal{\Psi^<_\ell(\nu_+)\xi,\chi}  \, .
\end{eqnarray*}

\noindent 
A direct computation shows that $\scS$ is a non-degenerate Hermitian pairing on $\Sigma$ and:
\begin{equation*}
\scS(\gamma_\ell(\theta)\xi,\chi)=(-1)^{p-1} \scS(\xi,\gamma_\ell (\theta)\chi)  
\end{equation*}

\noindent
for all $\theta\in V^*$, and hence we conclude. It can be seen that, if instead of $\Psi^<_\ell(\nu_{+})$, we use $\Psi^<_\ell(\nu_{-})$ together with a multiplicative factor of $i^{\binom{q+1}{2}}$ to construct $\scS$, we obtain the same Hermitian pairing modulo a multiplicative real constant.
\end{proof}

\noindent
The relation \eqref{eq:adjoint_type} can be written as: $$\gamma_\ell(z)^\dagger=\gamma_\ell(\overline{(\pi^{\frac{1-s}{2}}\circ\tau)(z)})$$ for all $z\in\C\mathrm{l}(V^*,h^*)$, where $\gamma_\ell(z)^\dagger$ denotes the adjoint with respect to the Hermitian pairing $\scS$ of adjoint type $s\in\mathbb{Z}_2$.\medskip

Arguing as in Theorem \ref{thm:even_Hermitian_forms}, we obtain the corresponding algebraic characterization of Hermitian spinorial exterior forms in odd dimensions.

\begin{thm}\label{thm_Hermitian_square_odd}
Let $(\Sigma,\gamma_\ell)$ be an irreducible complex Clifford module equipped with an admissible Hermitian pairing $\scS$ of adjoint type $s\in\mathbb{Z}_2$. Then the following statements are equivalent for a complex exterior form $\alpha\in\wedge^<V^*_\C$: 
\begin{enumerate}
\item[\normalfont(a)] $\alpha$ is the Hermitian square of a spinor $\xi\in\Sigma$, that is, $\alpha=\hcE_\ell^\kappa(\xi)$ for some $\kappa\in\U(1)$.
\item[\normalfont(b)] $\alpha$ satisfies the following relations: 
\begin{eqnarray*}
\alpha\vee\alpha = 2^m\alpha^{(0)}\alpha, \quad(\pi^{\frac{1-s}{2}}\circ\tau)(\bar\kappa\alpha) = \kappa \bar{\alpha}\, , \quad\alpha\vee\beta\vee\alpha=2^m(\alpha\vee\beta)^{(0)}\alpha
\end{eqnarray*}

\noindent
for a fixed exterior form $\beta\in\wedge^<V^*_\C$ satisfying $(\alpha\vee\beta)^{(0)}\neq0$, where $d = 2m + 1$.
\item[\normalfont(c)] The following relations hold: 
\begin{equation*}
(\pi^{\frac{1-s}{2}}\circ\tau)(\bar\kappa\alpha) = \kappa \bar{\alpha}\, , \quad\alpha\vee\beta\vee\alpha=2^m(\alpha\vee\beta)^{(0)}\alpha
\end{equation*}

\noindent
for every exterior form $\beta\in\wedge^<V^*_\C$.
\end{enumerate}
\end{thm}

\noindent
As it happened in the even-dimensional case considered in the previous section, we can explicitly expand the Hermitian square $\widehat{\alpha}_\eta\in\wedge^<V^*_\C$ of an irreducible complex spinor $\eta\in\Sigma$ in terms of an orthonormal basis of $(V^*,h^*)$. A direct computation gives:
\begin{equation}
\label{eq:oddhermitianexpansion}
 \widehat{\alpha}_{\eta} = \frac{\kappa}{2^{\frac{d-1}{2}}} \scS(\eta,\eta) + \frac{\kappa}{2^{\frac{d-1}{2}}} \sum_{k=1}^{\frac{d-1}{2}} \sum_{i_1 < \cdots < i_k} \scS((\gamma_\ell^{i_1 \cdots i_k})^{-1}\eta , \eta)\, e^{i_1}\wedge  \cdots \wedge e^{i_k},
 \end{equation}
where $\gamma_\ell^{i_1\cdots i_k}:=\gamma_\ell^{i_1}\cdots\gamma_\ell^{i_k}$ and $\gamma_\ell^i:=\Psi^<_\ell(e^i)$.

\begin{remark}
Similarly to the even-dimensional case considered in Subsection \ref{subsec:Hermitianequivariance}, the Hermitian spinor square map in odd dimensions is also equivariant with respect to the natural action of $\Spin^c_o (p,q)$.
\end{remark}


\subsection{Complex-bilinear spinorial forms}


In this subsection, we obtain the algebraic characterization of the complex spinorial forms associated to an irreducible complex Clifford module $(\Sigma,\gamma_\ell)$ equipped with an admissible complex-bilinear pairing $\scB$. We begin by proving the existence of admissible complex-bilinear pairings on every irreducible complex Clifford module $(\Sigma,\gamma_\ell)$. In order to do this, we need to distinguish between different signatures in odd dimensions. The case $p-q\equiv_81,5$ is analogous to the even-dimensional case discussed in Subsection \ref{subsec:bilinear}. We use the real or quaternionic structure of the complex Clifford module $(\Sigma,\gamma_\ell)$ for $\Cl(V^*,h^*)$ to construct the complex-bilinear admissible pairing on $(\Sigma,\gamma_\ell)$. Arguing as in Lemma \ref{lemma:real_or_quaternionic_type}, we obtain the following result.

\begin{lemma}\label{lemma:R_orQ_p-q=1,5}
Let $(\Sigma,\gamma_\ell)$ be an irreducible complex Clifford module for $\Cl(V^*,h^*)$ with $(V,h)$ odd-dimensional of signature $(p,q)$. Then:
\begin{itemize}
    \item If $p-q\equiv_8 1$ then $(\Sigma,\gamma_\ell)$ is of real type.
    \item If $p-q\equiv_8 5$ then $(\Sigma,\gamma_\ell)$ is of quaternionic type.
\end{itemize}
\end{lemma}

\noindent
We can always consider admissible Hermitian pairings on $(\Sigma,\gamma_\ell)$ which are either invariant or anti-invariant with respect to the underlying real or quaternionic structure of $(\Sigma,\gamma_\ell)$.

\begin{lemma}
\label{lemma:compatible_admissible_pairingsodd}
Let $(\Sigma,\gamma_\ell)$ be an irreducible complex Clifford module equipped with a real or quaternionic structure $K\colon \Sigma \to \Sigma$. Then $(\Sigma,\gamma_\ell)$ can be equipped with an admissible Hermitian pairing $\scS$ such that:
\begin{equation*}
\scS(K\xi,K\chi)=(-1)^{\binom{p}{2}}\overline{\scS(\xi,\chi)}
\end{equation*}
 for every $\xi , \chi \in \Sigma$.
 \end{lemma}

\noindent
The proof of this lemma is completely analogous to that of Lemma \ref{lemma:compatible_admissible_pairings}. Using such \emph{compatible} admissible Hermitian pairings and the real or quaternionic structure of $(\Sigma,\gamma_\ell)$, we proceed to construct natural admissible complex-bilinear pairings on $(\Sigma,\gamma_\ell)$.

\begin{prop}
Let $(V,h)$ be a quadratic vector space of signature $(p,q)$ satisfying $p-q\equiv_81,5$ and let $(\Sigma,\gamma_\ell)$ be an irreducible complex Clifford module for $\Cl(V^*,h^*)$. Then $(\Sigma,\gamma_\ell)$ admits a complex-bilinear admissible pairing $\scB$ whose adjoint type is given in the following table: 
\begin{table}[H]
\centering
\begin{tabular}{l|llll}
\hline
$\scB $&$p\equiv_40$&$p\equiv_41$&$p\equiv_42$&$p\equiv_43$\\
\hline
Adjoint type&Negative&Positive&Negative&Positive\\
\hline
\end{tabular}
\end{table}

\noindent
Here $p$ denotes the number of pluses in the signature of $h$.
\end{prop}

\begin{proof}
Let $(\Sigma,\gamma_\ell)$ be an irreducible complex Clifford module. If $(\Sigma,\gamma_\ell)$ is of real type, we choose a real structure $\frc$ and a compatible admissible Hermitian pairing $\scS$ on $(\Sigma,\gamma_\ell)$ and define:
\begin{equation*}
\scB(\xi,\chi) := \scS(\xi,\frc(\chi)) \qquad \forall\,\, \xi, \chi \in \Sigma.
\end{equation*}

\noindent
Then $\scB$ has the same adjoint type as $\scS$ and has the following symmetry type:
\begin{equation*}
\scB(\xi,\chi) = (-1)^{\binom{p}{2}}\, \scB(\chi,\xi).
\end{equation*}

\noindent
Similarly, if $(\Sigma,\gamma_\ell)$ is of quaternionic type, we choose a quaternionic structure $J$ and a compatible admissible Hermitian pairing $\scS$ on $(\Sigma,\gamma_\ell)$ and define:
\begin{equation*}
\scB(\xi,\chi) := \scS(\xi,J(\chi)) \qquad \forall\,\, \xi, \chi \in \Sigma.
\end{equation*}

\noindent
Then it can be readily checked that $\scB$ is an admissible complex-bilinear pairing on $(\Sigma,\gamma_\ell)$, whose adjoint type is the same as that of $\scS$ and whose symmetry type is given as follows:
\begin{equation*}
\scB(\xi,\chi)  =  - (-1)^{\binom{p}{2}}\, \scB(\chi,\xi).
\end{equation*}
\end{proof}

\noindent
The symmetry and adjoint type of $\scB$ are summarized in Table \ref{tab:symmetry_parity1} and Table \ref{tab:symmetry_parity5}.\medskip

\begin{table}
\centering
\caption{Properties of $\scB$ when $p-q\equiv_8 1$}
\label{tab:symmetry_parity1}
\renewcommand{\arraystretch}{1.3} 
\begin{tabular}{|c|c|c|}
\hline
\textbf{$p$ mod 4} & \textbf{Adjoint type} & \textbf{Symmetry type} \\ \hline
0 & Negative & Symmetric \\ \hline
1 & Positive & Symmetric \\ \hline
2 & Negative & Skew-symmetric \\ \hline
3 & Positive & Skew-symmetric \\ \hline
\end{tabular}
\end{table}

\begin{table}
\centering
\caption{Properties of $\scB$ when $p-q\equiv_8 5$}
\label{tab:symmetry_parity5}
\renewcommand{\arraystretch}{1.3} 
\begin{tabular}{|c|c|c|}
\hline
\textbf{$p$ mod 4} & \textbf{Adjoint type} & \textbf{Symmetry type} \\ \hline
0 & Negative & Skew-symmetric \\ \hline
1 & Positive & Skew-symmetric \\ \hline
2 & Negative & Symmetric \\ \hline
3 & Positive & Symmetric \\ \hline
\end{tabular}
\end{table}

\noindent
In the case of signature $p-q\equiv_8 3,7$ we also make use of an auxiliary complex anti-linear endomorphism of $\Sigma$ to construct the admissible complex-bilinear pairings. This endomorphism is, however, not of the same type as the real and quaternionic structures considered earlier, since it does not belong to the Schur algebra of the underlying Clifford module. The next result follows from \cite[Proposition 4.4]{LS22_complex_type}.

\begin{prop}
\label{prop:D_endomorphism}
Let $(V,h)$ be a quadratic vector space of signature $(p,q)$ satisfying $p-q\equiv_83,7$ and let $(\Sigma,\gamma_\ell)$ be an irreducible complex Clifford module for $\Cl(V^*,h^*)$. Then there exists a complex anti-linear endomorphism $D\colon \Sigma \to \Sigma$ such that: \begin{itemize}
\item $\gamma_\ell(z)\circ D=D\circ\gamma_\ell(\overline{\pi(z)})$ for all $z\in\C\mathrm{l}(V^*,h^*)$.
\item $D^2=\epsilon\,\Id$, where $\epsilon=-1$ if $p-q\equiv_83$ and $\epsilon = 1$ if $p-q\equiv_87$.
\end{itemize}
\end{prop}

\noindent
Given an irreducible complex Clifford module $(\Sigma,\gamma_\ell)$ in signature $p-q \equiv_8 3,7$, for every such $D\colon \Sigma \to \Sigma$ as in Proposition \ref{prop:D_endomorphism} there exists a compatible admissible Hermitian pairing $\scS$.

\begin{lemma}
\label{lemma:compatible_admissible_pairingsD}
Let $(\Sigma,\gamma_\ell)$ be an irreducible complex Clifford module in signature $p-q \equiv_8 3,7$. Then $(\Sigma,\gamma_\ell)$ can be equipped with an admissible Hermitian pairing $\scS$ such that:
\begin{equation*}
\scS(D \xi , D \chi)=(-1)^{\binom{p+1}{2}}\overline{\scS(\xi,\chi)} 
\end{equation*}
 for every $\xi , \chi \in \Sigma$.
\end{lemma}

\noindent
Using the endomorphism $D$ and a compatible admissible Hermitian pairing $\scS$, we proceed to construct an admissible complex-bilinear pairing in the following proposition.

\begin{prop}
\label{prop:B_D_admissible}
Let $(V,h)$ be a quadratic vector space of signature $(p,q)$ satisfying $p-q\equiv_83,7$ and let $(\Sigma,\gamma_\ell)$ be an irreducible complex Clifford module for $\Cl(V^*,h^*)$. Then $(\Sigma,\gamma_\ell)$ admits a complex-bilinear admissible pairing $\scB$ whose adjoint type is given in the following table:  
\begin{table}[H]
\centering
\begin{tabular}{l|llll}
\hline
$\scB $&$p\equiv_40$&$p\equiv_41$&$p\equiv_42$&$p\equiv_43$\\
\hline
Adjoint type & Positive & Negative & Positive & Negative\\
\hline
\end{tabular}
\end{table}

\noindent
Here $p$ denotes the number of pluses in the signature of the metric $h$.
\end{prop}

\begin{proof}
Choose an admissible Hermitian pairing $\scS$ on $(\Sigma,\gamma_\ell)$ compatible with a choice of complex anti-linear isomorphism $D\colon \Sigma \to \Sigma$ and define:
\begin{equation*}
\scB(\xi,\chi) := \scS(\xi, D\chi)
\end{equation*}

\noindent
for every $\xi, \chi \in \Sigma$. If $p-q\equiv_83$, then $\scB$ has the following symmetry type:
\begin{equation*}
\scB(\xi,\chi) = -(-1)^{\binom{p+1}{2}}\, \scB(\chi,\xi).
\end{equation*}

\noindent
Similarly, if $p-q\equiv_8 7$, then $\scB$ has the following symmetry type:
\begin{equation*}
\scB(\xi,\chi) = (-1)^{\binom{p+1}{2}}\, \scB(\chi,\xi).
\end{equation*}

\noindent
The table of the proposition follows by noticing that the adjoint type of the complex-bilinear pairing $\scB$ is opposite to that of $\scS$.
\end{proof}

\noindent
The symmetry and adjoint type of $\scB$ are summarized in Table \ref{tab:symmetry_parity3} and Table \ref{tab:symmetry_parity7}.

\begin{table}
\centering
\caption{Properties of $\scB$ when $p-q\equiv_8 3$}
\label{tab:symmetry_parity3}
\renewcommand{\arraystretch}{1.3} 
\begin{tabular}{|c|c|c|}
\hline
\textbf{$p$ mod 4} & \textbf{Adjoint type} & \textbf{Symmetry type} \\ \hline
0 & Positive & Skew-symmetric \\ \hline
1 & Negative & Symmetric \\ \hline
2 & Positive & Symmetric \\ \hline
3 & Negative & Skew-symmetric \\ \hline
\end{tabular}
\end{table}

\begin{table}
\centering
\caption{Properties of $\scB$ when $p-q\equiv_8 7$}
\label{tab:symmetry_parity7}
\renewcommand{\arraystretch}{1.3} 
\begin{tabular}{|c|c|c|}
\hline
\textbf{$p$ mod 4} & \textbf{Adjoint type} & \textbf{Symmetry type} \\ \hline
0 & Positive & Symmetric \\ \hline
1 & Negative & Skew-symmetric \\ \hline
2 & Positive & Skew-symmetric \\ \hline
3 & Negative & Symmetric \\ \hline
\end{tabular}
\end{table}

\begin{thm}
\label{thm:odd_complex-bilinear_square}
Let $(\Sigma,\gamma_\ell)$ be an irreducible complex Clifford module equipped with an admissible complex-bilinear pairing $\scB$ of adjoint type $s\in\mathbb{Z}_2$ and symmetry type $\sigma\in\mathbb{Z}_2$. Then the following statements are equivalent for a complex exterior form $\alpha\in\wedge^<V^*_\C$: 
\begin{enumerate}
\item[\normalfont(a)] $\alpha$ is the complex-bilinear square of a spinor $\xi\in\Sigma$, that is, $\alpha=\cE_\ell(\xi)$.
\item[\normalfont(b)] $\alpha$ satisfies the following relations: $$\alpha\vee\alpha=2^m\alpha^{(0)}\alpha,\quad(\pi^{\frac{1-s}{2}}\circ\tau)(\alpha)=\sigma\alpha,\quad\alpha\vee\beta\vee\alpha=2^m(\alpha\vee\beta)^{(0)}\alpha$$ for a fixed exterior form $\beta\in\wedge^<V^*_\C$ satisfying $(\alpha\vee\beta)^{(0)}\neq0$, where $m=\frac{1}{2}(p+q-1)$.
\item[\normalfont(c)] The following relations hold: $$(\pi^{\frac{1-s}{2}}\circ\tau)(\alpha)=\sigma\alpha,\quad\alpha\vee\beta\vee\alpha=2^m(\alpha\vee\beta)^{(0)}\alpha$$ for every exterior form $\beta\in\wedge^<V^*_\C$.
\end{enumerate}
\end{thm}

\noindent
Analogously to the Hermitian case, the complex-bilinear square $\alpha_\eta\in\wedge^<V^*_\C$ of an irreducible complex spinor $\eta\in\Sigma$ can be expanded as follows:  
\begin{equation}
\label{eq:odd_explicit_complexbilinearsquare}
\alpha_{\eta} = \frac{1}{2^{\frac{d-1}{2}}} \scB(\eta,\eta) + \frac{1}{2^{\frac{d-1}{2}}} \sum_{k=1}^{\frac{d-1}{2}} \sum_{i_1 < \cdots < i_k} \scB((\gamma_\ell^{i_1 \cdots i_k})^{-1}\eta , \eta)\, e^{i_1}\wedge  \cdots \wedge e^{i_k}
\end{equation}

\noindent
in terms of any orthonormal basis $\{e^1,\hdots , e^d\}$ of $(V^*,h^*)$, where $\gamma^i = \Psi_{\gamma}^{<}(e^i)$.

\begin{remark}
Similarly to the even-dimensional case considered in Subsection \ref{subsec:Bilinearequivariance}, the complex-bilinear spinor square map in odd dimensions is also equivariant with respect to the double cover homomorphism $\Spin^c_o (p,q) \to \SO_o (p,q) \times \U(1) $.
\end{remark}


\subsection{The square of a conjugate spinor}


Let $(V,h)$ be an odd-dimensional quadratic vector space of signature $(p,q)$ and let $(\Sigma,\gamma_\ell)$ be an irreducible complex Clifford module for $\Cl(V^*,h^*)$ equipped with a Hermitian pairing $\scS$ of adjoint type $s_\scS$ and a complex-bilinear pairing $\scB$ of adjoint type $s_\scB$ and symmetry type $\sigma$. We assume that these pairings are compatible, i.e.\ there exists a unique complex anti-linear map $K\colon \Sigma \to \Sigma$ such that:
\begin{equation*}
\scS( \xi , \chi) = \scB( \xi , K \chi)
\end{equation*}

\noindent
and such that $K^2 = \epsilon \Id$, with $\epsilon=1$ if $p-q\equiv_81,7$ and $\epsilon = -1$ if $p-q\equiv_83,5$. Recall that $s_\scB=s_\scS$ if $p-q\equiv_81,5$ and  $s_\scB=-s_\scS$ if $p-q\equiv_83,7$.\medskip

As in the even-dimensional case, see Subsection \ref{subsec:conjugateeven}, we define the \emph{conjugate} of a complex spinor $\eta\in\Sigma$ as $K\eta\in\Sigma$. The explicit expressions given in equations \eqref{eq:oddhermitianexpansion} and \eqref{eq:odd_explicit_complexbilinearsquare} for the Hermitian and complex-bilinear squares of a complex spinor, together with the invariance properties of $\scS$ and $\scB$, lead to the following result, which is proven analogously as in the even-dimensional case.

\begin{prop}
\label{prop:Hermitiansquareconjugate_odd}
Let $\widehat{\alpha}_{\eta}$ and $\widehat{\alpha}_{K\eta}$ respectively denote the Hermitian squares of an irreducible complex spinor $\eta\in \Sigma$ and its conjugate $K\eta \in \Sigma$. 
\begin{itemize}
\item If $p-q \equiv_8 1, 5$ then $\widehat{\alpha}_{K\eta} = \tfrac{\kappa}{\bar\kappa}(-1)^{\binom{p}{2}} \, \overline{\widehat{\alpha}}_{\eta}$.
 
\item If $p-q \equiv_8 3, 7$ then $\widehat{\alpha}_{K\eta} = \tfrac{\kappa}{\bar\kappa}(-1)^{\binom{p + 1}{2}} \, \pi(\overline{\widehat{\alpha}}_{\eta})$.
\end{itemize}

\noindent
Here $p$ is the number of pluses in the signature of $h$.
\end{prop}

\begin{prop}
\label{prop:Bilinearsquareconjugate_odd}
Let $\alpha_{\eta}$ and $\alpha_{K\eta}$ respectively denote the complex-bilinear squares of an irreducible complex spinor $\eta\in \Sigma$ and its conjugate $K\eta \in \Sigma$.  
\begin{itemize}
\item If $p-q \equiv_8 1, 5$ then $\alpha_{K\eta} = (-1)^{\binom{p}{2}} \,\overline{\alpha}_{\eta}$.
 
\item If $p-q \equiv_8 3, 7$ then $\alpha_{K\eta} = (-1)^{\binom{p + 1}{2}} \, \pi(\overline{\alpha}_{\eta})$.
\end{itemize}

\noindent
Here $p$ is the number of pluses in the signature of $h$.
\end{prop} 

\noindent
Hence, the Hermitian square of a conjugate complex spinor is, modulo a global phase, the complex conjugate of the Hermitian square of the given spinor, whereas the complex-bilinear square of a conjugate complex spinor is the complex conjugate of the complex-bilinear square of the given spinor, modulo a global sign.


\subsection{Compatibility of the Hermitian and complex-bilinear squares}


We fix an irreducible complex Clifford module $(\Sigma,\gamma_\ell)$ equipped with compatible Hermitian $\scS$ and complex-bilinear $\scB$ pairings related by the anti-linear isomorphism $K\colon \Sigma \to \Sigma$. The analog to Proposition \ref{prop:compatibilitysquares} is the following result.

\begin{prop}
\label{prop:compatibilitysquares_odd}
Let $\alpha_{\eta} , \alpha_{K\eta} \in \wedge^<V^{\ast}_{\C}$ respectively denote the complex-bilinear squares of $\eta\in \Sigma$ and its conjugate $K\eta\in \Sigma$. A complex exterior form $\widehat{\alpha} \in \wedge^<V^{\ast}_{\C}$ is the Hermitian square of $\eta$ if and only if:
\begin{equation*}  
\widehat{\alpha} \vee \beta \vee \alpha_{\eta}  = 2^{\frac{d-1}{2}} (\widehat{\alpha} \vee \beta)^{(0)} \alpha_{\eta} \, , \qquad   \alpha_{\eta} \vee \beta \vee \alpha_{K\eta} = 2^{\frac{d-1}{2}} \sigma \bar{\kappa}^{2} (\widehat{\alpha} \vee (\pi^{\frac{1-s_\scB}{2}}\circ \tau)(\beta))^{(0)}  \widehat{\alpha}
\end{equation*}

\noindent
for a complex exterior form $\beta \in \wedge^<V^{\ast}_{\C}$ such that $(\widehat{\alpha} \vee \beta)^{(0)} \neq 0$.
\end{prop}


\section{Irreducible spinor squares in low dimensions}
\label{sec:lowdimensions}


In this section, we apply the formalism developed in the previous sections to completely characterize the complex-bilinear and Hermitian squares of an irreducible complex spinor and its conjugate in the case where $(V,h)$ is an Euclidean vector space of dimension $d=2,\ldots,6$. Furthermore, we will obtain the algebraic conditions arising from imposing that $\eta$ lies in the kernel of a given complex two-form or three-form. These conditions will be interpreted in the introduction as \emph{instanton conditions} on a principal-bundle connection and a bundle gerbe curving, respectively.


\subsection{The square of an irreducible chiral complex spinor in dimension two}


Let $(V,h)$ be a two-dimensional Euclidean vector space. The Clifford algebra $\mathrm{Cl}(V^*,h^*)$ is isomorphic to $\Mat(2,\R)$ and the irreducible complex Clifford module $(\Sigma,\gamma)$ is of real type by Lemma \ref{lemma:real_or_quaternionic_type}. We endow $\Sigma$ with compatible admissible pairings $\scS$ and $\scB$ of positive symmetry and adjoint type, which are related by a compatible real structure $\mathfrak{c}\colon\Sigma\to\Sigma$.


\subsubsection{The Hermitian square}


We proceed to compute the Hermitian square of $\eta$. Set $\kappa=1$ for simplicity. By Corollary \ref{cor:sq_S_chiral}, a complex exterior form $\widehat{\alpha}\in\wedge V^*_\C$ is the Hermitian square of a spinor $\eta\in\Sigma$ with chirality $\mu\in\Z_2$ if and only if: 
\begin{eqnarray*}
\widehat{\alpha} \diamond \widehat{\alpha} = 2\widehat{\alpha}^{(0)} \widehat{\alpha}\, , \quad \tau(\widehat{\alpha}) = \overline{\widehat{\alpha}}\, , \quad i*(\pi\circ\tau)(\widehat{\alpha}) = \mu \widehat{\alpha}.
\end{eqnarray*}

\noindent
Let $\widehat{\alpha} = \widehat{\alpha}^{(0)} + \widehat{\alpha}^{(1)} + \widehat{\alpha}^{(2)}$. The linear condition $\tau(\widehat{\alpha})=\overline{\widehat{\alpha}}$ implies that $\widehat{\alpha}^{(0)}$ and $\widehat{\alpha}^{(1)}$ are real and that $\widehat{\alpha}^{(2)}$ is imaginary. The other linear condition $i*(\pi\circ\tau)(\widehat{\alpha})=\mu\widehat{\alpha}$ implies that $i\ast\widehat{\alpha}^{(0)}=\mu\widehat{\alpha}^{(2)}$ and $\widehat{\alpha}^{(1)}=0$. Set $r:=\widehat{\alpha}^{(0)}\in\R$. Hence $\widehat{\alpha} = r+i\mu r\nu$ and the quadratic equation $\widehat{\alpha}\diamond\widehat{\alpha} = 2\widehat{\alpha}^{(0)} \widehat{\alpha}$ automatically holds.

\begin{cor}
A complex exterior form $\widehat{\alpha} \in\wedge V^*_\C$ is the Hermitian square of a complex irreducible spinor $\eta\in\Sigma^{\mu}$ with chirality $\mu$ if and only if: $$\widehat{\alpha} = r+i\mu r\nu$$ for a non-zero real number $r\in\R\setminus\{0\}$.
\end{cor}

Given an irreducible complex spinor $\eta\in \Sigma$, consider now the following equation:
\begin{eqnarray*}
F\cdot \eta := \Psi_{\gamma}(F)(\eta) = 0
\end{eqnarray*}

\noindent
for a complex two-form $F\in \wedge^2 V^{\ast}_{\mathbb{C}}$. This equation \emph{models} the spinorial instanton condition described in the introduction. By Lemma \ref{lemma:constrainedspinoreven}, a two-form $F$ satisfies the previous equation if and only if:
\begin{eqnarray*}
F\diamond \widehat{\alpha} = F\diamond (r+i\mu r\nu) = 0,
\end{eqnarray*}

\noindent
where $\widehat{\alpha} = r+i\mu r\nu$ is the Hermitian square of $\eta\in \Sigma$. This equation is equivalent to $F = 0$, and thus we conclude that chiral spinorial instantons on an oriented Riemann surface are precisely flat connections, as expected.
 

\subsubsection{The complex-bilinear square}


Let $\mu\in\mathbb{Z}_2$. By Corollary \ref{cor:sq_B_chiral}, an exterior form $\alpha\in\wedge V^*_\C$ is the complex-bilinear square of a spinor $\eta \in\Sigma$ with chirality $\mu$ if and only if:
\begin{eqnarray*}
\alpha\diamond\alpha=2\alpha^{(0)}\alpha,\quad\tau(\alpha)=\alpha,\quad\alpha\diamond\beta\diamond\alpha=2(\alpha\diamond\beta)^{(0)}\alpha,\quad i*(\pi\circ\tau)(\alpha)=\mu\alpha    
\end{eqnarray*}

\noindent
for a fixed exterior form $\beta\in\wedge V^*_\C$ satisfying $(\alpha\diamond\beta)^{(0)}\neq0$. Let $\alpha = \alpha^{(0)} + \alpha^{(1)} +\alpha^{(2)}$, where $\alpha^{(k)}\in\wedge^kV^*_\C$. The linear condition $\tau(\alpha)=\alpha$ implies that $\alpha^{(2)}=0$ and the \emph{chirality} linear condition $i*(\pi\circ\tau)(\alpha)=\mu\alpha$ implies that $\alpha^{(0)}=0$ and $*\alpha^{(1)}=i\mu\alpha^{(1)}$. Set $\theta:=\alpha^{(1)}$. Equation $\alpha\diamond\alpha=2\alpha^{(0)}\alpha$ reduces to:
\begin{eqnarray*}
\theta\diamond\theta=\theta\wedge\theta+\escal{\theta,\theta}=0   
\end{eqnarray*}

\noindent
and thus $\escal{\theta,\theta}=0$. Note that the condition $*\theta=i\mu\theta$ implies that $\escal{\theta,\theta}=0$ since the Hodge star operator is an isometry. Since $\alpha^{(0)}=0$, taking $\beta=1$ does not suffice to characterize the square of the spinor. Let $\beta=\bar\theta\in V^*_\C$. Then $\alpha \diamond \beta = \theta \diamond \bar\theta = \theta \wedge \bar \theta +\escal{\theta,\bar\theta}$ and $(\alpha \diamond \beta)^{(0)} =\escal{\theta,\bar\theta} = \abs{\Re(\theta)}^2 + \abs{\Im(\theta)}^2 \neq0$ since $\theta\neq 0$. Using $\theta\diamond\bar\theta+\bar\theta\diamond\theta=2\escal{\theta,\bar\theta}$ and $\theta\diamond\theta=0$ we compute:
\begin{eqnarray*}
\alpha\diamond\beta\diamond\alpha=\theta\diamond\bar\theta\diamond\theta=(-\bar\theta\diamond\theta+2\escal{\theta,\bar\theta})\diamond\theta=2\escal{\theta,\bar\theta}\theta=2(\alpha\diamond\beta)^{(0)}\alpha.
\end{eqnarray*}
 
\noindent
Therefore, we conclude:

\begin{cor}
A complex exterior form $\alpha\in\wedge V^*_\C$ is the complex-bilinear square of a complex irreducible spinor $\eta\in\Sigma^{\mu}$ with chirality $\mu$ if and only if: $$\alpha=\theta$$ for a uniquely determined complex one-form $\theta\in V^*_\C$ satisfying $*\theta=i\mu\theta$.
\end{cor}

\noindent
Let $\theta=\theta^R+i\theta^I$ and take $\mu =1$ for definiteness. Since $\escal{\theta,\bar\theta}\neq 0$ and $*\theta=i \theta$, then $\theta \wedge \bar\theta \neq 0$, which implies that $\theta^R,\theta^I$ are linearly independent real one-forms. We thus have $V^*=\Span_\R\{\theta^R,\theta^I\}$. Note that equation $*\theta=i\theta$ implies $*\theta^I = \theta^R$. We define the complex structure $I$ on $V$ by $I(u^I):= u^R$, where $u^I:=(\theta^I)^\sharp$, $u^R:=(\theta^R)^\sharp$. One can check that $I$ is $h$-orthogonal and therefore $(V,h,I)$ becomes a Hermitian vector space. Furthermore, the complex one-form $\theta$ is of type $(1,0)$ with respect to the complex structure $I$. We conclude that there is a one-to-one correspondence between irreducible chiral spinors on $(V,h)$ and triples $(h,I,\theta)$ consisting of an Euclidean metric $h$, a complex structure $I$ and a \emph{holomorphic} volume form $\theta \in \wedge^{1,0}_I V^{\ast}$, where $V^{\ast} \otimes \C = \wedge^{1,0}_I V^{\ast} \oplus \wedge^{0,1}_I V^{\ast}$ is decomposed with respect to $I$. This is of course, nothing but the simplest manifestation of the correspondence between \emph{pure} irreducible and chiral complex spinors and $\SU(n)$-structures, this case corresponding to $n=1$. We will see in the following that in higher dimensions this correspondence becomes more intricate, as the $\SU(n)$-structure associated to the spinor appears in an arguably non-standard form.


\subsubsection{Compatibility conditions and the conjugate square}


In the previous subsections we have established that the complex-bilinear square of an irreducible and chiral complex spinor is an isotropic complex one-form $\alpha=\theta \in V^{\ast}_{\C}$ satisfying $\ast\theta=i\mu\theta$, whereas the Hermitian square of such a spinor is of the form $\widehat{\alpha} = r+i\mu r\nu$ for a real constant $r \in\R$. In the following, we take $\mu = 1$ for simplicity in the exposition. The goal of this subsection is to characterize the complex-bilinear and Hermitian squares of the \emph{same} irreducible complex and chiral spinor $\eta\in\Sigma^+$, and the complex-bilinear and Hermitian squares of its conjugate, namely $\frc(\eta)\in\Sigma^-$ since $(\Sigma,\gamma)$ is of real type in two Euclidean dimensions. By Proposition \ref{prop:Hermitiansquareconjugate} and Proposition \ref{prop:Bilinearsquareconjugate}, we know that the Hermitian $\widehat{\alpha}_c$ and complex-bilinear $\alpha_c$ squares of $\frc(\eta) \in \Sigma^{-}$ are given by:
\begin{equation*}
\widehat{\alpha}_c =  r - i  r\nu\, , \qquad \alpha_c = \bar{\theta}.
\end{equation*}

\noindent
Furthermore, by Proposition \ref{prop:compatibilitysquares}, the Hermitian and complex-bilinear squares are the squares of the \emph{same} spinor $\eta\in \Sigma^{+}$ if and only if the following algebraic relations are satisfied: 
\begin{equation*}  
(1 + i\nu) \diamond \beta \diamond \theta = 2 ((1 + i\nu) \diamond \beta)^{(0)} \theta \, , \qquad   \theta \diamond \beta \diamond \bar{\theta} = 2 r^2  ((1 + i\nu) \diamond  \tau(\beta))^{(0)}  (1 + i\nu)
\end{equation*}

\noindent
for every $\beta\in \wedge V^{\ast}_{\C}$ or, equivalently, a choice of $\beta\in \wedge V^{\ast}_{\C}$ such that $((1 + i\nu) \diamond \beta)^{(0)}  \neq 0$. Taking $\beta = 1$, this system of equations reduces to:
\begin{equation*}
\theta\wedge \bar{\theta} + \escal{\theta,\bar{\theta}} = 2 r^2 (1 + i \nu),
\end{equation*}

\noindent
which in turn is equivalent to $\escal{\theta,\bar{\theta}}=2r^2$.  It can be seen that this necessary condition is also sufficient according to the previous criteria for the choice of $\beta$, and hence we obtain the following result. 
\begin{cor}
A pair of complex forms $\widehat{\alpha}, \alpha \in \wedge V^{\ast}_{\mathbb{C}}$, and a pair of complex forms $\widehat{\alpha}_c, \alpha_c \in \wedge V^{\ast}_{\mathbb{C}}$ are respectively the Hermitian and complex-bilinear squares of a spinor $\eta\in \Sigma^{+}$ and its conjugate $\frc(\eta) \in \Sigma^{-}$ if and only if:
\begin{equation*}
\widehat{\alpha} = \sqrt{\frac{\escal{\theta,\bar\theta}}{2}} (1 +   \frac{\theta\wedge \bar{\theta}}{\escal{\theta,\bar\theta}})\, , \qquad \alpha = \theta\, , \qquad \widehat{\alpha}_c = \sqrt{\frac{\escal{\theta,\bar\theta}}{2}}(1 - \frac{\theta\wedge \bar{\theta}}{\escal{\theta,\bar\theta}})\, , \qquad \alpha_c = \bar\theta ,
\end{equation*}

\noindent
where $\theta \in V^{\ast}_{\mathbb{C}}$ is an isotropic complex one-form satisfying $\ast\theta = i\theta$.
\end{cor}

\noindent
In particular, the previous corollary implies that $\eta$ is of unit norm if and only if $2 \escal{\theta,\bar{\theta}} =1$.


\subsection{The square of an irreducible complex spinor in dimension three}


Let $(V,h)$ be an Euclidean vector space of real dimension three. Its associated Clifford algebra $\Cl(V^*,h^*)$ is isomorphic to $\Mat(2,\mathbb{C})$ and the irreducible complex Clifford module $(\Sigma,\gamma_\ell)$ admits an anti-linear endomorphism $D\colon \Sigma \to \Sigma$ satisfying $D^2=-\Id$. We consider $\Sigma$ to be endowed with an admissible skew-symmetric complex-bilinear pairing $\scB$ of negative adjoint type, and with a compatible Hermitian pairing $\scS$ of positive adjoint type, which is related to $\scB$ through the endomorphism $D$. The truncated Kähler-Atiyah algebra is defined on the following complex vector space: 
\begin{eqnarray*}
\wedge^<V^*_\C=\C\oplus V^*_\C\, ,
\end{eqnarray*}

\noindent
endowed with the \emph{truncated} geometric product:
\begin{eqnarray*}
\alpha_1\vee\alpha_2=\cP_<(\alpha_1\diamond\alpha_2+i\ell*\tau(\alpha_1\diamond\alpha_2))  
\end{eqnarray*}

\noindent
for every $\alpha_1,\alpha_2\in\wedge^<V^*_\C$.


\subsubsection{The Hermitian square}


We set $\kappa=1$ for simplicity. By Theorem \ref{thm_Hermitian_square_odd}, a complex exterior form $\widehat{\alpha}\in\wedge^<V^*_\C$ is the Hermitian square of a spinor $\eta \in\Sigma$ if and only if: 
\begin{eqnarray*}
\widehat{\alpha}\vee\widehat{\alpha}=2\widehat{\alpha}^{(0)}\widehat{\alpha},\quad \tau(\widehat{\alpha})=\overline{\widehat{\alpha}},
\end{eqnarray*}

\noindent
where we have used that $\scS$ is positive-definite. Let $\widehat{\alpha}=\widehat{\alpha}^{(0)}+\widehat{\alpha}^{(1)}\in\wedge^<V^*_\C$. The linear condition $\tau(\widehat{\alpha})=\overline{\widehat{\alpha}}$ implies that $\widehat{\alpha}^{(0)}$ and $\widehat{\alpha}^{(1)}$ are real. Set $r:=\widehat{\alpha}^{(0)}\in\R$ and $\vartheta:=\widehat{\alpha}^{(1)}\in V^*$. Hence $\widehat{\alpha}=r+\vartheta$ and the quadratic equation $\widehat{\alpha}\vee\widehat{\alpha}=2\widehat{\alpha}^{(0)}\widehat{\alpha}$ is equivalent to $\escal{\vartheta,\vartheta}=r^2$.

\begin{cor}
A complex exterior form $\widehat{\alpha}\in\wedge^<V^*_\C$ is the Hermitian square of a complex irreducible spinor $\eta\in\Sigma$ if and only if: $$\textstyle \widehat{\alpha}=\sqrt{\escal{\vartheta,\vartheta}}+\vartheta$$ for a real one-form $\vartheta\in V^*$.
\end{cor}

Given an irreducible complex spinor $\eta\in \Sigma$, consider now the following equation:
\begin{equation*}
F\cdot\eta:=(\Psi_{\ell}^< \circ 2\cP_{<} \circ \cP_{\ell})(F)(\eta)=0
\end{equation*}

\noindent
for a complex two-form $F\in \wedge^2 V^{\ast}_{\mathbb{C}}$. A computation gives:
\begin{equation*}
2\cP_<(\cP_\ell(F))=-i\ell*F
\end{equation*} 

\noindent
and thus by Lemma \ref{lemma:constrainedspinorodd}, a complex two-form $F$ satisfies the previous equation if and only if:
\begin{equation*}
\textstyle (-i\ell *F)\vee\widehat{\alpha}=(-i\ell*F)\vee(\sqrt{\escal{\vartheta,\vartheta}}+\vartheta)=0.
\end{equation*}

\noindent
Expanding the equation and multiplying by $i\ell$, we obtain:
\begin{equation*}
\textstyle \sqrt{\escal{\vartheta,\vartheta}} *F +\escal{*F,\vartheta}-i\ell*((*F)\wedge\vartheta)=0.    
\end{equation*}
 
\noindent
Separating by degree, this equation is equivalent to:
\begin{equation*}
\textstyle \sqrt{\escal{\vartheta,\vartheta}} \ast F  =  i\ell\,\iota_{\vartheta^\sharp}F,
\end{equation*}

\noindent
which gives the necessary and sufficient conditions for a complex two-form to \emph{annihilate} $\eta$ via Clifford multiplication. This condition can be reduced to a self-duality condition in two dimensions. Indeed, take $\escal{\vartheta,\vartheta} = 1$ for simplicity and split $V_\C^{\ast}$ orthogonally as follows:
\begin{eqnarray*}
V_\C^{\ast} = \langle \mathbb{C} \vartheta \rangle \oplus \langle \mathbb{C} \vartheta \rangle^{\perp}.
\end{eqnarray*}

\noindent
In this splitting, $F$ can be written as follows $F = \vartheta\wedge \beta + c \nu_o$, where $\beta\in \langle \mathbb{C} \vartheta \rangle^{\perp}$, $c\in \mathbb{C}$ is a constant, and $\nu_o$ is the induced Riemannian volume form on $\langle \mathbb{C} \vartheta \rangle^{\perp}$. Then, it can be seen that the previous condition for $F$ is equivalent to $c=0$ and an $\ast\beta = i\ell \beta$, where $\beta$ is understood as a one-form on the oriented Euclidean vector space $\langle \mathbb{C} \vartheta \rangle^{\perp}$.
\medskip

Now consider equation $H\cdot\eta=0$ for a complex three-form $H\in\wedge^3V_\C^*$ and irreducible complex spinor $\eta\in \Sigma$. Proceeding as in the previous case, it follows that $H$ satisfies $H\cdot\eta=0$ if and only if:
\begin{equation*}
\textstyle (-i\ell *H)\vee\widehat{\alpha}=(-i\ell*H)\vee(\sqrt{\escal{\vartheta,\vartheta}}+\vartheta)=0,
\end{equation*}

\noindent
which implies that $H=0$.


\subsubsection{The complex-bilinear square}


By Theorem \ref{thm:odd_complex-bilinear_square}, a complex exterior form $\alpha\in\wedge^< V^*_\C$ is the complex-bilinear square of an irreducible complex spinor $\eta \in \Sigma$ if and only if:
\begin{equation*}
\alpha\vee\alpha=2\alpha^{(0)}\alpha,\qquad(\pi\circ\tau)(\alpha) =-\alpha \, , \qquad \alpha\vee\beta\vee\alpha=2(\alpha\vee\beta)^{(0)}\alpha    
\end{equation*}

\noindent
for a fixed exterior form $\beta\in\wedge^<V^*_\C$ satisfying $(\alpha\vee\beta)^{(0)} \neq 0$. Let $\alpha = \alpha^{(0)} + \alpha^{(1)} \in \wedge^<V^*_\C$. The linear condition $(\pi\circ\tau)(\alpha)=-\alpha$ implies that $\alpha^{(0)}=0$. Set $\theta:=\alpha^{(1)}$. Equation $\alpha\vee\alpha=2\alpha^{(0)}\alpha$ then reduces to:
\begin{equation*}
\alpha\vee\alpha=\escal{\theta,\theta}=0\, .
\end{equation*}

\noindent
Since $\alpha^{(0)}=0$, taking $\beta=1$ in Theorem \ref{thm:odd_complex-bilinear_square} does not suffice to characterize the square of the spinor. Take $\beta=\bar\theta\in V^*_\C$. Then $\alpha\vee\beta=\theta\vee\bar\theta=\escal{\theta,\bar\theta}-i\ell*(\theta\wedge\bar\theta)$ and $(\alpha\vee\beta)^{(0)} = \escal{\theta,\bar\theta} = \abs{\Re(\theta)}^2 + \abs{\Im(\theta)}^2 \neq 0$ since $\theta\neq0$. This shows that choosing $\beta=\bar\theta $ in Theorem \ref{thm:odd_complex-bilinear_square} gives the necessary and sufficient conditions for $\theta$ to be the complex-bilinear square of $\eta\in\Sigma$. Using $\theta\vee\bar\theta+\bar\theta\vee\theta=2\escal{\theta,\bar\theta}$ and $\theta\vee\theta=0$ we compute: 
\begin{equation*}
\alpha\vee\beta\vee\alpha = \theta\vee\bar\theta\vee\theta=(2\escal{\theta,\bar\theta}-\bar\theta \vee \theta) \vee \theta = 2 \escal{\theta,\bar\theta} \theta= 2(\alpha\vee\beta)^{(0)}\alpha.
\end{equation*}

\begin{cor}
A complex exterior form $\alpha\in\wedge^< V^*_\C$ is the complex-bilinear square of a complex irreducible spinor $\eta\in\Sigma$ if and only if: $$\alpha=\theta$$ for an isotropic complex one-form $\theta\in V^{\ast}_{\C}$.
\end{cor}


\subsubsection{Compatibility conditions and the conjugate square}


We have determined that the complex-bilinear square of an irreducible complex spinor $\eta\in\Sigma$ is a complex one-form $\alpha_\eta=\theta\in V^*_\C$ satisfying $\escal{\theta,\theta}=0$, whereas the Hermitian square of $\eta$ is of the form $\widehat{\alpha}_\eta=r+\vartheta$ with $r\in\R$ and $\vartheta\in V^*$ satisfying $\escal{\vartheta,\vartheta}=r^2$. We consider now the complex-bilinear $\alpha_c$ and Hermitian $\widehat{\alpha}_c$ squares of the spinor \emph{conjugate} to $\eta$, namely $D\eta$, where $D\colon\Sigma\to\Sigma$ is the anti-linear endomorphism of Proposition \ref{prop:D_endomorphism}. By Proposition \ref{prop:Hermitiansquareconjugate_odd} and Proposition \ref{prop:Bilinearsquareconjugate_odd}, we have:
\begin{equation*}
\widehat{\alpha}_c = r - \vartheta\, , \qquad \alpha_c = - \bar{\theta}.
\end{equation*}

\noindent
On the other hand, by Proposition \ref{prop:compatibilitysquares_odd}, the Hermitian and complex-bilinear squares are the squares of the \emph{same} irreducible complex spinor $\eta\in \Sigma$ if and only if the following algebraic relations are satisfied:
\begin{equation*}  
(r + \vartheta) \vee \beta \vee \theta  = 2  ((r + \vartheta) \vee \beta)^{(0)} \theta \, , \qquad    \theta \vee \beta \vee \bar{\theta} = 2    ((r + \vartheta)\vee (\pi \circ \tau)(\beta))^{(0)}  (r + \vartheta)
\end{equation*}

\noindent
for every $\beta \in \wedge V^{\ast}_{\C}$, or, equivalently, for a choice of $\beta \in \wedge V^{\ast}_{\C}$ such that $((r + \vartheta) \vee \beta)^{(0)} \neq 0$. Setting $\beta = 1$ in the previous equations, we obtain: 
\begin{equation*}  
(r + \vartheta)  \vee \theta  = 2  r  \theta \, , \qquad    \theta  \vee \bar{\theta} = 2    r    (r + \vartheta).
\end{equation*}
 
\noindent
These equations are solved by: 
\begin{equation*}
\escal{\theta,\bar{\theta}} = 2r^2    \, , \qquad \vartheta = - \frac{i\ell}{2r}*(\theta\wedge \bar{\theta})\, , \qquad \theta=-\frac{i\ell}{r}*(\vartheta\wedge\theta).
\end{equation*}

\noindent
Since $\theta$ is isotropic, the second and third equations are equivalent. Furthermore, it can be seen that these necessary conditions are also sufficient according to the previous criteria for the choice of $\beta$. Hence, we obtain the following result.

\begin{cor}
A pair of complex forms $\widehat{\alpha}_\eta, \alpha_\eta \in \wedge^< V^{\ast}_{\mathbb{C}}$, and a pair of complex forms $\widehat{\alpha}_{D\eta}, \alpha_{D\eta} \in \wedge^< V^{\ast}_{\mathbb{C}}$ are respectively the Hermitian and complex-bilinear squares of a spinor $\eta\in\Sigma$ and its conjugate $D\eta\in\Sigma$ if and only if:
\begin{equation*}
\begin{aligned}
\widehat{\alpha}_\eta & =\sqrt{\frac{\escal{\theta,\bar\theta}}{2}}\left(1-\frac{i\ell}{\escal{\theta,\bar\theta}}*(\theta\wedge\bar\theta)\right),\quad \alpha_\eta=\theta,\\
\widehat{\alpha}_{D\eta}&=\sqrt{\frac{\escal{\theta,\bar\theta}}{2}}\left(1+\frac{i\ell}{\escal{\theta,\bar\theta}}*(\theta\wedge\bar\theta)\right),\quad \alpha_{D\eta}=-\bar\theta,
\end{aligned}
\end{equation*}

where $\theta\in V^*_\C$ is isotropic.
\end{cor}


\subsection{The square of an irreducible chiral complex spinor in dimension four}
\label{subsec:square4dEuclidean}


Let $(V,h)$ be a four-dimensional Euclidean vector space. Its associated real Clifford algebra $\mathrm{Cl}(V^*,h^*)$ is isomorphic to $\Mat(2,\mathbb{H})$ and the irreducible complex Clifford module $(\Sigma,\gamma)$ is of quaternionic type by Lemma \ref{lemma:real_or_quaternionic_type}. We endow $\Sigma$ with an admissible skew-symmetric complex-bilinear pairing $\scB$ of positive adjoint type and with a compatible Hermitian pairing $\scS$ of positive adjoint type, which are related by a compatible quaternionic structure $J\colon\Sigma\to\Sigma$.


\subsubsection{The Hermitian square}


Set $\kappa=1$ for simplicity. By Corollary \ref{cor:sq_S_chiral}, a complex exterior form $\widehat{\alpha}\in\wedge V^*_\C$ is the Hermitian square of a spinor $\eta \in \Sigma^{\mu}$ with chirality $\mu\in\Z_2$ if and only if: 
\begin{equation*}
\widehat{\alpha}\diamond\widehat{\alpha}=4\widehat{\alpha}^{(0)}\widehat{\alpha}, \qquad \tau(\widehat{\alpha}) = \overline{\widehat{\alpha}}, \qquad \widehat{\alpha} \diamond \beta \diamond \widehat{\alpha} = 4 (\widehat{\alpha}\diamond\beta)^{(0)} \widehat{\alpha},\qquad -*(\pi\circ\tau)(\widehat{\alpha})=\mu\widehat{\alpha}
\end{equation*}

\noindent
for a fixed exterior form $\beta\in\wedge V^*_\C$ satisfying $(\widehat{\alpha}\diamond\beta)^{(0)}\neq0$. Let $\widehat{\alpha}=\sum_{k=0}^4\widehat{\alpha}^{(k)}\in\wedge V^*_\C$. The linear condition $\tau(\widehat{\alpha}) = \overline{\widehat{\alpha}}$ implies that $\widehat{\alpha}^{(0)}$, $\widehat{\alpha}^{(1)}$ and $\widehat{\alpha}^{(4)}$ are real and that $\widehat{\alpha}^{(2)}$ and $\widehat{\alpha}^{(3)}$ are imaginary. The remaining linear condition $-*(\pi\circ\tau)(\widehat{\alpha})=\mu\widehat{\alpha}$ implies that $-*\widehat{\alpha}^{(0)}=\mu\widehat{\alpha}^{(4)}$, $*\widehat{\alpha}^{(2)}=\mu\widehat{\alpha}^{(2)}$ and $\widehat{\alpha}^{(1)}=\widehat{\alpha}^{(3)}=0$. Set $r:=\widehat{\alpha}^{(0)}\in\R$ and $i\omega:=\widehat{\alpha}^{(2)}$ for $\omega\in\wedge^2V^*$. Hence $\widehat{\alpha}=r+i\omega-\mu r\nu$. We compute: 
\begin{equation*}
 \widehat{\alpha}\diamond\widehat{\alpha}=(r+i\omega-\mu r\nu)\diamond(r+i\omega-\mu r\nu)=2r^2+4ir\omega-2\mu r^2\nu-\omega\wedge\omega+\escal{\omega,\omega}.  
\end{equation*}

\noindent
Then the equation $\widehat{\alpha}\diamond\widehat{\alpha}=4\widehat{\alpha}^{(0)}\widehat{\alpha}$ reduces to:
\begin{equation*}
 \escal{\omega,\omega}=2r^2,\quad\omega\wedge\omega=2\mu r^2\nu.   
\end{equation*}

\noindent
Note that these two equations are equivalent since $*\omega=\mu\omega$. Hence, we obtain the following result.

\begin{cor}
\label{cor:Hermitian4d}
A complex exterior form $\widehat{\alpha} \in\wedge V^*_\C$ is the Hermitian square of a complex irreducible spinor $\eta\in\Sigma^{\mu}$ with chirality $\mu$ if and only if:
\begin{equation*}
\widehat{\alpha} =   \sqrt{\frac{ \escal{\omega,\omega}}{2}} + i\omega-\mu  \sqrt{\frac{ \escal{\omega,\omega}}{2}} \nu =  \sqrt{\frac{ \escal{\omega,\omega}}{2}} + i\omega - \frac{\omega\wedge \omega}{\sqrt{2\escal{\omega,\omega}}}  
\end{equation*}

\noindent
for a real two-form $\omega\in\wedge^2V^*$ satisfying $\ast \omega=\mu\omega$.
\end{cor}

Given an irreducible complex spinor $\eta\in \Sigma$, consider now the following equation:
\begin{equation*}
F\cdot \eta := \Psi_{\gamma}(F)(\eta) = 0
\end{equation*}

\noindent
for a complex two-form $F\in \wedge^2 V_\C^{\ast}$. By Lemma \ref{lemma:constrainedspinoreven}, a two-form $F$ satisfies the previous equation if and only if: 
\begin{equation*}
 F\diamond\widehat{\alpha}=F\diamond(r+i\omega-\mu r\nu)=0.   
\end{equation*} 

\noindent
Separating by degree, this equation is equivalent to:
\begin{equation*}
F\wedge\omega=0,\qquad    F+\mu \ast F = i \sqrt{\frac{2}{ \escal{\omega,\omega}}} F\triangle_1\omega ,\qquad \escal{F,\omega}=0\, .    
\end{equation*}

\noindent
Note that the first and third equations are equivalent since $*\omega=\mu\omega$.\medskip

Now consider a complex three-form $H\in\wedge^3V_\C^*$. It satisfies the equation $H\cdot\eta=0$ if and only if: $$H\diamond\widehat{\alpha}=H\diamond(r+i\omega-\mu r\nu)=0.$$

Separating by degree, this equation is equivalent to:  
\begin{equation*}
H+i \sqrt{\frac{2}{ \escal{\omega,\omega}}}  H\triangle_1\omega=0,\qquad \ast H = i  \mu \sqrt{\frac{2}{ \escal{\omega,\omega}}} H\triangle_2\omega.
\end{equation*}

\noindent
The bilinear operations $\triangle_1$ and $\triangle_2$ are defined in Appendix \ref{app:KA}.


\subsubsection{The complex-bilinear square}


Let $\mu\in\mathbb{Z}_2$. By Corollary \ref{cor:sq_B_chiral}, an exterior form $\alpha\in\wedge V^*_\C$ is the complex-bilinear square of a spinor $\eta\in\Sigma$ with chirality $\mu$ if and only if:
\begin{equation*}
 \alpha\diamond\alpha=4\alpha^{(0)}\alpha,\quad\tau(\alpha)=-\alpha,\quad\alpha\diamond\beta\diamond\alpha=4(\alpha\diamond\beta)^{(0)}\alpha,\quad-*(\pi\circ\tau)(\alpha)=\mu\alpha   
\end{equation*}

\noindent
for a fixed exterior form $\beta\in\wedge V^*_\C$ satisfying $(\alpha\diamond\beta)^{(0)}\neq0$. Let $\alpha=\sum_{k=0}^4\alpha^{(k)}\in\wedge V^*_\C$. The linear condition $\tau(\alpha)=-\alpha$ implies that $\alpha^{(0)} = \alpha^{(1)} = \alpha^{(4)}=0$ and the other linear condition $-*(\pi\circ\tau)(\alpha)=\mu\alpha$ implies that $\alpha^{(3)}=0$ and $*\alpha^{(2)}=\mu\alpha^{(2)}$. Set $\omega:=\alpha^{(2)}$. Equation $\alpha\diamond\alpha=4\alpha^{(0)}\alpha$ then reduces to:
\begin{eqnarray*}
 \omega\diamond\omega=\omega\wedge\omega-\escal{\omega,\omega}=0,
\end{eqnarray*}

\noindent
thus $\omega\wedge\omega=0$ and $\escal{\omega,\omega}=0$. Equation $\omega\wedge\omega=0$ implies that $\omega$ is decomposable, that is $\omega=\theta_1\wedge\theta_2$ for $\theta_1,\theta_2\in V^*_\C$. Then, equation $\escal{\omega,\omega}=0$ becomes $\abs{\theta_1}^2\abs{\theta_2}^2=\escal{\theta_1,\theta_2}^2$.\medskip

\noindent
Since $\alpha^{(0)}=0$, taking $\beta=1$ does not suffice to characterize the square of the spinor. We take $\beta=\bar\omega=\bar\theta_1\wedge\bar\theta_2\in\wedge^2V^*_\C$. Then:
\begin{equation*}
\alpha\diamond\beta=\omega\diamond\bar\omega=\omega\wedge\bar\omega-\omega\triangle_1\bar\omega-\escal{\omega,\bar\omega}
\end{equation*}

\noindent
and $(\alpha\diamond\beta)^{(0)}=-\escal{\omega,\bar\omega}=-\abs{\Re(\omega)}^2-\abs{\Im(\omega)}^2\neq0$ since $\omega\neq0$. After a long but straightforward computation, we get: 
\begin{eqnarray*}
 \omega\diamond\bar\omega\diamond\omega=(\bar\omega\diamond\omega-2\omega\triangle_1\bar\omega)\diamond\omega=-4\escal{\omega,\bar\omega}\omega-2\mathbb{A}(\theta_1,\theta_2),
\end{eqnarray*}

\noindent
where: 
\begin{align*}
\mathbb{A}(\theta_1,\theta_2)&:=\big(\abs{\theta_2}^2\escal{\theta_1,\bar\theta_2}-\escal{\theta_1,\theta_2}\escal{\theta_2,\bar\theta_2}\big)\theta_1\wedge\bar\theta_1+\big(\escal{\theta_1,\theta_2}\escal{\theta_2,\bar\theta_1}-\abs{\theta_2}^2\escal{\theta_1,\bar\theta_1}\big)\theta_1\wedge\bar\theta_2\\
&\quad+\big(\abs{\theta_1}^2\escal{\theta_2,\bar\theta_2}-\escal{\theta_1,\theta_2}\escal{\theta_1,\bar\theta_2}\big)\theta_2\wedge\bar\theta_1+\big(\escal{\theta_1,\theta_2}\escal{\theta_1,\bar\theta_1}-\abs{\theta_1}^2\escal{\theta_2,\bar\theta_1}\big)\theta_2\wedge\bar\theta_2.
\end{align*}

\noindent
Equation $\alpha\diamond\beta\diamond\alpha=4(\alpha\diamond\beta)^{(0)}\alpha$ holds if and only if $\mathbb{A}=0$, which amounts to: 
$$\begin{aligned}
    \abs{\theta_1}^2\escal{\theta_2,\bar\theta_1}&=\escal{\theta_1,\theta_2}\escal{\theta_1,\bar\theta_1},\\
    \abs{\theta_1}^2\escal{\theta_2,\bar\theta_2}&=\escal{\theta_1,\theta_2}\escal{\theta_1,\bar\theta_2},
\end{aligned}\qquad\begin{aligned}
    \abs{\theta_2}^2\escal{\theta_1,\bar\theta_1}&=\escal{\theta_1,\theta_2}\escal{\theta_2,\bar\theta_1},\\
    \abs{\theta_2}^2\escal{\theta_1,\bar\theta_2}&=\escal{\theta_1,\theta_2}\escal{\theta_2,\bar\theta_2}.
\end{aligned}$$

\noindent
Note that $\escal{\theta_i,\bar\theta_i}\neq0$ since $\theta_i\neq0$, $i=1,2$. Manipulating the above expressions we obtain: 
\begin{equation*}
 \escal{\theta_1,\theta_2}\big(\lvert\escal{\theta_1,\bar\theta_2}\rvert^2-\escal{\theta_1,\bar\theta_1}\escal{\theta_2,\bar\theta_2}\big)=0.
\end{equation*}
 
\noindent
If $\lvert\escal{\theta_1,\bar\theta_2}\rvert^2=\escal{\theta_1,\bar\theta_1}\escal{\theta_2,\bar\theta_2}$, then $\theta_2=c\theta_1$ for $c\in\C^*$ by the Cauchy-Schwarz inequality since $\escal{\cdot,\bar\cdot}$ is an inner product in $V^*_\C$. However, this would imply that $\omega=\theta_1\wedge\theta_2=0$. Therefore $\escal{\theta_1,\theta_2}=0$, so $\abs{\theta_1}=\abs{\theta_2}=0$ and $\mathbb{A}=0$.\medskip

We conclude the following:

\begin{cor}
\label{cor:Bilinear4d}
A complex exterior form $\alpha\in\wedge V^*_\C$ is the complex-bilinear square of a complex irreducible spinor $\eta\in\Sigma^\mu$ with chirality $\mu$ if and only if: 
\begin{equation*}
\alpha=\theta_1\wedge\theta_2    
\end{equation*}

\noindent
for complex one-forms $\theta_1,\theta_2 \in V^{\ast}_{\C}$ satisfying:
\begin{equation*}
*(\theta_1\wedge\theta_2)=\mu(\theta_1\wedge\theta_2)\quad\text{and}\quad\abs{\theta_1}=\abs{\theta_2}=\escal{\theta_1,\theta_2}=0.
\end{equation*}
\end{cor}

\noindent
As a consequence of this theorem, we recover the following well-known correspondence between irreducible complex chiral spinors in four dimensions and $\SU(2)$-structures.

\begin{prop}
There is a one-to-one correspondence between $\SU(2)$-structures and chiral spinorial forms on a four-dimensional Euclidean vector space $(V,h)$.
\end{prop}

\begin{proof}
Let $\theta_i=\theta_i^R+i\theta_i^I$ for $i=1,2$ and set $\mu = 1$ for definiteness. Since $\escal{\alpha,\bar\alpha}\neq0$ and $*\alpha=\alpha$, then $\alpha\wedge\bar\alpha\neq0$ and this implies that $\theta_1^R,\theta_1^I,\theta_2^R,\theta_2^I$ are linearly independent real one-forms. We thus have $V^*=\Span_\R\{\theta_1^R,\theta_1^I,\theta_2^R,\theta_2^I\}$. The conditions $\abs{\theta_1}=\abs{\theta_2}=\escal{\theta_1,\theta_2}=0$ imply that:
\begin{eqnarray*}
\abs{\theta_i^R}=\abs{\theta_i^I},\quad\escal{\theta_i^R,\theta_i^I}=0
\end{eqnarray*}

\noindent
for $i=1,2$ and:
\begin{eqnarray*}
\escal{\theta_1^R,\theta_2^R}=\escal{\theta_1^I,\theta_2^I},\quad\escal{\theta_1^R,\theta_2^I}+\escal{\theta_1^I,\theta_2^R}=0.
\end{eqnarray*}

\noindent
We define the complex structure $I$ on $V^*$ by $I(\theta_i^I):=\theta_i^R$ and $I(\theta_i^R):=-\theta_i^I$ for $i=1,2$. One can check that $I$ is $h$-orthogonal and therefore $(h,I,\theta_1\wedge\theta_2)$ defines a $\SU(2)$-structure on $V$. Note that the complex two-form $\theta_1\wedge\theta_2$ is of type $(2,0)$ with respect to the complex structure $I$.
\end{proof}

\noindent
Using the fact that $\SU(2)$ is connected and simply connected, together with the equivariance of the spinor square map, we recover the following well-known result.

\begin{cor}
The stabilizer in $\Spin(4)$ of an irreducible and chiral complex spinor is $\SU(2)$.
\end{cor}


\subsubsection{Compatibility conditions and the conjugate square}


The goal of this subsection is to characterize the complex-bilinear and Hermitian squares of the \emph{same} irreducible complex and chiral spinor $\eta\in\Sigma^{\mu}$, as well as the complex-bilinear and Hermitian squares of its conjugate, namely $J(\eta)\in\Sigma^{\mu}$ since $(\Sigma,\gamma)$ is of quaternionic type in four Euclidean dimensions. By Proposition \ref{prop:Hermitiansquareconjugate} and Proposition \ref{prop:Bilinearsquareconjugate}, together with Corollary \ref{cor:Hermitian4d} and Corollary \ref{cor:Bilinear4d}, we know that the Hermitian $\widehat{\alpha}_c$ and complex-bilinear $\alpha_c$ squares of $J(\eta) \in \Sigma^{\mu}$ are given by:
\begin{equation*}
\widehat{\alpha}_c =  \sqrt{\frac{ \escal{\omega,\omega}}{2}} - i\omega - \frac{\omega\wedge \omega}{\sqrt{2\escal{\omega,\omega}}}  \, , \qquad \alpha_c = \bar{\theta}_1 \wedge \bar{\theta}_2,
\end{equation*}

\noindent
where $\theta_1 , \theta_2 \in V^{\ast}_{\C}$ are complex one-forms satisfying $\ast (\theta_1\wedge\theta_2) = \mu (\theta_1\wedge\theta_2)$ and $\abs{\theta_1}=\abs{\theta_2}=\escal{\theta_1,\theta_2}=0$.

\noindent
Furthermore, by Proposition \ref{prop:compatibilitysquares}, the Hermitian and complex-bilinear squares are the squares of the \emph{same} spinor $\eta\in \Sigma^{\mu}$ if and only if the following algebraic relations are satisfied: 
\begin{equation*}  
\widehat{\alpha} \diamond \beta \diamond \alpha   = 4 (\widehat{\alpha} \diamond \beta)^{(0)} \alpha \, , \qquad   \alpha \diamond \beta \diamond \alpha_c = -  4    (\widehat{\alpha} \diamond   \tau(\beta))^{(0)} \widehat{\alpha}
\end{equation*} 

\noindent
for every $\beta\in \wedge V^{\ast}_{\C}$ or, equivalently, a choice of $\beta\in \wedge V^{\ast}_{\C}$ such that $(\widehat{\alpha} \diamond \beta)^{(0)}  \neq 0$. Taking $\beta = 1$, this system of equations reduces to:
\begin{eqnarray*}
\sqrt{2\escal{\omega,\omega}}\, \theta_1 \wedge \theta_2 =  i \omega \diamond (\theta_1 \wedge \theta_2)    \, , \qquad  \frac{(\theta_1 \wedge \theta_2) \diamond (\bar{\theta}_1 \wedge \bar{\theta}_2)}{ \sqrt{8 \escal{\omega,\omega}}} = - \sqrt{\frac{ \escal{\omega,\omega}}{2}} - i\omega + \mu \sqrt{\frac{ \escal{\omega,\omega}}{2}} \nu.
\end{eqnarray*}

\noindent
We compute:
\begin{align*}
\omega \diamond (\theta_1\wedge \theta_2) &= \omega \wedge \theta_1 \wedge \theta_2 - \omega(\theta_1^{\sharp}) \wedge \theta_2 + \omega(\theta_2^{\sharp}) \wedge \theta_1 - \omega(\theta_1^{\sharp} , \theta_2^{\sharp}),\\
(\theta_1 \wedge \theta_2) \diamond (\bar{\theta}_1 \wedge \bar{\theta}_2) &= \theta_1 \wedge \theta_2 \wedge\bar{\theta}_1 \wedge \bar{\theta}_2\\
& \quad - \langle \theta_2 , \bar{\theta}_2 \rangle \theta_1 \wedge \bar{\theta}_1 - \langle \theta_1 , \bar{\theta}_1 \rangle \theta_2 \wedge \bar{\theta}_2 - \langle \theta_1 , \bar{\theta}_2 \rangle \bar{\theta}_1 \wedge \theta_2 + \langle \bar{\theta}_1 , \theta_2 \rangle \theta_1 \wedge \bar{\theta}_2 \\
&\quad - \langle \theta_1 , \bar{\theta}_1 \rangle \langle \theta_2 , \bar{\theta}_2 \rangle + \langle \bar{\theta}_1 , \theta_2 \rangle \langle \theta_1 , \bar{\theta}_2 \rangle.
\end{align*}

\noindent
Hence, we obtain the following conditions:
\begin{eqnarray*}
& \omega \wedge \theta_1 \wedge \theta_2 = 0\, , \qquad  \sqrt{2\escal{\omega,\omega}}\, i \theta_1 \wedge \theta_2  =  -\omega(\theta_2^{\sharp}) \wedge \theta_1 + \omega(\theta_1^{\sharp}) \wedge \theta_2\, , \qquad \omega(\theta_1^{\sharp} , \theta_2^{\sharp}) = 0,\\
&  \sqrt{8 \escal{\omega,\omega}}\, i \omega =  \langle \theta_2 , \bar{\theta}_2 \rangle \theta_1 \wedge \bar{\theta}_1 + \langle \theta_1 , \bar{\theta}_1 \rangle \theta_2 \wedge \bar{\theta}_2 + \langle \theta_1 , \bar{\theta}_2 \rangle \bar{\theta}_1 \wedge \theta_2 - \langle \bar{\theta}_1 , \theta_2 \rangle \theta_1 \wedge \bar{\theta}_2,\\
& 2 \escal{\omega,\omega} =  \langle \theta_1 , \bar{\theta}_1 \rangle \langle \theta_2 , \bar{\theta}_2 \rangle - \langle \bar{\theta}_1 , \theta_2 \rangle \langle \theta_1 , \bar{\theta}_2 \rangle\, , \qquad \theta_1\wedge\theta_2\wedge\bar\theta_1\wedge\bar\theta_2=2\mu\escal{\omega,\omega}\nu.
\end{eqnarray*}

\noindent
These are solved by:
\begin{equation}
\label{eq:omega_in_terms_of_thetas}
2i \omega = \frac{\langle \theta_2 , \bar{\theta}_2 \rangle \theta_1 \wedge \bar{\theta}_1 + \langle \theta_1 , \bar{\theta}_1 \rangle \theta_2 \wedge \bar{\theta}_2 + \langle \theta_1 , \bar{\theta}_2 \rangle \bar{\theta}_1 \wedge \theta_2 - \langle \bar{\theta}_1 , \theta_2 \rangle \theta_1 \wedge \bar{\theta}_2}{\sqrt{\langle \theta_1 , \bar{\theta}_1 \rangle \langle \theta_2 , \bar{\theta}_2 \rangle - \langle \bar{\theta}_1 , \theta_2 \rangle \langle \theta_1 , \bar{\theta}_2 \rangle}},
\end{equation}

\noindent
which determines the Hermitian square of $\eta \in \Sigma^{\mu}$ in terms of its complex-bilinear square. Hence, we arrive at the following final characterization of the squares of an irreducible and chiral complex spinor in four Euclidean dimensions.

\begin{cor}
A pair of complex forms $\widehat{\alpha}_\eta, \alpha_\eta \in \wedge V^{\ast}_{\mathbb{C}}$, and a pair of complex forms $\widehat{\alpha}_{J\eta}, \alpha_{J\eta} \in \wedge V^{\ast}_{\mathbb{C}}$ are respectively the Hermitian and complex-bilinear squares of a spinor $\eta\in\Sigma^\mu$ and its conjugate $J\eta\in\Sigma^\mu$ if and only if:
\begin{equation*}
\begin{aligned}
\widehat{\alpha}_{\eta} =    \frac{1}{2} \sqrt{\langle \theta_1 , \bar{\theta}_1 \rangle \langle \theta_2 , \bar{\theta}_2 \rangle - \vert \langle \bar{\theta}_1 , \theta_2 \rangle\vert^2} + i\omega-\frac{\theta_1 \wedge \theta_2 \wedge \bar{\theta}_1 \wedge \bar{\theta}_2}{2\sqrt{\langle \theta_1 , \bar{\theta}_1 \rangle \langle \theta_2 , \bar{\theta}_2 \rangle - \vert \langle \bar{\theta}_1 , \theta_2 \rangle\vert^2}} \, , \qquad \alpha_{\eta} = \theta_1 \wedge \theta_2,\\
\widehat{\alpha}_{J\eta} =    \frac{1}{2} \sqrt{\langle \theta_1 , \bar{\theta}_1 \rangle \langle \theta_2 , \bar{\theta}_2 \rangle - \vert \langle \bar{\theta}_1 , \theta_2 \rangle\vert^2} - i\omega-\frac{\theta_1 \wedge \theta_2 \wedge \bar{\theta}_1 \wedge \bar{\theta}_2}{2\sqrt{\langle \theta_1 , \bar{\theta}_1 \rangle \langle \theta_2 , \bar{\theta}_2 \rangle - \vert \langle \bar{\theta}_1 , \theta_2 \rangle\vert^2}} \, , \qquad \alpha_{J\eta} = \bar{\theta}_1 \wedge \bar{\theta}_2,
\end{aligned}
\end{equation*}

\noindent
where $\theta_1 , \theta_2 \in V^*_\C$ are orthogonal and isotropic and $\omega$ is as in Equation \eqref{eq:omega_in_terms_of_thetas}.
\end{cor}


\subsection{The square of an irreducible complex spinor in dimension five}


Let $(V,h)$ be an Euclidean vector space of real dimension five. Its associated real Clifford algebra $\Cl(V^*,h^*)$ is isomorphic to $\Mat(2,\mathbb{H})\oplus\Mat(2,\mathbb{H})$ and the irreducible complex Clifford module $(\Sigma,\gamma_\ell)$ is quaternionic by Lemma \ref{lemma:R_orQ_p-q=1,5}. We endow $\Sigma$ with an admissible skew-symmetric complex-bilinear pairing $\scB$ of positive adjoint type and with a compatible Hermitian pairing $\scS$ of positive adjoint type, which are related by a compatible quaternionic structure $J\colon\Sigma\to\Sigma$. The truncated Kähler-Atiyah algebra is isomorphic to: $$\wedge^<V^*_\C=\C\oplus V^*_\C\oplus\wedge^2V^*_\C,$$
and $$\alpha_1\vee\alpha_2=\cP_<(\alpha_1\diamond\alpha_2-\ell*\tau(\alpha_1\diamond\alpha_2))$$ for all $\alpha_1,\alpha_2\in\wedge^<V^*_\C$.


\subsubsection{The Hermitian square}


We set $\kappa=1$ for simplicity. By Theorem \ref{thm_Hermitian_square_odd}, a complex exterior form $\widehat{\alpha}\in\wedge^<V^*_\C$ is the Hermitian square of a spinor $\eta\in\Sigma$ if and only if: $$\widehat{\alpha}\vee\widehat{\alpha}=4\widehat{\alpha}^{(0)}\widehat{\alpha},\quad\tau(\widehat{\alpha})=\overline{\widehat{\alpha}},\quad\widehat{\alpha}\vee\beta\vee\widehat{\alpha}=4(\widehat{\alpha}\vee\beta)^{(0)}\widehat{\alpha}$$ for a fixed exterior form $\beta\in\wedge^<V^*_\C$ satisfying $(\widehat{\alpha}\vee\beta)^{(0)}\neq0$. Let $\widehat{\alpha}=\widehat{\alpha}^{(0)}+\widehat{\alpha}^{(1)}+\widehat{\alpha}^{(2)}\in\wedge^<V^*_\C$. The linear condition $\tau(\widehat{\alpha})=\overline{\widehat{\alpha}}$ implies that $\widehat{\alpha}^{(0)}$ and $\widehat{\alpha}^{(1)}$ are real and that $\widehat{\alpha}^{(2)}$ is imaginary. Set $r:=\widehat{\alpha}^{(0)}\in\R$, $\theta:=\widehat{\alpha}^{(1)}\in V^*$ and $i\omega:=\widehat{\alpha}^{(2)}$ for $\omega\in\wedge^2V^*$. Hence $\widehat{\alpha}=r+\theta+i\omega$. Now we compute:
\begin{align*}
\widehat{\alpha}\vee\widehat{\alpha}&=(r+\theta+i\omega)\vee(r+\theta+i\omega)\\
&=r^2+\escal{\theta,\theta}+\escal{\omega,\omega}+2r\theta+\ell*(\omega\wedge\omega)+2ir\omega+2i\ell*(\theta\wedge\omega).
\end{align*}

Then the equation $\widehat{\alpha}\vee\widehat{\alpha}=4\widehat{\alpha}^{(0)}\widehat{\alpha}$ reduces to: $$\escal{\theta,\theta}+\escal{\omega,\omega}=3r^2,\quad\ell*(\omega\wedge\omega)=2r\theta,\quad\ell*(\theta\wedge\omega)=r\omega.$$

From the second equation we get $\theta=\frac{\ell}{2r}*(\omega\wedge\omega)$. Plugging this into the last equation and taking the Hodge dual of it gives us $2r^2*\omega=\omega\wedge*(\omega\wedge\omega)$. Taking the wedge with $\omega$ of this expression implies that $\escal{\omega\wedge\omega,\omega\wedge\omega}=2r^2\escal{\omega,\omega}$. Using this formula we obtain $\escal{\theta,\theta}=\frac{1}{2}\escal{\omega,\omega}$. Hence the first equation above implies $\escal{\omega,\omega}=2r^2$.\medskip

We conclude the following:

\begin{cor}
\label{cor:Hermitian5d}
A complex exterior form $\widehat{\alpha}\in\wedge^<V^*_\C$ is the Hermitian square of a complex irreducible spinor $\eta\in\Sigma$ if and only if:
\begin{equation*}
\widehat{\alpha} = \sqrt{\frac{\escal{\omega,\omega}}{2}} + \frac{\ell *(\omega\wedge\omega)}{\sqrt{2\escal{\omega,\omega}}}+i\omega
\end{equation*}

\noindent
for a real two-form $\omega\in\wedge^2V^*$ satisfying $\escal{\omega,\omega} \ast \omega = \omega\wedge*(\omega\wedge\omega)$.
\end{cor}

\noindent
Given an irreducible complex spinor $\eta\in \Sigma$, consider now the following equation:
\begin{equation*}
F\cdot\eta:=(\Psi_{\ell}^< \circ 2\cP_{<} \circ \cP_{\ell})(F)(\eta)=0
\end{equation*}

\noindent
for a complex two-form $F\in \wedge^2 V^{\ast}_{\mathbb{C}}$. Since $F\in\wedge^<V_\C^*$, we have $2\cP_<(\cP_\ell(F))=F$. By Lemma \ref{lemma:constrainedspinorodd}, a complex two-form $F$ satisfies the previous equation if and only if:
\begin{equation*}
F\vee\widehat{\alpha}=F\vee(r+\theta+i\omega)=0.
\end{equation*}

\noindent
Separating by degree, this equation is equivalent to: $$\sqrt{\frac{\escal{\omega,\omega}}{2}} F + \ell*(F\wedge\theta)-iF\triangle_1\omega=0,\quad F\triangle_1\theta+i\ell*(F\wedge\omega)=0,\quad\escal{F,\omega}=0.$$

Now consider a complex three-form $H\in\wedge^3V_\C^*$. It satisfies the equation $H\cdot\eta=0$ if and only if:
\begin{equation*}
(\ell*H)\vee\widehat{\alpha}=(\ell*H)\vee(r+\theta+i\omega)=0.
\end{equation*}

\noindent
Separating by degree, this equation is equivalent to: 
\begin{eqnarray*}
\sqrt{\frac{\escal{\omega,\omega}}{2}} \ast H + \ell H\triangle_1\theta-i(*H)\triangle_1\omega=0,\quad *(H\wedge\theta)+i\ell H\triangle_2\omega=0,\quad \escal{*H,\omega}=0.    
\end{eqnarray*}

\noindent
The bilinear operations $\triangle_1$ and $\triangle_2$ are defined in Appendix \ref{app:KA}.


\subsubsection{The complex-bilinear square}


By Theorem \ref{thm:odd_complex-bilinear_square}, a complex exterior form $\alpha\in\wedge^< V^*_\C$ is the complex-bilinear square of a spinor $\eta\in\Sigma$ if and only if:
$$\alpha\vee\alpha=4\alpha^{(0)}\alpha,\quad\tau(\alpha)=-\alpha,\quad\alpha\vee\beta\vee\alpha=4(\alpha\vee\beta)^{(0)}\alpha$$ for a fixed exterior form $\beta\in\wedge^<V^*_\C$ satisfying $(\alpha\vee\beta)^{(0)}\neq0$. Let $\alpha=\alpha^{(0)}+\alpha^{(1)}+\alpha^{(2)}\in\wedge^<V^*_\C$. The linear condition $\tau(\alpha)=-\alpha$ implies that $\alpha^{(0)}=\alpha^{(1)}=0$. Set $\omega:=\alpha^{(2)}$. Equation $\alpha\vee\alpha=4\alpha^{(0)}\alpha$ then reduces to:
$$\omega\vee\omega=-\escal{\omega,\omega}-\ell*(\omega\wedge\omega)=0,$$ thus $\escal{\omega,\omega}=0$ and $\omega\wedge\omega=0$. Equation $\omega\wedge\omega=0$ implies that $\omega$ is decomposable, that is $\omega=\theta_1\wedge\theta_2$ for $\theta_1,\theta_2\in V^*_\C$. Then, equation $\escal{\omega,\omega}=0$ becomes $\abs{\theta_1}^2\abs{\theta_2}^2=\escal{\theta_1,\theta_2}^2$.\medskip

Since $\alpha^{(0)}=0$, taking $\beta=1$ does not suffice to characterize the square of the spinor. We take $\beta=\bar\omega=\bar\theta_1\wedge\bar\theta_2\in\wedge^2V^*_\C$. Then: $$\alpha\vee\beta=\omega\vee\bar\omega=-\escal{\omega,\bar\omega}-\ell*(\omega\wedge\bar\omega)-\omega\triangle_1\bar\omega$$ and $(\alpha\vee\beta)^{(0)}=-\escal{\omega,\bar\omega}=-\abs{\Re(\omega)}^2-\abs{\Im(\omega)}^2\neq0$ since $\omega\neq0$. Using $\omega\vee\bar\omega-\bar\omega\vee\omega=-2\omega\triangle_1\bar\omega$ and $\omega\vee\omega=0$ we get:
$$\omega\vee\bar\omega\vee\omega=-2(\omega\triangle_1\bar\omega)\vee\omega=2\escal{\omega\triangle_1\bar\omega,\omega}+2\ell*((\omega\triangle_1\bar\omega)\wedge\omega)+2(\omega\triangle_1\bar\omega)\triangle_1\omega,$$
where: $$\omega\triangle_1\bar\omega=\escal{\theta_1,\bar\theta_1}\theta_2\wedge\bar\theta_2-\escal{\theta_1,\bar\theta_2}\theta_2\wedge\bar\theta_1-\escal{\theta_2,\bar\theta_1}\theta_1\wedge\bar\theta_2+\escal{\theta_2,\bar\theta_2}\theta_1\wedge\bar\theta_1.$$

A similar computation as the one performed in the four-dimensional case gives us that the equation $\omega\vee\bar\omega\vee\omega=-4\escal{\omega,\bar\omega}\omega$ holds if and only if $\abs{\theta_1}=\abs{\theta_2}=\escal{\theta_1,\theta_2}=0$.

\begin{cor}
\label{cor:Bilinear5d}
A complex exterior form $\alpha\in\wedge^< V^*_\C$ is the complex-bilinear square of a complex irreducible spinor $\eta\in\Sigma$ if and only if: $$\alpha=\theta_1\wedge\theta_2$$ for complex one-forms $\theta_1,\theta_2 \in V^{\ast}_{\C}$ satisfying $\abs{\theta_1}=\abs{\theta_2}=\escal{\theta_1,\theta_2}=0$.
\end{cor}

Let $\theta_i=\theta_i^R+i\theta_i^I$ for $i=1,2$. Define the following real one-form on $V$: 
\begin{equation*}
\vartheta:=*(\theta_1^R\wedge\theta_1^I\wedge\theta_2^R\wedge\theta_2^I)=\tfrac{1}{4}*(\omega\wedge\bar\omega)\neq0    
\end{equation*}

\noindent
as well as the following real two-forms on $V$: 
\begin{equation*}
\varpi_1:=\theta_1^R\wedge\theta_1^I+\theta_2^R\wedge\theta_2^I,\quad\varpi_2:=\theta_1^R\wedge\theta_2^R+\theta_2^I\wedge\theta_1^I=\Re(\omega),\quad\varpi_3:=\theta_1^R\wedge\theta_2^I+\theta_1^I\wedge\theta_2^R=\Im(\omega).    
\end{equation*}
 
\noindent
The tuple $(\vartheta,\varpi_1,\varpi_2,\varpi_3)$ defines a $\SU(2)$-structure on the five-dimensional vector space $V$ as explained in \cite[Proposition 1.1]{CS07}. Therefore, we nicely recover the following well-known result. 

\begin{cor}
There is a one-to-one correspondence between $\SU(2)$-structures and irreducible complex spinors on a five-dimensional Euclidean vector space $(V,h)$.
\end{cor}


\subsubsection{Compatibility conditions and the conjugate square}


The goal of this subsection is to characterize the complex-bilinear and Hermitian squares of the \emph{same} irreducible complex spinor $\eta\in\Sigma$, as well as the complex-bilinear and Hermitian squares of its conjugate, namely $J\eta\in\Sigma$, since $(\Sigma,\gamma_\ell)$ is of quaternionic type in five Euclidean dimensions. By Proposition \ref{prop:Hermitiansquareconjugate_odd} and Proposition \ref{prop:Bilinearsquareconjugate_odd}, together with Corollary \ref{cor:Hermitian5d} and Corollary \ref{cor:Bilinear5d}, we know that the Hermitian $\widehat{\alpha}_c$ and the complex-bilinear $\alpha_c$ squares of $J\eta\in\Sigma$ are given by:
\begin{equation*}
\widehat{\alpha}_c=\sqrt{\frac{\escal{\omega,\omega}}{2}} + \frac{\ell *(\omega\wedge\omega)}{\sqrt{2\escal{\omega,\omega}}}-i\omega,\qquad \alpha_c=\bar\theta_1\wedge\bar\theta_2,
\end{equation*}

\noindent
where $\theta_1,\theta_2\in V^{\ast}_{\C}$ are complex one-forms satisfying and $\abs{\theta_1}=\abs{\theta_2}=\escal{\theta_1,\theta_2}=0$. Furthermore, by Proposition \ref{prop:compatibilitysquares_odd}, the Hermitian and complex-bilinear squares are the squares of the \emph{same} spinor $\eta\in\Sigma$ if and only if the following algebraic relations are satisfied:
\begin{equation*}
\widehat{\alpha} \vee \beta \vee \alpha   = 4 (\widehat{\alpha} \vee \beta)^{(0)} \alpha \, , \qquad   \alpha \vee \beta \vee \alpha_c = -  4    (\widehat{\alpha} \vee   \tau(\beta))^{(0)} \widehat{\alpha}
\end{equation*}

\noindent
for every $\beta\in \wedge^<V^{\ast}_{\C}$ or, equivalently, a choice of $\beta\in \wedge^< V^{\ast}_{\C}$ such that $(\widehat{\alpha} \vee \beta)^{(0)}  \neq 0$. Taking $\beta = 1$, this system of equations reduces to:
\begin{align*}
\left(\frac{\ell *(\omega\wedge\omega)}{\sqrt{2\escal{\omega,\omega}}}+i\omega\right)\vee(\theta_1\wedge\theta_2)&=3\sqrt{\frac{\escal{\omega,\omega}}{2}}\theta_1\wedge\theta_2,\\
\frac{(\theta_1 \wedge \theta_2) \vee (\bar{\theta}_1 \wedge \bar{\theta}_2)}{ \sqrt{8 \escal{\omega,\omega}}}&=-\sqrt{\frac{\escal{\omega,\omega}}{2}} - \frac{\ell *(\omega\wedge\omega)}{\sqrt{2\escal{\omega,\omega}}} - i\omega.
\end{align*}

We compute:
\begin{align*}
*(\omega\wedge\omega)\vee(\theta_1\wedge\theta_2)&=(*(\omega\wedge\omega))\triangle_1(\theta_1\wedge\theta_2)+\ell(\omega\wedge\omega)\triangle_2(\theta_1\wedge\theta_2),\\
\omega\vee(\theta_1\wedge\theta_2)&=-\escal{\omega,\theta_1\wedge\theta_2}-\ell*(\omega\wedge\theta_1\wedge\theta_2)-\omega\triangle_1(\theta_1\wedge\theta_2),\\
(\theta_1 \wedge \theta_2) \vee (\bar{\theta}_1 \wedge \bar{\theta}_2)&=\escal{\theta_1,\bar\theta_2}\escal{\theta_2,\bar\theta_1}-\escal{\theta_1,\bar\theta_1}\escal{\theta_2,\bar\theta_2}-\ell*(\theta_1\wedge\theta_2\wedge\bar\theta_1\wedge\bar\theta_2)\\
&\quad-\escal{\theta_2,\bar\theta_2}\theta_1\wedge\bar\theta_1+\escal{\theta_2,\bar\theta_1}\theta_1\wedge\bar\theta_2+\escal{\theta_1,\bar\theta_2}\theta_2\wedge\bar\theta_1-\escal{\theta_1,\bar\theta_1}\theta_2\wedge\bar\theta_2.
\end{align*}

Hence, we obtain the following conditions:
\begin{gather*}
\escal{\omega,\theta_1\wedge\theta_2}=0,\qquad \sqrt{2\escal{\omega,\omega}}i*(\omega\wedge\theta_1\wedge\theta_2)=(*(\omega\wedge\omega))\triangle_1(\theta_1\wedge\theta_2),\\
(\omega\wedge\omega)\triangle_2(\theta_1\wedge\theta_2)-\sqrt{2\escal{\omega,\omega}}i\omega\triangle_1(\theta_1\wedge\theta_2)=3\escal{\omega,\omega}\theta_1\wedge\theta_2,\\
2\escal{\omega,\omega}=\escal{\theta_1,\bar\theta_1}\escal{\theta_2,\bar\theta_2}-\vert\escal{\theta_1,\bar\theta_2}\vert^2,\qquad *(\omega\wedge\omega)=\tfrac{1}{2}*(\theta_1\wedge\theta_2\wedge\bar\theta_1\wedge\bar\theta_2),\\
\sqrt{8 \escal{\omega,\omega}}i\omega=\escal{\theta_2,\bar\theta_2}\theta_1\wedge\bar\theta_1-\escal{\theta_2,\bar\theta_1}\theta_1\wedge\bar\theta_2-\escal{\theta_1,\bar\theta_2}\theta_2\wedge\bar\theta_1+\escal{\theta_1,\bar\theta_1}\theta_2\wedge\bar\theta_2.
\end{gather*}

These are solved by:
\begin{equation}
\label{eq:omega_in_terms_of_thetas_5d}
2i\omega=\frac{\escal{\theta_2,\bar\theta_2}\theta_1\wedge\bar\theta_1-\escal{\theta_2,\bar\theta_1}\theta_1\wedge\bar\theta_2-\escal{\theta_1,\bar\theta_2}\theta_2\wedge\bar\theta_1+\escal{\theta_1,\bar\theta_1}\theta_2\wedge\bar\theta_2}{\sqrt{\escal{\theta_1,\bar\theta_1}\escal{\theta_2,\bar\theta_2}-\vert\escal{\theta_1,\bar\theta_2}\vert^2}},
\end{equation}

\noindent
which determines the Hermitian square of $\eta\in\Sigma$ in terms of its complex-bilinear square. Hence, we arrive at the following final characterization of the squares of an irreducible complex spinor in five Euclidean dimensions.

\begin{cor}
A pair of complex forms $\widehat{\alpha}_\eta, \alpha_\eta \in \wedge^< V^{\ast}_{\mathbb{C}}$, and a pair of complex forms $\widehat{\alpha}_{J\eta}, \alpha_{J\eta} \in \wedge^< V^{\ast}_{\mathbb{C}}$ are respectively the Hermitian and complex-bilinear squares of a spinor $\eta\in\Sigma$ and its conjugate $J\eta\in\Sigma$ if and only if:
\begin{gather*}
\widehat{\alpha}_\eta=\frac{1}{2}\sqrt{\escal{\theta_1,\bar\theta_1}\escal{\theta_2,\bar\theta_2}-\vert\escal{\theta_1,\bar\theta_2}\vert^2}+\frac{\ell*(\theta_1\wedge\theta_2\wedge\bar\theta_1\wedge\bar\theta_2)}{2\sqrt{\escal{\theta_1,\bar\theta_1}\escal{\theta_2,\bar\theta_2}-\vert\escal{\theta_1,\bar\theta_2}\vert^2}}+i\omega,\qquad \alpha_\eta=\theta_1\wedge\theta_2,\\
\widehat{\alpha}_{J\eta}=\frac{1}{2}\sqrt{\escal{\theta_1,\bar\theta_1}\escal{\theta_2,\bar\theta_2}-\vert\escal{\theta_1,\bar\theta_2}\vert^2}+\frac{\ell*(\theta_1\wedge\theta_2\wedge\bar\theta_1\wedge\bar\theta_2)}{2\sqrt{\escal{\theta_1,\bar\theta_1}\escal{\theta_2,\bar\theta_2}-\vert\escal{\theta_1,\bar\theta_2}\vert^2}}-i\omega,\qquad \alpha_{J\eta}=\bar\theta_1\wedge\bar\theta_2,
\end{gather*}

\noindent
where $\theta_1 , \theta_2 \in V^*_\C$ are orthogonal and isotropic and $\omega$ is as in Equation \eqref{eq:omega_in_terms_of_thetas_5d}.
\end{cor}


\subsection{The square of an irreducible chiral complex spinor in dimension six}


Let $(V,h)$ be a six-dimensional Euclidean vector space. Its associated real Clifford algebra $\mathrm{Cl}(V^*,h^*)$ is isomorphic to $\Mat(4,\mathbb{H})$ and the irreducible complex Clifford module $(\Sigma,\gamma)$ is of quaternionic type by Lemma \ref{lemma:real_or_quaternionic_type}. We endow $\Sigma$ with an admissible skew-symmetric complex-bilinear pairing $\scB$ of positive adjoint type and with a compatible Hermitian pairing $\scS$ of positive adjoint type, which are related by a compatible quaternionic structure $J\colon\Sigma\to\Sigma$.


\subsubsection{The Hermitian square}


Set $\kappa=1$ for simplicity. By Corollary \ref{cor:sq_S_chiral}, a complex exterior form $\widehat{\alpha}\in\wedge V^*_\C$ is the Hermitian square of a spinor $\eta\in\Sigma$ with chirality $\mu\in\Z_2$ if and only if: $$\widehat{\alpha}\diamond\widehat{\alpha}=8\widehat{\alpha}^{(0)}\widehat{\alpha},\quad\tau(\widehat{\alpha})=\overline{\widehat{\alpha}},\quad\widehat{\alpha}\diamond\beta\diamond\widehat{\alpha}=8(\widehat{\alpha}\diamond\beta)^{(0)}\widehat{\alpha},\quad -i*(\pi\circ\tau)(\widehat{\alpha})=\mu\widehat{\alpha}$$ for a fixed exterior form $\beta\in\wedge V^*_\C$ satisfying $(\widehat{\alpha}\diamond\beta)^{(0)}\neq0$. Let $\widehat{\alpha}=\sum_{k=0}^6\widehat{\alpha}^{(k)}\in\wedge V^*_\C$. The linear condition $\tau(\widehat{\alpha})=\overline{\widehat{\alpha}}$ implies that $\widehat{\alpha}^{(0)}$, $\widehat{\alpha}^{(1)}$, $\widehat{\alpha}^{(4)}$ and $\widehat{\alpha}^{(5)}$ are real and that $\widehat{\alpha}^{(2)}$, $\widehat{\alpha}^{(3)}$ and $\widehat{\alpha}^{(6)}$ are imaginary. The other linear condition $-i*(\pi\circ\tau)(\widehat{\alpha})=\mu\widehat{\alpha}$ implies that $-i*\widehat{\alpha}^{(0)}=\mu\widehat{\alpha}^{(6)}$, $i*\widehat{\alpha}^{(2)}=\mu\widehat{\alpha}^{(4)}$ and $\widehat{\alpha}^{(1)}=\widehat{\alpha}^{(3)}=\widehat{\alpha}^{(5)}=0$. Set $r:=\widehat{\alpha}^{(0)}\in\R$ and $i\omega:=\widehat{\alpha}^{(2)}$ for $\omega\in\wedge^2V^*$. Hence:
\begin{equation*}
\widehat{\alpha}=r+i\omega-\mu*\omega-i\mu r\nu=(r+i\omega)\diamond(1-i\mu\nu).
\end{equation*}

\noindent
Now we compute: $$\widehat{\alpha}\diamond\widehat{\alpha}=2r^2+4ir\omega-4\mu r*\omega-2i\mu r^2\nu-2\omega\wedge\omega+2\escal{\omega,\omega}+2i\mu*(\omega\wedge\omega)-2i\mu\escal{\omega,\omega}\nu.$$

Then the equation $\widehat{\alpha}\diamond\widehat{\alpha}=8\widehat{\alpha}^{(0)}\widehat{\alpha}$ reduces to: $$\escal{\omega,\omega}=3r^2,\quad\omega\wedge\omega=2\mu r*\omega.$$

\begin{cor}
\label{cor:Hermitian6d}
A complex exterior form $\widehat{\alpha} \in\wedge V^*_\C$ is the Hermitian square of a complex irreducible spinor $\eta\in\Sigma^{\mu}$ with chirality $\mu$ if and only if:
\begin{equation*}
\widehat{\alpha} = \sqrt{\frac{\escal{\omega,\omega}}{3}} + i\omega-\mu*\omega-i\mu\sqrt{\frac{\escal{\omega,\omega}}{3}}\nu
\end{equation*}

\noindent
for a uniquely determined real two-form $\omega\in\wedge^2V^*$ satisfying $\sqrt{\escal{\omega,\omega}}*\omega=\mu\frac{\sqrt{3}}{2}\omega\wedge\omega$.
\end{cor}

Given an irreducible complex spinor $\eta\in \Sigma$, consider now the following equation:
\begin{equation*}
F\cdot \eta := \Psi_{\gamma}(F)(\eta) = 0
\end{equation*}

\noindent
for a complex two-form $F\in \wedge^2 V_\C^{\ast}$. By Lemma \ref{lemma:constrainedspinoreven}, a two-form $F$ satisfies the previous equation if and only if: 
\begin{equation*}
F\diamond\widehat{\alpha}=F\diamond(r+i\omega-\mu*\omega-i\mu r\nu)=0.
\end{equation*}

\noindent
Separating by degree, this equation is equivalent to: 
\begin{equation*}
\mu \sqrt{\frac{\escal{\omega,\omega}}{3}} \ast F + F\wedge\omega = i\mu*(F\triangle_1\omega) ,\qquad \escal{F,\omega}=0.    
\end{equation*}
 
\noindent
Now consider a complex three-form $H\in\wedge^3V_\C^*$. It satisfies the equation $H\cdot\eta=0$ if and only if: 
\begin{equation*}
H\diamond\widehat{\alpha}=H\diamond(r+i\omega-\mu*\omega-i\mu r\nu)=0.
\end{equation*}

\noindent
Separating by degree, this equation is equivalent to: 
\begin{eqnarray*}
(H+i\mu \ast H )\wedge\omega=0,\quad \mu \sqrt{\frac{\escal{\omega,\omega}}{3}} \ast H + H\triangle_1\omega+i  ( \mu \ast( H\triangle_1\omega) - \sqrt{\frac{\escal{\omega,\omega}}{3}} H)=0.
\end{eqnarray*}

\noindent
The bilinear operations $\triangle_1$ and $\triangle_2$ are defined in Appendix \ref{app:KA}.


\subsubsection{The complex-bilinear square}


Let $\mu\in\mathbb{Z}_2$. By Corollary \ref{cor:sq_B_chiral}, an exterior form $\alpha\in\wedge V^*_\C$ is the complex-bilinear square of a spinor $\eta\in\Sigma$ with chirality $\mu$ if and only if:
\begin{equation*}
\alpha\diamond\alpha=8\alpha^{(0)}\alpha,\quad\tau(\alpha) =-\alpha\, , \quad \alpha \diamond \beta \diamond \alpha = 8 (\alpha\diamond\beta)^{(0)} \alpha,\quad-i*(\pi\circ\tau)(\alpha)=\mu\alpha  
\end{equation*}

\noindent
for a fixed exterior form $\beta\in\wedge V^*_\C$ satisfying $(\alpha\diamond\beta)^{(0)}\neq0$. Let $\alpha=\sum_{k=0}^6\alpha^{(k)}\in\wedge V^*_\C$. The linear condition $\tau(\alpha)=-\alpha$ implies that $\alpha^{(0)}=\alpha^{(1)}=\alpha^{(4)}=\alpha^{(5)}=0$ and the other linear condition $-i*(\pi\circ\tau)(\alpha)=\mu\alpha$ implies that $\alpha^{(2)}=\alpha^{(6)}=0$ and $*\alpha^{(3)}=i\mu\alpha^{(3)}$. Set $\rho:=\alpha^{(3)}$. The quadratic equation $\alpha\diamond\alpha=8\alpha^{(0)}\alpha$ is then:
\begin{equation*}
\rho\diamond\rho=\rho\triangle_1\rho-\escal{\rho,\rho}=0,
\end{equation*}

\noindent
thus $\rho\triangle_1\rho=0$ and $\escal{\rho,\rho}=0$. Note that the condition $*\rho=i\mu\rho$ already implies that $\rho\diamond\rho=0$. Indeed, on one hand we have: 
\begin{equation*}
*\rho\diamond*\rho=(i\mu\rho)\diamond(i\mu\rho)=-\rho\diamond\rho.
\end{equation*}

\noindent
On the other hand, using Lemma \ref{lemma:product_volume_form} and $\nu\diamond\nu=-1$ we have: 
\begin{equation*}
*\rho\diamond*\rho=-*\tau(\rho)\diamond*(\pi\circ\tau)(\rho)=-\rho\diamond\nu\diamond\nu\diamond\rho=\rho\diamond\rho.
\end{equation*}

\noindent
Since $\alpha^{(0)}=0$, taking $\beta=1$ does not suffice to characterize the square of the spinor. We consider $\beta = \bar \rho \in \wedge^3V^*_\C$. Then:
\begin{equation*}
\alpha\diamond\beta=\rho\diamond\bar\rho=\rho\wedge\bar\rho+\rho\triangle_1\bar\rho-\rho\triangle_2\bar\rho-\escal{\rho,\bar\rho}  
\end{equation*}

\noindent
and $(\alpha\diamond\beta)^{(0)}=-\escal{\rho,\bar\rho}=-\abs{\Re(\rho)}^2-\abs{\Im(\rho)}^2\neq0$ since $\rho\neq 0$.

\begin{cor}
\label{cor:6dbilinear}
A complex exterior form $\alpha\in\wedge V^*_\C$ is the complex-bilinear square of a complex irreducible spinor $\eta\in\Sigma^\mu$ with chirality $\mu$ if and only if: 
\begin{equation*}
\alpha=\theta_1\wedge\theta_2\wedge\theta_3    
\end{equation*}

\noindent
for complex one-forms $\theta_1,\theta_2,\theta_3\in V^*_\C$ isotropic and orthogonal satisfying: 
\begin{equation*}
*(\theta_1\wedge\theta_2\wedge\theta_3)=i\mu(\theta_1\wedge\theta_2\wedge\theta_3).
\end{equation*}
\end{cor}

\begin{proof}
By the previous discussion we know that necessarily $\alpha=\rho\in\wedge^3V^*_\C$ satisfying $*\rho=i\mu\rho$. Let us show that $\rho$ is decomposable. Let $v\in V_\C$ be any complex vector and take $\beta=v^\flat\in V^*_\C$. Then $(\alpha\diamond\beta)^{(0)}=(\rho\diamond v^\flat)^{(0)}=0$ and the quadratic equation $\alpha\diamond\beta\diamond\alpha=8(\alpha\diamond\beta)^{(0)}\alpha$ becomes: 
\begin{equation*}
(\iota_v\rho)\wedge\rho-(\iota_v\rho)\triangle_1\rho-(\iota_v\rho)\triangle_2\rho=0,  
\end{equation*}

\noindent
where we have used $\rho\diamond v^\flat+v^\flat\diamond\rho=2\iota_v\rho$ and $\rho\diamond\rho=0$. In particular, we have $(\iota_v\rho)\wedge\rho=0$ for all $v\in V_\C$, hence $\rho$ is decomposable by classical Plücker relations \cite{EM00}. Hence, we can write:
\begin{equation*}
\rho = \theta_1 \wedge \theta_2 \wedge \theta_3
\end{equation*}

\noindent
for complex one-forms $\theta_1 , \theta_2 , \theta_3 \in V^{\ast}_{\C}$. Then, the condition $\ast\rho = i\mu \rho$ translates into $\ast (\theta_1 \wedge \theta_2 \wedge \theta_3) = i \mu (\theta_1 \wedge \theta_2 \wedge \theta_3)$. Plugging $\theta_1^{\sharp}$, $ \theta_2^{\sharp}$, and $\theta_3^{\sharp}$ into the latter, we conclude that the one-forms are isotropic and mutually orthogonal. We choose a basis:
\begin{equation*}
\{\theta_1,\theta_2,\theta_3,\bar\theta_1,\bar\theta_2,\bar\theta_3\}  
\end{equation*}

\noindent
of $(V^*_\C,h^*_\C)$ given by isotropic one-forms that are conjugate in pairs, that is, they are mutually orthogonal except for:
\begin{equation*}
\escal{\theta_1,\bar\theta_1}=\escal{\theta_2,\bar\theta_2}=\escal{\theta_3,\bar\theta_3}=1.    
\end{equation*}

\noindent 
Then, setting $\beta = \bar\rho = \bar{\theta}_1 \wedge \bar{\theta}_2 \wedge \bar{\theta}_3$ we have $(\alpha\diamond\beta)^{(0)}=-\escal{\rho,\bar\rho}=-1$ and a tedious but straightforward computation shows that the equation $\alpha \diamond \beta \diamond \alpha = -8 \alpha$ is automatically satisfied by $\alpha$ as given in the statement.
\end{proof}

\noindent
As in the previous cases, the complex-bilinear square of an irreducible and chiral complex spinor, as obtained in the corollary above, defines naturally a $\SU(3)$-structure. We then recover the following well-known correspondence.  

\begin{cor}
There is a one-to-one correspondence between $\SU(3)$-structures and irreducible and chiral complex spinors on a six-dimensional Euclidean vector space $(V,h)$.
\end{cor}


\subsubsection{Compatibility conditions and the conjugate square}


The goal of this subsection is to characterize the complex-bilinear and Hermitian squares of the \emph{same} irreducible complex and chiral spinor $\eta\in\Sigma^\mu$, and the complex-bilinear and Hermitian squares of its conjugate, namely $J\eta\in\Sigma^{-\mu}$ since $(\Sigma,\gamma)$ is of quaternionic type in six Euclidean dimensions. By Proposition \ref{prop:Hermitiansquareconjugate} and Proposition \ref{prop:Bilinearsquareconjugate}, together with Corollary \ref{cor:Hermitian6d} and Corollary \ref{cor:6dbilinear}, we know that the Hermitian $\widehat{\alpha}_c$ and complex-bilinear $\alpha_c$ squares of $J\eta\in\Sigma^{-\mu}$ are given by:
\begin{equation*}
\widehat{\alpha}_c = \sqrt{\frac{\escal{\omega,\omega}}{3}} - i\omega - \mu*\omega + i\mu\sqrt{\frac{\escal{\omega,\omega}}{3}}\nu,\qquad \alpha_c = \bar\theta_1\wedge\bar\theta_2\wedge\bar\theta_3,
\end{equation*}

\noindent
where $\theta_1,\theta_2,\theta_3\in V^*_\C$ are isotropic and orthogonal complex one-forms satisfying the duality condition $*(\theta_1\wedge\theta_2\wedge\theta_3)=i\mu(\theta_1\wedge\theta_2\wedge\theta_3)$. Furthermore, by Proposition \ref{prop:compatibilitysquares}, the Hermitian and complex-bilinear squares are the squares of the \emph{same} spinor $\eta\in\Sigma^{\mu}$ if and only if the following algebraic relations are satisfied:
\begin{equation*}  
\widehat{\alpha} \diamond \beta \diamond \alpha   = 8 (\widehat{\alpha} \diamond \beta)^{(0)} \alpha \, , \qquad   \alpha \diamond \beta \diamond \alpha_c = -  8    (\widehat{\alpha} \diamond   \tau(\beta))^{(0)} \widehat{\alpha}
\end{equation*}

\noindent
for every $\beta\in \wedge V^{\ast}_{\C}$ or, equivalently, a choice of $\beta\in \wedge V^{\ast}_{\C}$ such that $(\widehat{\alpha} \diamond \beta)^{(0)}  \neq 0$. Taking $\beta=1$ is sufficient since $\widehat{\alpha}^{(0)}=3^{-\frac{1}{2}}\sqrt{\escal{\omega,\omega}}\neq0$. We compute:
\begin{equation*}
\alpha\diamond\alpha_c=\alpha\wedge\alpha_c+\alpha\triangle_1\alpha_c-\alpha\triangle_2\alpha_c-\escal{\alpha,\alpha_c}
\end{equation*}

\noindent
and the equation $\alpha\diamond\alpha_c=-8\widehat{\alpha}^{(0)}\widehat{\alpha}$ gives us:
\begin{equation*}
8\escal{\omega,\omega}=3\escal{\alpha,\alpha_c},\qquad i\omega=\frac{\alpha\triangle_2\alpha_c}{\sqrt{8\escal{\alpha,\alpha_c}}},\qquad \mu*\omega=\frac{\alpha\triangle_1\alpha_c}{\sqrt{8\escal{\alpha,\alpha_c}}},\qquad i\mu\nu=\frac{\alpha\wedge\alpha_c}{\escal{\alpha,\alpha_c}}.
\end{equation*}

These equations determine the Hermitian square of $\eta\in\Sigma^\mu$ in terms of its complex-bilinear square. Hence, we arrive at the following final characterization of the squares of an irreducible and chiral complex spinor in six Euclidean dimensions.

\begin{cor}
A pair of complex forms $\widehat{\alpha}_\eta, \alpha_\eta \in \wedge V^{\ast}_{\mathbb{C}}$, and a pair of complex forms $\widehat{\alpha}_{J\eta}, \alpha_{J\eta} \in \wedge V^{\ast}_{\mathbb{C}}$ are respectively the Hermitian and complex-bilinear squares of a spinor $\eta\in\Sigma^\mu$ and its conjugate $J\eta\in\Sigma^{-\mu}$ if and only if:
\begin{gather*}
\widehat{\alpha}_\eta=\sqrt{\frac{\escal{\alpha_\eta,\alpha_{J\eta}}}{8}}+\frac{\alpha_\eta\triangle_2\alpha_{J\eta}}{\sqrt{8\escal{\alpha_\eta,\alpha_{J\eta}}}}-\frac{\alpha_\eta\triangle_1\alpha_{J\eta}}{\sqrt{8\escal{\alpha_\eta,\alpha_{J\eta}}}}-\frac{\alpha_\eta\wedge\alpha_{J\eta}}{\sqrt{8\escal{\alpha_\eta,\alpha_{J\eta}}}},\qquad \alpha_\eta=\theta_1\wedge\theta_2\wedge\theta_3,\\
\widehat{\alpha}_{J\eta}=\sqrt{\frac{\escal{\alpha_\eta,\alpha_{J\eta}}}{8}}-\frac{\alpha_\eta\triangle_2\alpha_{J\eta}}{\sqrt{8\escal{\alpha_\eta,\alpha_{J\eta}}}}-\frac{\alpha_\eta\triangle_1\alpha_{J\eta}}{\sqrt{8\escal{\alpha_\eta,\alpha_{J\eta}}}}+\frac{\alpha_\eta\wedge\alpha_{J\eta}}{\sqrt{8\escal{\alpha_\eta,\alpha_{J\eta}}}},\qquad \alpha_{J\eta}=\bar\theta_1\wedge\bar\theta_2\wedge\bar\theta_3,
\end{gather*}

\noindent
where $\theta_1,\theta_2,\theta_3\in V^*_\C$ are orthogonal and isotropic.
\end{cor}

We conclude by summarizing the results of this section in Table \ref{tab:spinor_squares_plain}.

\begin{table}[!ht]
\centering
\caption{Complex-bilinear and Hermitian square of an irreducible complex spinor (chiral when possible) in Euclidean dimensions two to six}
\label{tab:spinor_squares_plain}
\renewcommand{\arraystretch}{1.8} 
\begin{tabular}{|c|p{6cm}|p{6cm}|}
\hline
\emph{Dimension} & \emph{Complex-bilinear square $\alpha$} & \emph{Hermitian square $\widehat{\alpha}$} \\
\hline
$d=2$ & 
$\alpha=\theta\in V^*_\C$ satisfying $*\theta=i\mu\theta$. & 
$\widehat{\alpha}=r(1+i\mu\nu)$, where $r\in\R\setminus\{0\}$.\\
\hline
$d=3$ & 
$\alpha=\theta\in V^*_\C$ satisfying $\escal{\theta,\theta}=0$. & 
$\widehat{\alpha}=\sqrt{\escal{\vartheta,\vartheta}}+\vartheta$, where $\vartheta\in V^*$.\\
\hline
$d=4$ & 
$\alpha=\theta_1\wedge\theta_2\in \wedge^2V^*_\C$ satisfying& 
$\widehat{\alpha}=\sqrt{\frac{ \escal{\omega,\omega}}{2}} + i\omega - \frac{\omega\wedge \omega}{\sqrt{2\escal{\omega,\omega}}},$\\
&$*(\theta_1\wedge\theta_2)=\mu(\theta_1\wedge\theta_2)$.&where $\omega\in\wedge^2V^*$ satisfies $*\omega=\mu\omega$.\\
\hline
$d=5$ & 
$\alpha=\theta_1\wedge\theta_2\in \wedge^2V^*_\C$ satisfying& 
$\widehat{\alpha} = \sqrt{\frac{\escal{\omega,\omega}}{2}} + \frac{\ell *(\omega\wedge\omega)}{\sqrt{2\escal{\omega,\omega}}}+i\omega,$\\
&$\abs{\theta_1}=\abs{\theta_2}=\escal{\theta_1,\theta_2}=0$.&where $\omega\in\wedge^2V^*$ satisfies $\escal{\omega,\omega} \ast \omega = \omega\wedge*(\omega\wedge\omega)$.\\
\hline
$d=6$ & 
$\alpha=\theta_1\wedge\theta_2\wedge\theta_3\in \wedge^3V^*_\C$ satisfying& 
$\widehat{\alpha} = \sqrt{\frac{\escal{\omega,\omega}}{3}} + i\omega-\mu*\omega-i\mu\sqrt{\frac{\escal{\omega,\omega}}{3}}\nu,$\\
&$*(\theta_1\wedge\theta_2\wedge\theta_3)=i\mu(\theta_1\wedge\theta_2\wedge\theta_3)$.&where $\omega\in\wedge^2V^*$ satisfies $\sqrt{\escal{\omega,\omega}}*\omega=\mu\frac{\sqrt{3}}{2}\omega\wedge\omega$.\\
\hline
\end{tabular}
\end{table}


\section{Complex chiral spinors on a Riemannian eight-manifold}
\label{section:8dRiemannian}


In this section, we examine the squares of an irreducible complex chiral spinor in eight Euclidean dimensions. This case is particularly important because of its potential geometric applications in the theory of $\Spin(7)$ and $\SU(4)$-structures and the fact that this is the lowest dimension for which \emph{impure} chiral spinors can exist. As we will see below, this is neatly reflected in the algebraic structure of the corresponding squares. Prior results on the square of a normalized impure spinor in an explicit basis have been obtained in \cite{Bhoja:2023dgk,Krasnov2024}. Throughout this section, we fix an oriented eight-dimensional Euclidean vector space $(V,h)$ and an irreducible complex Clifford module $(\Sigma,\gamma)$ with chiral decomposition $\Sigma=\Sigma^{+}\oplus\Sigma^{-}$ in terms of eigenspaces of the complex volume form $\nu_{\C} = \nu$. Since $(V,h)$ is an eight-dimensional Euclidean vector space, its associated real Clifford algebra $\mathrm{Cl}(V^*,h^*)$ is isomorphic to $\Mat(16,\R)$ and the irreducible complex Clifford module $(\Sigma,\gamma)$ is of real type by Lemma \ref{lemma:real_or_quaternionic_type}. We endow $\Sigma$ with compatible admissible pairings $\scS$ and $\scB$ of positive symmetry and adjoint type, which are related by a compatible real structure $\mathfrak{c}\colon\Sigma\to\Sigma$. Note that the chiral splitting of $\Sigma$ is orthogonal with respect to $\scS$ and $\scB$.


\subsection{The Hermitian square}


We proceed to compute first the Hermitian square of $\eta \in \Sigma^{\mu}$. Set $\kappa=1$ for simplicity. By Corollary \ref{cor:sq_S_chiral}, a complex exterior form $\widehat{\alpha}\in\wedge V^*_\C$ is the Hermitian square of a spinor $\eta\in\Sigma^{\mu}$ with chirality $\mu\in\Z_2$ if and only if:
\begin{equation*}
\widehat{\alpha}\diamond\widehat{\alpha}=16\widehat{\alpha}^{(0)}\widehat{\alpha},\qquad \tau(\widehat{\alpha})=\overline{\widehat{\alpha}},\qquad *(\pi\circ\tau)(\widehat{\alpha})=\mu\widehat{\alpha}.   
\end{equation*}

\noindent
We expand $\widehat{\alpha}$ by degree as follows:
\begin{equation*}
\widehat{\alpha}=\sum_{k=0}^8\widehat{\alpha}^{(k)}\in\wedge V^*_\C.
\end{equation*}

\noindent
The linear condition $\tau(\widehat{\alpha})=\overline{\widehat{\alpha}}$ implies that $\widehat{\alpha}^{(0)}$, $\widehat{\alpha}^{(1)}$, $\widehat{\alpha}^{(4)}$, $\widehat{\alpha}^{(5)}$, $\widehat{\alpha}^{(8)}$ are real and that $\widehat{\alpha}^{(2)}$, $\widehat{\alpha}^{(3)}$, $\widehat{\alpha}^{(6)}$, $\widehat{\alpha}^{(7)}$ are imaginary. The other linear condition $*(\pi\circ\tau)(\widehat{\alpha})=\mu\widehat{\alpha}$ implies that $\widehat{\alpha}^{(1)}=\widehat{\alpha}^{(3)}=\widehat{\alpha}^{(5)}=\widehat{\alpha}^{(7)}=0$, $\widehat{\alpha}^{(8)}=\mu*\widehat{\alpha}^{(0)}$, $\widehat{\alpha}^{(6)}=-\mu*\widehat{\alpha}^{(2)}$ and $*\widehat{\alpha}^{(4)}=\mu\widehat{\alpha}^{(4)}$. Set $r:=\widehat{\alpha}^{(0)}\in\R$, $i\omega:=\widehat{\alpha}^{(2)}$ for $\omega\in\wedge^2V^*$, and $\Theta:=\widehat{\alpha}^{(4)}\in\wedge^4V^*$. The term $\widehat{\alpha}^{(8)}$ is a multiple of the volume form $\nu$, that is $\widehat{\alpha}^{(8)}=c\nu$ for some $c\in\R$. From the equation $c\nu=\mu*r=\mu r*1=\mu r\nu$ we get $c=\mu r$. Hence: 
\begin{equation*}
\widehat{\alpha}=r+i\omega+\Theta-i\mu*\omega+\mu r\nu.
\end{equation*}

\noindent
Using the properties of the geometric product together with Lemma \ref{lemma:product_volume_form} we compute: 
\begin{align*}
\widehat{\alpha}\diamond\widehat{\alpha}&=2r^2+4ir\omega+4r\Theta-4i\mu r*\omega+2\mu r^2\nu\\
&\quad-2\omega\wedge\omega+2\escal{\omega,\omega}-2\mu*(\omega\wedge\omega)+2\mu\escal{\omega,\omega}\nu\\
&\quad+4i\omega\wedge\Theta-4i\omega\triangle_2\Theta+\Theta\wedge\Theta-\Theta\triangle_2\Theta+\escal{\Theta,\Theta}.
\end{align*}

\noindent
Then, the quadratic equation $\widehat{\alpha}\diamond\widehat{\alpha}=16\widehat{\alpha}^{(0)}\widehat{\alpha}$ reduces to:
\begin{equation*}
2\escal{\omega,\omega}+\escal{\Theta,\Theta}=14r^2,\quad 2\omega\wedge\omega+2\mu*(\omega\wedge\omega)+\Theta\triangle_2\Theta+12r\Theta=0,\quad \omega\wedge\Theta+3\mu r*\omega=0.    
\end{equation*}

\noindent
This leads to the following algebraic characterization of the Hermitian square of $\eta$.

\begin{cor}
\label{cor:Hermitian8dimpure}
A complex exterior form $\widehat{\alpha} \in\wedge V^*_\C$ is the Hermitian square of a complex irreducible spinor $\eta\in\Sigma^{\mu}$ with chirality $\mu$ if and only if:
\begin{equation*}
\widehat{\alpha} = \sqrt{\frac{2\escal{\omega,\omega}+\escal{\Theta,\Theta}}{14}} + i\omega + \Theta-i\mu*\omega +  \sqrt{\frac{2\escal{\omega,\omega}+\escal{\Theta,\Theta}}{14 \escal{\Theta,\Theta}^2 }}\Theta \wedge \Theta
\end{equation*}

\noindent
for a real two-form $\omega\in\wedge^2V^*$ and a real four-form $\Theta\in\wedge^4V^*$ satisfying $*\Theta=\mu\Theta$ and the following algebraic system:
\begin{equation*}
\begin{gathered}
2\omega\wedge\omega+2\mu*(\omega\wedge\omega)+\Theta\triangle_2\Theta+12 \sqrt{\frac{2\escal{\omega,\omega}+\escal{\Theta,\Theta}}{14}}\Theta=0,\\
\omega\wedge\Theta+3\mu \sqrt{\frac{2\escal{\omega,\omega}+\escal{\Theta,\Theta}}{14}}*\omega=0.
\end{gathered}
\end{equation*} 
\end{cor}

\noindent
Note that if $\omega=0$, then we recover the algebraic conditions obtained in \cite[Theorem 3.22]{LS24_Spin7} to characterize a conformal $\Spin(7)$-structure on $(V,h)$. Hence, the case $\omega = 0$ corresponds to the case for which $\widehat{\alpha}$ is the Hermitian square of a possibly complex multiple of an irreducible chiral real spinor in eight Euclidean dimensions \cite{Spin89}. On the other hand, if we set $\Theta = 0$, then $\omega = 0$ and therefore we conclude that $\Theta\neq 0$ necessarily. In order to obtain another type of solution to the previous equations for the square of $\eta$, consider $\omega\in\wedge^2V^*$ to be a Kähler form on $(V,h)$, not necessarily normalized. Then the following identities hold: 
\begin{equation*}
*(\omega\wedge\omega)=\omega\wedge\omega,\quad *\omega=\frac{2}{3\escal{\omega,\omega}}\omega\wedge\omega\wedge\omega,\quad(\omega\wedge\omega)\triangle_2(\omega\wedge\omega)=2\escal{\omega,\omega}\omega\wedge\omega.    
\end{equation*}

\noindent
Set: 
\begin{equation*}
\Theta = -\frac{\omega\wedge\omega }{\escal{\omega,\omega}^{\frac{1}{2}}}.
\end{equation*}

\noindent
Then the complex exterior form: 
\begin{equation*}
\widehat{\alpha}=\frac{1}{2}\escal{\omega,\omega}^{\frac{1}{2}}+i\omega -\frac{\omega\wedge\omega }{\escal{\omega,\omega}^{\frac{1}{2}}} -i*\omega+\frac{1}{2}\escal{\omega,\omega}^{\frac{1}{2}}\nu    
\end{equation*}

\noindent
is the Hermitian square of a complex irreducible spinor $\eta\in\Sigma^+$. This corresponds to the Hermitian square of a pure irreducible and chiral complex spinor, and gives a precise formula in eight dimensions for the statement made in \cite{Wang89} regarding the structure of the Hermitian square of an irreducible and chiral pure spinor.\medskip

Given an irreducible complex spinor $\eta\in \Sigma^{\mu}$, we consider now the following equation:
\begin{equation*}
F\cdot \eta = 0
\end{equation*}

\noindent
for a complex two-form $F\in \wedge^2 V^{\ast}_{\C}$. By Lemma \ref{lemma:constrainedspinoreven}, this equation is equivalent to:
\begin{align*}
F\diamond \widehat{\alpha} & = F\diamond (r+i\omega+\Theta-i\mu*\omega+\mu r\nu ) \\
& = rF + i F\diamond \omega + F\diamond \Theta + i\mu F\diamond \omega\diamond \nu - \mu r \ast F\\
& = r F + i F\wedge \omega - i F \triangle_1 \omega - i \langle F , \omega \rangle + F\wedge \Theta - F\triangle_1 \Theta - F\triangle_2 \Theta \\
&\quad + i \mu \ast (F\wedge \omega) + i\mu \ast (F\triangle_1\omega) - i \mu \langle F , \omega \rangle \nu - \mu r\ast F = 0.
\end{align*}

\noindent
Isolating by degree, this equation reduces to:
\begin{equation}
\label{eq:algebraicF8d}
\begin{gathered}
\langle F , \omega \rangle = 0\, , \qquad   14^{-\frac{1}{2}}\sqrt{ 2\escal{\omega,\omega}+\escal{\Theta,\Theta}} \, F  = i F \triangle_1 \omega  + F\triangle_2 \Theta,\\
F\wedge \omega  + i F\triangle_1 \Theta  +  \mu \ast (F\wedge \omega)   = 0\, , \qquad 14^{-\frac{1}{2}}\sqrt{ 2\escal{\omega,\omega}+\escal{\Theta,\Theta}} \,\ast F = \mu F\wedge \Theta + i \ast (F\triangle_1\omega),
\end{gathered}
\end{equation}

\noindent
which we will interpret in the introduction as the condition characterizing a \emph{spinorial instanton} in eight Euclidean dimensions.\medskip

Given an irreducible complex spinor $\eta\in \Sigma^{\mu}$, we consider now the following equation:
\begin{equation*}
H \cdot \eta = 0
\end{equation*}

\noindent
for a complex three-form $H\in \wedge^3 V^{\ast}_{\C}$. By Lemma \ref{lemma:constrainedspinoreven}, this equation is equivalent to:
\begin{align*}
H \diamond \widehat{\alpha} & = H \diamond (r+i\omega+\Theta-i\mu*\omega+\mu r\nu )\\
& = r H + i H \diamond \omega + H \diamond \Theta + i\mu H \diamond \omega\diamond \nu - \mu r \ast H\\
& = r H + i H\wedge \omega +  i H \triangle_1 \omega - i H \triangle_2 \omega + H \wedge \Theta + H \triangle_1 \Theta - H\triangle_2 \Theta - H \triangle_3 \Theta \\
&\quad  + i \mu \ast (H\wedge \omega) - i \mu \ast ( H \triangle_1 \omega ) - i \mu \ast ( H \triangle_2 \omega) - \mu r\ast H = 0.
\end{align*}

\noindent
Isolating by degree, this equation reduces to:
\begin{equation}
\label{eq:algebraicH8d}
\begin{gathered}
i H \triangle_2 \omega + H \triangle_3 \Theta = 0,\qquad H \wedge \Theta   = i \mu \ast ( H \triangle_2 \omega),\\
14^{-\frac{1}{2}}\sqrt{ 2\escal{\omega,\omega}+\escal{\Theta,\Theta}} H +  i H \triangle_1 \omega    - H\triangle_2 \Theta   + i \mu \ast (H\wedge \omega) = 0,\\
14^{-\frac{1}{2}} \mu \sqrt{ 2\escal{\omega,\omega}+\escal{\Theta,\Theta}} \ast H = i H\wedge \omega + H \triangle_1 \Theta - i \mu \ast ( H \triangle_1 \omega ).
\end{gathered}
\end{equation}

\noindent
As explained in the introduction, these conditions can be interpreted as characterizing a \emph{spinorial curving} on a $\C^{\ast}$-bundle gerbe defined on an eight-dimensional Riemannian manifold.


\subsection{The complex-bilinear square}


Let $\mu\in\mathbb{Z}_2$. By Corollary \ref{cor:sq_B_chiral}, a complex exterior form $\alpha\in\wedge V^*_\C$ is the complex-bilinear square of a spinor $\eta\in\Sigma$ with chirality $\mu$ if and only if:
\begin{equation*}
\alpha\diamond\alpha=16\alpha^{(0)}\alpha,\quad\tau(\alpha)=\alpha,\quad\alpha\diamond\beta\diamond\alpha=16(\alpha\diamond\beta)^{(0)}\alpha,\quad *(\pi\circ\tau)(\alpha)=\mu\alpha
\end{equation*}

\noindent
for a fixed exterior form $\beta\in\wedge V^*_\C$ satisfying $(\alpha\diamond\beta)^{(0)}\neq0$. We expand $\alpha$ by degree as follows:
\begin{equation*}
\alpha=\sum_{k=0}^8\alpha^{(k)}\in\wedge V^*_\C.
\end{equation*}

\noindent
The linear condition $\tau(\alpha)=\alpha$ implies that $\alpha^{(2)} = \alpha^{(3)} = \alpha^{(6)} = \alpha^{(7)} = 0$ and the remaining condition $*(\pi\circ\tau)(\alpha)=\mu\alpha$ implies that $\alpha^{(1)}=\alpha^{(5)}=0$, $*\alpha^{(0)}=\mu\alpha^{(8)}$ and $*\alpha^{(4)}=\mu\alpha^{(4)}$. Set $\lambda := \alpha^{(0)}$ and $\Omega:=\alpha^{(4)}$. The term $\alpha^{(8)}$ is a multiple of the volume form $\nu$, that is $\alpha^{(8)}=c\nu$ for some $c\in\C$. From $*\lambda=\mu c\nu$ we get $c^2=\lambda^2$. Now we compute:
\begin{align*}
\alpha\diamond\alpha&=(\lambda + \Omega+c\nu)\diamond(\lambda + \Omega + c\nu)\\
&=2\lambda^2+\escal{\Omega,\Omega}+(2\lambda+2c\mu)\Omega-\Omega\triangle_2\Omega+2 \lambda c\nu+\Omega\wedge\Omega.
\end{align*}

\noindent
The quadratic equation $\alpha\diamond\alpha=16\alpha^{(0)}\alpha$ holds if and only if:
\begin{equation*}
\escal{\Omega,\Omega}=14\lambda^2,\qquad(14\lambda - 2c\mu)\Omega+\Omega\triangle_2\Omega=0,\qquad\Omega\wedge\Omega=14 \lambda c\nu.   
\end{equation*}

\noindent
The first and third equations imply that $c=\mu \lambda$, and then they become equivalent. Hence, the above system of equations reduces to: 
\begin{equation}
\label{eq:Spin(7)_type_equation}
\escal{\Omega,\Omega}=14\lambda^2\quad\text{and}\quad 12\lambda\Omega+\Omega\triangle_2\Omega=0.
\end{equation}

\noindent
Hence, we arrive at the following result.

\begin{cor}
\label{cor:Bilinear8d}
A complex exterior form $\alpha\in\wedge V^*_\C$ is the complex-bilinear square of a complex irreducible spinor $\eta\in\Sigma^\mu$ with chirality $\mu$ if and only if:
\begin{equation}
\label{eq:alphaSU(4)}
\alpha =  \frac{\langle \Omega , \Omega \rangle^{\frac{1}{2}}}{\sqrt{14}}  + \Omega + \frac{\mu\langle \Omega , \Omega \rangle^{\frac{1}{2}}}{\sqrt{14}}   \nu
\end{equation}

\noindent
for a complex four-form $\Omega \in \wedge^4 V^{\ast}_{\C}$ satisfying $*\Omega=\mu\Omega$ and:
\begin{equation}
\label{eq:SU(4)invariance}
\sqrt{14}\Omega\triangle_2\Omega + 12 \langle \Omega , \Omega \rangle^{\frac{1}{2}} \Omega = 0.
\end{equation}

\noindent
Furthermore, if $\scB(\eta,\eta)\neq0$ these conditions are also sufficient.
\end{cor}

\begin{remark}
Note that $\scB(\eta,\eta)\neq0$ if and only if $\escal{\Omega,\Omega}\neq0$.
\end{remark}

\noindent
Therefore, we formally obtain the very same algebraic equation that was obtained in \cite[Theorem 4.5]{LS24_Spin7} to characterize the square of an irreducible and chiral real spinor in eight Euclidean dimensions, albeit for a complex four-form instead of a real four-form. This leads to the following algebraic characterization of complex four-forms stabilized by $\SU(4)\subset \mathrm{GL}(8,\mathbb{R})$.

\begin{prop}
A complex four-form $\Omega\in \wedge^4 V^{\ast}_{\C}$ is stabilized by a group isomorphic to $\SU(4)\subset \mathrm{GL}(8,\mathbb{R})$ if and only if there exists a metric $h$ on $V$ such that $*\Omega=\mu\Omega$ and the pair $(\Omega,h)$ satisfies Equation \eqref{eq:SU(4)invariance}.
\end{prop}

\begin{proof}
Suppose that $(\Omega,h)$ satisfies Equation \eqref{eq:SU(4)invariance} and $*\Omega=\mu\Omega$. Then, by Corollary \ref{cor:sq_B_chiral} together with Corollary \ref{cor:Bilinear8d}, it follows that the exterior form $\alpha$ given as in Equation \eqref{eq:alphaSU(4)} in terms of $(\Omega,h)$ is the square of a complex chiral irreducible spinor $\eta\in \Sigma^{\mu}$. Since the latter is stabilized by $\SU(4)\subset \Spin(V,h)$, which is connected and simply connected, and the spinor square map is equivariant with respect to the adjoint action, see Subsection \ref{subsec:Bilinearequivariance}, we conclude that $\alpha$ is invariant under $\SU(4)\subset \SO(V,h)$, which in turn implies that $\Omega$ is stabilized by $\SU(4)$. Conversely, suppose that $\Omega$ is stabilized by $\SU(4)\subset \SO(V,h)$ for an Euclidean metric $h$ on $V$ and consider the exterior form $\alpha$ as defined in Equation \eqref{eq:alphaSU(4)}. Then, the preimage of $\SU(4)\subset \SO(V,h)$ by the spin double cover morphism defines a subgroup $\SU(4)\subset \Spin(V,h)$, which in turn defines a two-dimensional subspace of invariant spinors in $\Sigma^{\mu}$ \cite{Wang89}. We must have that the complex-bilinear square of an element of this invariant subspace is precisely $\alpha$ as given in Equation \eqref{eq:alphaSU(4)} in terms of $(\Omega,h)$, for otherwise there would be two different pairs $(\Omega,h)$ and $(\Omega^{\prime},h^{\prime})$ stabilized both by the same $\SU(4)\subset\Aut(V)$.
\end{proof}

\begin{remark}
Note that the analogy with the algebraic description of $\Spin(7)$-structures given in \cite{LS24_Spin7} cannot naively be taken any further. This is because the main results obtained in \cite{LS24_Spin7} regarding the existence of a cubic potential whose critical points are conformal $\Spin(7)$-structures rely on the decomposition of the space of real four-forms in irreducible $\Spin(7)$-representations, which is different from the decomposition of the space of complex four-forms in $\SU(4)$-representations.
\end{remark}

\noindent
We distinguish in the following between the case $\scB(\eta,\eta) = 0$, equivalently $\langle \Omega , \Omega \rangle = 0$, which corresponds to $\eta$ being a pure spinor, and the generic case $\scB(\eta,\eta) \neq 0$, equivalently $\langle \Omega , \Omega \rangle \neq 0$, which corresponds to $\eta$ being \emph{impure} in the terminology of \cite{CharltonThesis}.


\subsubsection{Case $\scB(\eta,\eta)=0$.}

 
If $\scB(\eta,\eta) = 0$, then as stated in Corollary \ref{cor:sq_B_chiral} and Corollary \ref{cor:Bilinear8d}, taking $\beta=1$ does not suffice to characterize the square of the spinor. Choosing $\beta = \bar\Omega \in \wedge^4V^*_\C$, we have: 
\begin{eqnarray*}
\alpha\diamond\beta=\Omega\diamond\bar\Omega=\Omega\wedge\bar\Omega-\Omega\triangle_1\bar\Omega-\Omega\triangle_2\bar\Omega+\Omega\triangle_3\bar\Omega+\escal{\Omega,\bar\Omega}
\end{eqnarray*}

\noindent
and $(\alpha\diamond\beta)^{(0)}=\escal{\Omega,\bar\Omega}=\abs{\Re(\Omega)}^2+\abs{\Im(\Omega)}^2\neq0$ since $\Omega\neq0$.\medskip

We now proceed as in the six-dimensional case to prove the eight-dimensional analogue of Corollary \ref{cor:6dbilinear}.

\begin{cor}
A complex exterior form $\alpha\in\wedge V^*_\C$ is the complex-bilinear square of a complex irreducible spinor $\eta\in\Sigma^\mu$ with chirality $\mu$ if and only if: $$\alpha=\theta_1\wedge\theta_2\wedge\theta_3\wedge\theta_4$$ for complex one-forms $\theta_1,\theta_2,\theta_3,\theta_4\in V^*_\C$ isotropic and orthogonal satisfying: $$*(\theta_1\wedge\theta_2\wedge\theta_3\wedge\theta_4)=\mu(\theta_1\wedge\theta_2\wedge\theta_3\wedge\theta_4).$$
\end{cor}

\begin{proof}
By the previous discussion we know that necessarily $\alpha=\Omega\in\wedge^4V^*_\C$ satisfying $*\Omega=\mu\Omega$. Let us show that $\Omega$ is decomposable. Let $v\in\wedge^2V_\C$ be any complex bi-vector and take $\beta=v^\flat\in\wedge^2V^*_\C$. Then $(\alpha\diamond\beta)^{(0)}=(\Omega\diamond v^\flat)^{(0)}=0$ and the quadratic equation $\alpha\diamond\beta\diamond\alpha=16(\alpha\diamond\beta)^{(0)}\alpha$ becomes:
\begin{equation}\label{eq:8d_norm_zero_auxiliar}
(\Omega\wedge v^\flat-\Omega\triangle_2 v^\flat)\diamond\Omega=0,
\end{equation}

\noindent
where we have used $\Omega\diamond v^\flat+v^\flat\diamond\Omega=2(\Omega\wedge v^\flat-\Omega\triangle_2 v^\flat)$ and $\Omega\diamond\Omega=0$. Note that since $*\Omega=\mu\Omega$ we have: $$*(\Omega\wedge v^\flat)=v^\flat\triangle_2(*\Omega)=\mu v^\flat\triangle_2\Omega.$$

Hence: $$(\Omega\wedge v^\flat)\diamond\Omega=\mu(\Omega\wedge v^\flat)\diamond(*\Omega)=\mu(\Omega\wedge v^\flat)\diamond\nu\diamond\Omega=-\mu*(\Omega\wedge v^\flat)\diamond\Omega=-(v^\flat\triangle_2\Omega)\diamond\Omega.$$

Then \eqref{eq:8d_norm_zero_auxiliar} becomes: $$(v^\flat\triangle_2\Omega)\diamond\Omega=(v^\flat\triangle_2\Omega)\wedge\Omega-(v^\flat\triangle_2\Omega)\triangle_1\Omega-(v^\flat\triangle_2\Omega)\triangle_2\Omega=0.$$

In particular, we have $(v^\flat\triangle_2\Omega)\wedge\Omega=0$ for all $v\in\wedge^2V_\C$, hence $\Omega$ is decomposable by \cite{EM00} since $v^\flat\triangle_2\Omega=\iota_{\tau(v)}\Omega$. Hence, we can write: $$\Omega=\theta_1\wedge\theta_2\wedge\theta_3\wedge\theta_4$$ for complex one-forms $\theta_1,\theta_2,\theta_3,\theta_4\in V^*_\C$. Then, the condition $*\Omega=\mu\Omega$ translates into $*(\theta_1\wedge\theta_2\wedge\theta_3\wedge\theta_4)=\mu(\theta_1\wedge\theta_2\wedge\theta_3\wedge\theta_4)$. Plugging $\theta_1^\sharp$, $\theta_2^\sharp$, $\theta_3^\sharp$, and $\theta_4^\sharp$ into the latter, we conclude that the one-forms are isotropic and mutually orthogonal. We choose a basis:
\begin{equation*}
\{\theta_1,\theta_2,\theta_3,\theta_4,\bar\theta_1,\bar\theta_2,\bar\theta_3,\bar\theta_4\}  
\end{equation*}

\noindent
of $(V^*_\C,h^*_\C)$ given by isotropic one-forms that are conjugate in pairs, that is, they are mutually orthogonal except for:
\begin{equation*}
\escal{\theta_1,\bar\theta_1}=\escal{\theta_2,\bar\theta_2}=\escal{\theta_3,\bar\theta_3}=\escal{\theta_4,\bar\theta_4}=1.    
\end{equation*}

\noindent
Then, setting $\beta=\bar\Omega=\bar\theta_1\wedge\bar\theta_2\wedge\bar\theta_3\wedge\bar\theta_4$ we have $(\alpha\diamond\beta)^{(0)}=\escal{\Omega,\bar\Omega}=1$ and a tedious but straightforward computation shows that the equation $\alpha\diamond\beta\diamond\alpha=16\alpha$ is automatically satisfied by $\alpha$ as given in the statement.
\end{proof}

\begin{remark}
It is interesting to observe how the Plücker relations emerge from the algebraic relations that characterize the square of an irreducible chiral complex spinor. This illustrates how the algebraic and geometric information contained in the general algebraic conditions for the square of an irreducible complex spinor manifests in different intricate ways depending on the dimension and signature.  
\end{remark}


\subsubsection{Case $\scB(\eta,\eta)\neq0$.}


This case corresponds to the square of an irreducible and chiral complex spinor that is \emph{impure} \cite{CharltonThesis}. Corollary \ref{cor:Bilinear8d} can be refined in this case as follows.

\begin{cor}
\label{cor:bilinear8dimpure}
A complex exterior form $\alpha\in \wedge V^*_{\C}$ is the complex-bilinear square of a complex irreducible impure spinor $\eta\in\Sigma^\mu$ of chirality $\mu$ if and only if:
\begin{eqnarray*}
\alpha =  \frac{\langle \Omega , \Omega \rangle^{\frac{1}{2}}}{\sqrt{14}} + \Omega + \frac{\Omega\wedge\Omega}{\sqrt{14} \langle \Omega , \Omega \rangle^{\frac{1}{2}}}
\end{eqnarray*}

\noindent
for a complex four-form $\Omega \in \wedge^4 V^*_\C$ satisfying $*\Omega=\mu\Omega$ and $\sqrt{14}\Omega\triangle_2\Omega + 12  \langle \Omega , \Omega \rangle^{\frac{1}{2}} \Omega = 0$.
\end{cor}

\noindent
We can obtain an alternative point of view by returning to the characterization of the complex-bilinear square of $\eta\in \Sigma^{\mu}$ as a complex exterior form:
\begin{equation*}
\alpha=\lambda + \Omega + \mu \lambda \nu    
\end{equation*}

\noindent
satisfying Equation \eqref{eq:Spin(7)_type_equation}. Writing $\lambda = \lambda_{R} + i \lambda_{I}$ and $\Omega = \Omega_R + i \Omega_I$ for uniquely determined $\lambda_{R} , \lambda_I \in \mathbb{R}$ and $\Omega_R , \Omega_I \in \wedge^4 V^{\ast}$, Equation \eqref{eq:Spin(7)_type_equation} reduces to the following system of \emph{real} algebraic equations:
\begin{eqnarray}
& \vert \Omega_R \vert^2 + 14 \lambda_I^2 = \vert \Omega_I \vert^2 + 14 \lambda_R^2\, , \qquad \langle \Omega_R , \Omega_I\rangle = 14 \lambda_R \lambda_I, \label{eq:impure1}\\
& 6 (\lambda_R \Omega_I + \lambda_I \Omega_R) + \Omega_R \triangle_2 \Omega_I = 0\, , \qquad 12(\lambda_R \Omega_R - \lambda_I \Omega_I) + \Omega_R \triangle_2 \Omega_R - \Omega_I \triangle_2 \Omega_I = 0. \label{eq:impure2}
\end{eqnarray}

\noindent
Assume that $\lambda_R  \neq 0$. Isolating in the second equation in \eqref{eq:impure1}, we obtain $\lambda_I = 14^{-1} \lambda_R^{-1} \langle \Omega_R , \Omega_I\rangle$, which plugged into the first equation in \eqref{eq:impure1}, gives:
\begin{eqnarray*}
(14 \lambda_R^2)^2 + 14 \lambda_R^2 ( \vert \Omega_I \vert^2 - \vert \Omega_R \vert^2) = \langle \Omega_R , \Omega_I\rangle^2.
\end{eqnarray*}

\noindent
The only admissible solution to this algebraic equation reads:
\begin{equation*}
28 \lambda_R^2 =    \vert \Omega_R \vert^2 - \vert \Omega_I \vert^2 + \sqrt{(\vert \Omega_R \vert^2 - \vert \Omega_I \vert^2)^2 + 4 \langle \Omega_R , \Omega_I\rangle^2},
\end{equation*}

\noindent
which immediately implies:
\begin{eqnarray*}
28\lambda^2_I = \frac{4 \langle \Omega_R , \Omega_I\rangle^2}{\vert \Omega_R \vert^2 - \vert \Omega_I \vert^2 + \sqrt{(\vert \Omega_R \vert^2 - \vert \Omega_I \vert^2)^2 + 4 \langle \Omega_R , \Omega_I\rangle^2}},\end{eqnarray*}

\noindent
hence solving both equations in \eqref{eq:impure1}. Plugging this back into the Equation \eqref{eq:impure2} gives an algebraic system depending exclusively on $\Omega_R$ and $\Omega_I$ whose solutions are the real and imaginary parts of the complex four-form $\Omega$ occurring in the complex-bilinear square of an impure irreducible and chiral complex spinor. If we \emph{normalize} $\eta$ with respect to $\scB$, namely we assume that $\scB(\eta,\eta) =16$, then $\lambda_R = 1$, $\lambda_I = 0$, and the algebraic system \eqref{eq:impure1} and \eqref{eq:impure2} reduces to:
\begin{eqnarray*}
& \vert \Omega_R \vert^2  = \vert \Omega_I \vert^2 + 14 \, , \qquad \langle \Omega_R , \Omega_I\rangle = 0,\\
& 6  \Omega_I +   \Omega_R \triangle_2 \Omega_I = 0\, , \qquad 12 \Omega_R   + \Omega_R \triangle_2 \Omega_R - \Omega_I \triangle_2 \Omega_I = 0.
\end{eqnarray*}

\noindent
It can be seen that the first line above recovers the relations found in \cite[Section 3.7]{Krasnov2024} in their computation of the square of a normalized complex and irreducible chiral spinor in eight Euclidean dimensions. The second line in the previous equation seems to be new and guarantees, in an intrinsic way that uses no explicit basis or choices, that both $\Omega_R$ and $\Omega_I$ are the real and imaginary parts of a complex four-form arising from the square of $\eta$ as determined in Corollary \ref{cor:bilinear8dimpure}.


\section{Spinorial instantons and self-dual bundle gerbes}
\label{section:SDGerbe}


In this section, we consider a particular class of spinorial instantons, to which we refer as \emph{spinorial curvings}, on a $\mathbb{C}^{\ast}$-bundle gerbe on a Lorentzian six-dimensional manifold $(M,g)$, a dimension and signature in which spinorial curvings are intimately related to the notion of \emph{self-duality} for three-forms. We begin by computing the Hermitian square of an irreducible complex spinor in six Lorentzian dimensions. This square was considered explicitly in \cite{Figueroa_Lorentzian_6d_2018}, where a set of necessary conditions was obtained for its characterization. Let $(V,h)$ denote a six-dimensional Minkowski space and let $(\Sigma , \gamma)$ be a complex irreducible module of $\Cl(V^{\ast},h^{\ast})$. Hence, $(\Sigma , \gamma)$ is eight-dimensional and of quaternionic type. Denote by: 
\begin{eqnarray*}
\Sigma = \Sigma^{+} \oplus \Sigma^{-}
\end{eqnarray*}

\noindent
the chiral decomposition of $\Sigma$ with respect to the complex volume form $\nu_{\C} = \nu$, which in this case is the \emph{real} Lorentzian volume form $\nu$ of $(V,h)$. Following Proposition \ref{prop:Hermitian_admissible} and Proposition \ref{prop:complex_bilinear_admissible}, we equip $(\Sigma,\gamma)$ with a Hermitian pairing $\scS$ of positive adjoint type together with a compatible complex-bilinear pairing $\scB$, also of positive adjoint type, which by construction is skew-symmetric and related to $\scS$ through a quaternionic structure $J \colon \Sigma \to \Sigma$ on $(\Sigma , \gamma)$. Note that the chiral splitting of $\Sigma$ is isotropic with respect to both $\scS$ and $\scB$. 


\subsection{The chiral Hermitian square}


By Corollary \ref{cor:sq_S_chiral}, a complex exterior form $\widehat{\alpha} \in \wedge V^{\ast}_{\C}$ is the Hermitian square of a non-zero element $\eta\in\Sigma^{\mu}$ of chirality $\mu\in\mathbb{Z}_2$ with $\kappa = 1$ if and only if:
\begin{eqnarray}
\label{eq:sesquilinear6d}
\tau(\widehat{\alpha})=\overline{\widehat{\alpha}}\, ,\qquad   \ast (\pi\circ\tau)(\widehat{\alpha}) =   \mu\widehat{\alpha}\, , \qquad \widehat{\alpha} \diamond \beta \diamond\widehat{\alpha}= 8 (\widehat{\alpha}\diamond\beta)^{(0)}\widehat{\alpha}
\end{eqnarray}

\noindent
for a complex exterior form $\beta\in \wedge V^{\ast}_{\C}$ such that $(\widehat{\alpha} \diamond\beta)^{(0)} \neq 0$. The first two equations, which are linear in $\widehat{\alpha}$, are immediately solved by:
\begin{eqnarray*}
\widehat{\alpha} = u + i \rho  - \mu \ast u\, , \qquad \ast \rho = \mu \rho,
\end{eqnarray*}

\noindent
where $u\in V^{\ast}$ is a real one-form and $\rho\in \wedge^3 V^{\ast}$ is a real three-form, both uniquely determined by $\eta$. Hence, it only remains to solve the third equation in \eqref{eq:sesquilinear6d} for a choice of $\beta$ such that $(\widehat{\alpha} \diamond\beta)^{(0)} \neq 0$. Note that this equation is automatically satisfied by the previous expression of $\widehat{\alpha}$ when $\beta = 1$. Hence, to proceed further, we need to make more educated choices of $\beta \in \wedge V^{\ast}_{\C}$, as the following lemma illustrates. 

\begin{lemma}
Let $\widehat{\alpha} = u + i \rho  - \mu \ast u$ be the Hermitian square of a non-zero element $\eta\in \Sigma^{\mu}$. Then $u\neq 0$.
\end{lemma}

\begin{proof}
Suppose that $\eta\neq 0$ and $u = 0$. Then, the Hermitian square of $\eta$ is of the form $\widehat{\alpha} = i\rho$ with $\ast\rho = \mu \rho$. Let $v\in V$ be and take $\beta=v^\flat\in V^*$. Then $(\widehat{\alpha}\diamond\beta)^{(0)}=i(\rho\diamond v^\flat)^{(0)}=0$ and the quadratic equation $\widehat{\alpha} \diamond \beta \diamond\widehat{\alpha}= 8 (\widehat{\alpha}\diamond\beta)^{(0)}\widehat{\alpha}$ becomes: 
\begin{equation*}
(\iota_v\rho)\wedge\rho-(\iota_v\rho)\triangle_1\rho-(\iota_v\rho)\triangle_2\rho=0,  
\end{equation*}

\noindent
where we have used $\rho\diamond v^\flat+v^\flat\diamond\rho=2\iota_v\rho$ and $\rho\diamond\rho=0$. In particular, we have $(\iota_v\rho)\wedge\rho=0$ for all $v\in V$. Hence, $\rho$ is decomposable by classical Plücker relations \cite{EM00}, and consequently we can write:
\begin{equation*}
\rho = \vartheta_1 \wedge \vartheta_2 \wedge \vartheta_3
\end{equation*}

\noindent
for real one-forms $\vartheta_1 , \vartheta_2 , \vartheta_3 \in V^{\ast}$. Arguing as in the proof of Corollary \ref{cor:6dbilinear}, we conclude that $\vartheta_1$, $\vartheta_2$, and $\vartheta_3$ must be isotropic and mutually orthogonal, which is not possible for real one-forms in Lorentzian signature. Hence $\rho = 0$ and thus $\widehat{\alpha} = 0$, in contradiction with $\eta \neq 0$.
\end{proof}

\noindent
Therefore, in the following, we will assume that $ u \neq 0$. Note, however, that at this point nothing prevents having $\rho = 0$ in the Hermitian square of $\eta \in \Sigma^{\mu}$. We will see later that indeed if $\eta \neq 0$ then $\rho\neq 0$ as well, and thus $\eta \neq 0$ if and only if $u \neq 0$ and $\rho\neq 0$.

\begin{lemma}
Using the notation introduced above, let $\widehat{\alpha} = u + i \rho  - \mu \ast u$ be the Hermitian square of an irreducible complex and chiral spinor $\eta\in \Sigma^{\mu}\setminus \{ 0 \}$. Then:
\begin{equation*}
\escal{u,u}= 0 \, , \qquad u\wedge \rho = 0\, , \qquad \rho(u^\sharp) = 0.
\end{equation*}
\end{lemma}

\begin{proof}
Suppose $u\neq 0$ and set $\beta = u\in V^*$ in the third equation of \eqref{eq:sesquilinear6d}. We compute:
\begin{equation*}
\widehat{\alpha}\diamond u=\escal{u,u}+i(\rho\wedge u+\rho(u^\sharp))+\mu\escal{u,u}\nu.
\end{equation*}

\noindent
From this, we obtain that the third equation in \eqref{eq:sesquilinear6d} is equivalent to:
\begin{equation}
\label{eq:3rdII}
\widehat{\alpha}\diamond u\diamond\widehat{\alpha}=2\escal{u,u}(u+2i\rho-\mu*u)-\rho\diamond u\diamond\rho=8\escal{u,u}(u + i \rho  - \mu \ast u)=8 (\widehat{\alpha}\diamond u)^{(0)}\widehat{\alpha}.
\end{equation}

\noindent
Regarding the three-form $\rho$, we have two possibilities: either $\rho = 0$ or $\rho \neq 0$. If $\rho = 0$ then the previous equation reduces to $\escal{u,u} = 0$. On the other hand, if $\rho\neq  0$ then the imaginary part of Equation \eqref{eq:3rdII} gives $\escal{u,u}\rho = 0$ and hence $\escal{u,u} = 0$ again. Hence $\escal{u,u} = 0$ and plugging this condition into Equation \eqref{eq:3rdII}, we obtain:
\begin{equation*}
\rho \diamond u\diamond \rho   =  2 \rho(u^\sharp) \diamond \rho = 2 (\rho(u^\sharp) \wedge \rho - \rho(u^\sharp) \triangle_1 \rho - \rho(u^\sharp) \triangle_2 \rho) = 0
\end{equation*}
 
\noindent
or, equivalently:
\begin{equation*}
\rho(u^\sharp) \wedge \rho  = 0\, , \qquad  \rho(u^\sharp) \triangle_1 \rho = 0\, , \qquad  \rho(u^\sharp) \triangle_2 \rho = 0.
\end{equation*}

\noindent
To proceed further, using that we have proven $u$ to be isotropic,  we choose a one-form $v\in V^*$ \emph{conjugate} to $u$. That is, $v$ is a nowhere vanishing isotropic one-form satisfying $\langle u , v \rangle = 1$. Hence:
\begin{equation*}
V  = \langle \mathbb{R} u^{\sharp} \rangle \oplus \langle \mathbb{R} v^{\sharp} \rangle \oplus (\langle \mathbb{R} u^{\sharp} \rangle \oplus \langle \mathbb{R} v^{\sharp} \rangle)^{\perp},
\end{equation*}

\noindent
where $V_{uv}:=(\langle \mathbb{R} u^{\sharp} \rangle \oplus \langle \mathbb{R} v^{\sharp} \rangle)^{\perp} \subset V$ denotes the orthogonal complement of $\langle \mathbb{R} u^{\sharp} \rangle \oplus \langle \mathbb{R} v^{\sharp} \rangle \subset V$. With respect to the previous splitting, we can decompose $\rho\in\wedge^3V^*$ as follows:
\begin{equation*}
\rho=u\wedge v\wedge\phi+u\wedge\omega+v\wedge\varrho+\rho^\perp,
\end{equation*}

\noindent
where $\phi\in V_{uv}^*$, $\omega,\varrho\in\wedge^2V_{uv}^*$, and $\rho^\perp\in\wedge^3V_{uv}^*$. Note that the metric $h$ on $V$ induces a positive-definite metric $h_{uv}$ on $V_{uv}$. Moreover, we define a volume form $\nu_{uv}$ on $V_{uv}$ by the relation $\nu=u\wedge v\wedge \nu_{uv}$, and consequently a Hodge star operator $*_{uv}$ on $V_{uv}$. Therefore, the duality condition $*\rho=\mu\rho$ becomes:
\begin{equation*}
*_{uv}\omega=-\mu\omega,\qquad *_{uv}\varrho=\mu\varrho,\qquad *_{uv}\rho^\perp=\mu\phi.
\end{equation*}

\noindent
We then have $\rho(u^\sharp)=-u\wedge\phi+\varrho$ and the equation $\rho(u^\sharp)\wedge\rho=0$ becomes: $$\phi\wedge\varrho=0,\qquad \phi\wedge\rho^\perp-\omega\wedge\varrho=0,\qquad \varrho\wedge\varrho=0.$$

Using that $*_{uv}\varrho=\mu\varrho$, the third equation above implies $\escal{\varrho,\varrho}=0$, which in turn implies $\varrho=0$ since the metric $h_{uv}$ on $V_{uv}$ is positive-definite. Using $*_{uv}\rho^\perp=\mu\phi$, the second equation above implies that $\rho^\perp=0$, so $\phi=0$ and $\rho=u\wedge\omega$. This immediately gives:
\begin{equation*}
\rho(u^\sharp)=0,\qquad u\wedge\rho=0
\end{equation*}

\noindent
and thus we conclude.
\end{proof}

\noindent
Hence, we conclude that $u$ is a non-zero isotropic one-form, and furthermore there exists a two-form $\omega\in \wedge^2 V^{\ast}$ such that $\rho = u \wedge \omega$ and $\omega(u^\sharp) = 0$. Rather crucially, such $\omega$ is unique modulo transformations of the form $\omega \mapsto \omega + u \wedge \theta$ for a one-form $\theta \in V^{\ast}$ orthogonal to $u$.\medskip

Since $u$ is isotropic, the choice $\beta = u$ in Theorem \ref{thm:even_Hermitian_forms} does not suffice to characterize the Hermitian square of $\eta\in\Sigma^{\mu}$. Instead, we consider a conjugate isotropic one-form $v\in V^{\ast}$ satisfying $\langle u , v\rangle =1$ and we set $\beta = v$ in the third equation of \eqref{eq:sesquilinear6d}. We compute: 
\begin{equation*}
\widehat{\alpha}\diamond v=(u+i\rho-\mu*u)\diamond v=u\wedge v+\escal{u,v}+i\rho\wedge v+i\rho(v^\sharp)-\mu*(u\wedge v)+\mu\escal{u,v}\nu.   
\end{equation*}

\noindent
Hence $(\widehat{\alpha}\diamond v)^{(0)}=\escal{u,v}=1$, and therefore plugging the choice $\beta = v$ in the third equation of \eqref{eq:sesquilinear6d} gives the necessary and sufficient conditions for $\widehat{\alpha}$ to be the Hermitian square of $\eta$.

\begin{remark}
It can be seen that if $v, v^{\prime} \in V^{\ast}$ are both conjugate to $u\in V^{\ast}$ then there exists a unique element $w\in \Span_\R\{u,v\}^{\perp}$ such that:
\begin{equation*}
v^{\prime} = v - \frac{\escal{w,w}}{2} u + w,
\end{equation*}

\noindent
where $\Span_\R\{u,v\}^{\perp} \subset V^{\ast}$ is the orthogonal complement of the real span of $u$ and $v$. Hence, the set of one-forms conjugate to a given non-zero isotropic one-form $u\in V^{\ast}$, which we denote by $\cI_u$, is a torsor over $\mathbb{R}^4$. The four-dimensional analog of this formula was used in \cite{Shahbazi:2025qnl} for the algebraic description of the square of an irreducible and real spinor in four Lorentzian dimensions.
\end{remark}

\noindent
We compute:
\begin{equation*}
\widehat{\alpha}\diamond v\diamond \widehat{\alpha} = (u + i \rho  - \mu \ast u) \diamond v \diamond (u + i \rho  - \mu \ast u)  = 4 u + 8 i \rho - 4\mu \ast u - \rho \diamond v\diamond \rho.
\end{equation*}

\noindent
On the other hand, we have:
\begin{equation*}
\rho \diamond v\diamond \rho =(u\wedge \omega) \diamond v \diamond (u\wedge \omega) = 2 u \wedge \omega\wedge \omega - 2\escal{\omega,\omega}u
\end{equation*}

\noindent
and thus the third equation in \eqref{eq:sesquilinear6d} reduces to:
\begin{equation*}
\widehat{\alpha}\diamond v\diamond \widehat{\alpha} = (4 + 2\escal{\omega,\omega})u + 8 i \rho - 4\mu \ast u - 2u \wedge \omega \wedge \omega = 8(u + i \rho  - \mu \ast u).
\end{equation*}

\noindent
This implies the following characterization of the Hermitian square of $\eta \in \Sigma^{\mu}$.

\begin{prop}
\label{prop:HermitiansquareLorentz6d}
A complex exterior form $\widehat{\alpha} \in \wedge V^{\ast}_{\mathbb{C}}$ is the Hermitian square of an irreducible and chiral non-zero complex spinor $\eta\in\Sigma^\mu$ if and only if:
\begin{equation*}
\widehat{\alpha} = u + i u\wedge \omega - \mu \ast u
\end{equation*}

\noindent
for a unique non-zero isotropic real one-form $u\in V^{\ast}$ and a two-form $\omega\in \wedge^{2} V^{\ast}$ satisfying:
\begin{equation}
\label{eq:omegaconditions}
\omega(u^\sharp) = 0\, , \qquad \escal{\omega,\omega}=2 \, , \qquad 2\mu \ast u = u\wedge \omega \wedge \omega.
\end{equation}

\noindent
In particular, $\omega$ is non-zero if and only if $\eta$ is non-zero. 
\end{prop}

\noindent
Note that if $\omega$ satisfies equations \eqref{eq:omegaconditions}, then every other two-form of the form $\omega + u\wedge \theta$, with $\theta$ orthogonal to $u$, also satisfies \eqref{eq:omegaconditions}. Given a non-zero isotropic one-form $u\in V^{\ast}$, we have a canonically defined five-dimensional vector subspace $\mathrm{Ker}(u) \subset V$ given by the kernel of $u$. Since $u$ is isotropic, its dual belongs to $\mathrm{Ker}(u)$ and we obtain a canonically defined four-dimensional vector space by taking the quotient with the line spanned by the metric dual of $u$:
\begin{equation*}
\cV_u := \frac{\mathrm{Ker}(u)}{\langle \mathbb{R} u^\sharp\rangle}.
\end{equation*}

\noindent
This is the standard algebraic model for the \emph{screen bundle} canonically associated to every Lorentzian manifold equipped with a nowhere vanishing isotropic vector field. Every two-form $\omega$ occurring as $\rho = u\wedge \omega$ descends to $\cV_u$ as a consequence of the condition $\omega (u^\sharp) = 0$. Hence, associated to every irreducible and chiral complex spinor $\eta\in \Sigma^{\mu}$ we obtain a two-form $\omega_u \in \wedge^2 \cV_u^{\ast}$ defined by $\omega_u(v_1 , v_2) = \omega(v_1,v_2)$ for every $v_1 , v_2 \in \mathrm{Ker}(u)$. Similarly, the Lorentzian metric $h$ on $V$ induces naturally a Riemannian metric $h_u$ on $\cV_u$. Every choice of conjugate one-form $v\in \cI_u$ defines a canonical isomorphism:
\begin{equation*}
\cV_u \xrightarrow{\simeq} V_{uv} := \mathrm{Ker}(u,v) \subset V \, , \qquad [w] \mapsto w,
\end{equation*}

\noindent
where $w\in [w]$ is the unique representative in $[w] \in \cV_u$ that satisfies $v(w) = 0$. Similarly, there exists a unique representative $\omega_{uv} \in \omega_u$ that satisfies $\omega_{uv}(v^\sharp) = 0$ and consequently defines a two-form on $\omega_{uv} \in \wedge^2 V^{\ast}_{uv} $. Let $\nu$ be the Lorentzian volume form of $(V,h)$. Then, $u$ and $v$ determine a natural volume form $\nu_{uv} \in \wedge^4 V^{\ast}$ on $V_{uv}$ through the following relation:
\begin{equation}
\label{eq:nunuuv}
\nu = u\wedge v \wedge \nu_{uv}.
\end{equation} 

\noindent
The volume form $\nu_{uv}$ on $V_{uv}$ corresponds to the induced volume form $\nu_u$ on $\cV_u$ via the aforementioned isomorphism $\cV_u \xrightarrow{\simeq} V_{uv}$. Analogously, the metric $h_u$ on $\cV_u$ maps by the previous isomorphism to a Euclidean metric $h_{uv}$ on $V_{uv}$.

\begin{lemma}
A two-form $\omega\in \wedge^2 V^{\ast}$ satisfies equations \eqref{eq:omegaconditions} only if $\omega_u \in \wedge^2 \cV_u^{\ast}$ satisfies:
\begin{equation*}
\escal{\omega_u,\omega_u} = 2\, , \qquad \omega_u\wedge\omega_u = -2 \mu \nu_u,
\end{equation*}

\noindent
where $\nu_u$ is the Riemannian volume on $\cV_u$ induced by $h$.
\end{lemma}

\begin{proof}
Suppose that $\omega$ satisfies equations \eqref{eq:omegaconditions}. Since:
\begin{eqnarray*}
\escal{\omega + u \wedge \theta,\omega + u \wedge \theta} = \escal{\omega,\omega} + 2\escal{\omega , u \wedge \theta} + \escal{u\wedge \theta,u\wedge \theta} = \escal{\omega,\omega}
\end{eqnarray*}

\noindent
it follows that $\escal{\omega_u,\omega_u} = 2$. Choose now the unique representative $\omega \in \omega_u$ that satisfies $\omega(v^\sharp) = 0$. Evaluating the third equation in \eqref{eq:omegaconditions} on $v$, we obtain:
\begin{equation*}
2\mu \ast(u\wedge v) = \omega \wedge \omega
\end{equation*}

\noindent
or, equivalently:
\begin{equation*}
2 \mu \nu_{uv} + \omega \wedge \omega = 0,
\end{equation*}

\noindent
where we have used Equation \eqref{eq:nunuuv}.
\end{proof}

\noindent
Given a complex chiral spinor $\eta\in \Sigma^{\mu}$ with Hermitian square $\widehat{\alpha} = u + i u\wedge \omega - \mu \ast u$, by the previous lemma it follows that $\omega_u$ is a \emph{Kähler form} on $(\cV_u , h_u)$ and on every explicit realization $(V_{uv} , h_{uv})$ determined by the choice of a one-form $v\in V^{\ast}$ conjugate to $u\in V^{\ast}$. This Kähler form determines the same orientation as $\nu_{uv}$ if $\eta$ has negative chirality, and the opposite orientation if $\eta$ has positive chirality. In particular:
\begin{eqnarray*}
\ast_u \omega_u = -\mu \omega_u,
\end{eqnarray*}

\noindent
where $\ast_u$ denotes the Hodge dual of the induced metric and orientation on $\cV_u$. The standard definition of a \emph{Kähler form} implicitly assumes that the underlying orientation is induced by the top power of the Kähler form. In the present situation we need to consider the slight generalization that allows for this top power to induce the \emph{opposite} orientation to the fixed underlying orientation in terms of which the Hodge dual is defined.  

\begin{prop}
There is a one-to-one correspondence between complex irreducible chiral non-zero spinors up to phase, that is $[\eta]\in(\Sigma^\mu\setminus\{0\})/\U(1)$, and pairs $(u,\omega_u)$ consisting of an isotropic one-form $u\in V^{\ast}$ and a two-form $\omega_u \in \wedge^2 \cV_u^{\ast}$ on $\cV_u$ satisfying $\escal{\omega_u,\omega_u}=2$ and $2\mu*u=u\wedge\omega_u\wedge\omega_u$.
\end{prop}

\noindent
Such a pair $(u,\omega_u)$ gives a natural geometric generalization of the notion of a \emph{parabolic pair} introduced in \cite{Murcia:2020zig,Murcia:2021dur} in the study of the square of a real and irreducible spinor in four Lorentzian dimensions. 

\begin{remark}
\label{remark:specialform6d} 
Note that the square $\widehat{\alpha} = u + i u\wedge \omega - \mu \ast u$ of an irreducible complex and chiral spinor can always be written as follows:
\begin{eqnarray*}
\widehat{\alpha} = u \diamond (1 + i \omega - \mu \nu),
\end{eqnarray*}

\noindent
which is sometimes useful for computations.
\end{remark}


\subsection{The chiral complex-bilinear square}


Let $\mu\in\mathbb{Z}_2$. By Corollary \ref{cor:sq_B_chiral}, an exterior form $\alpha\in\wedge V^*_\C$ is the complex-bilinear square of a spinor $\eta\in\Sigma$ with chirality $\mu$ if and only if:
\begin{equation*}
\alpha\diamond\alpha=8\alpha^{(0)}\alpha,\quad\tau(\alpha) =-\alpha\, , \quad \alpha \diamond \beta \diamond \alpha = 8 (\alpha\diamond\beta)^{(0)} \alpha,\quad *(\pi\circ\tau)(\alpha)=\mu\alpha  
\end{equation*}

\noindent
for a fixed exterior form $\beta\in\wedge V^*_\C$ satisfying $(\alpha\diamond\beta)^{(0)}\neq0$. Let $\alpha=\sum_{k=0}^6\alpha^{(k)}\in\wedge V^*_\C$. The first linear condition $\tau(\alpha)=-\alpha$ implies that $\alpha^{(0)}=\alpha^{(1)}=\alpha^{(4)}=\alpha^{(5)}=0$, while the second linear condition $*(\pi\circ\tau)(\alpha)=\mu\alpha$ implies that $\alpha^{(2)}=\alpha^{(6)}=0$ and $*\alpha^{(3)}=\mu\alpha^{(3)}$. Set $\rho:=\alpha^{(3)}$. The quadratic equation $\alpha\diamond\alpha=8\alpha^{(0)}\alpha$ reduces to $\rho\diamond\rho=0$, which is automatically satisfied since $\ast\rho = \mu \rho$. Since $\alpha^{(0)}=0$, taking $\beta=1$ in Theorem \ref{thm:bilinearsquare} does not suffice to characterize the square of the spinor. Note that in this case we also have $\rho\diamond\bar\rho=0$, thus taking $\beta=\bar\rho$ is also not sufficient to characterize the square of the spinor, in contrast to the six-dimensional Euclidean case. However, arguing as in Corollary \ref{cor:6dbilinear}, it is still possible to show that:
\begin{equation*}
\rho=\theta_1\wedge\theta_2\wedge\theta_3
\end{equation*}

\noindent
for complex one-forms $\theta_1,\theta_2,\theta_3\in V^*_\C$ isotropic and orthogonal satisfying: 
\begin{equation*}
*(\theta_1\wedge\theta_2\wedge\theta_3)=\mu(\theta_1\wedge\theta_2\wedge\theta_3).
\end{equation*}

We now choose a basis $\{\theta_1,\theta_2,\theta_3,\phi_1,\phi_2,\phi_3\}$ of $(V^*_\C,h^*_\C)$ given by isotropic one-forms that are conjugate in pairs, that is, they are mutually orthogonal except for $\escal{\theta_j,\phi_j}=1$, $j=1,2,3$. Then, setting $\beta=\phi_1\wedge\phi_2\wedge\phi_3\in\wedge^3V^*_\C$ we have $(\alpha\diamond\beta)^{(0)}=-\escal{\rho,\phi_1\wedge\phi_2\wedge\phi_3}=-1$, and a tedious but straightforward computation shows that the equation $\alpha\diamond\beta\diamond\alpha=-8\alpha$ is automatically satisfied. Hence, we obtain the following result.

\begin{cor}
A complex exterior form $\alpha\in\wedge V^*_\C$ is the complex-bilinear square of a complex irreducible spinor $\eta\in\Sigma^\mu$ with chirality $\mu \in \mathbb{Z}_2$ if and only if: 
\begin{equation*}
\alpha=\theta_1\wedge\theta_2\wedge\theta_3    
\end{equation*}

\noindent
for complex one-forms $\theta_1,\theta_2,\theta_3\in V^*_\C$ isotropic and orthogonal satisfying $\ast (\theta_1\wedge\theta_2\wedge\theta_3) = \mu(\theta_1\wedge\theta_2\wedge\theta_3)$.
\end{cor}


\subsection{Spinorial curvings on Lorentzian six-manifolds}


Let $(M,g)$ be an oriented and \emph{time}-oriented Lorentzian six-dimensional manifold equipped with a bundle of irreducible complex spinors $S$. By the results of \cite{LS18,LS19,LS22_complex_type}, such $S \to (M,g)$ exists if and only if $(M,g)$ admits a $\Spin_o^c(5,1)$-structure, in which case $S$ is associated to a $\Spin_o^c(5,1)$-structure in the standard way through the tautological representation induced by $\gamma\colon\C\mathrm{l}(5,1)\to\End(\Sigma)$ and the natural embedding $\Spin_o^c(5,1)\subset\C\mathrm{l}(5,1)$. We fix a $\mathbb{C}^{\ast}$-bundle gerbe with connective structure $(\cP, Y, \cA)$, where $\cP\to Y\times_M Y$ is a principal $\mathbb{C}^{\ast}$-bundle defined on the fibered product of the smooth submersion $Y\to M$ with itself, and $\cA$ is a connection on $\cP$ satisfying adequate compatibility conditions. We refer the reader to \cite{Brylinski1993,Bunk2021,Giraud1971,Murray1996} for more details about abelian bundle gerbes. A \emph{curving} on $\cC := (\cP , Y, \cA)$ is a complex-valued two-form $b\in \Omega^2(Y,\C)$ on the total space of the submersion $Y\to M$ satisfying:
\begin{equation*}
F_{\cA} = \delta b \in \Omega^2( Y \times_M Y,\mathbb{C}),
\end{equation*}

\noindent
where $F_{\cA}$ denotes the curvature of $\cA$ understood as a complex two-form on $Y\times_M Y$ and $\delta$ is the simplicial differential of the simplicial manifold defined by $Y\to M$. It can be seen that the exterior derivative of $b$ descends to a closed, possibly non-exact, complex-valued three-form $H_b \in \Omega^3(M ,\C)$, called the \emph{curvature} of $b$.
 
\begin{definition}
A \emph{spinorial curving} on $(\cC,M,g)$ is a pair $(\eta,b)$ consisting of a nowhere vanishing irreducible chiral complex spinor $\eta \in \Gamma(S)$ and a curving $b$ on $\cC$ satisfying:
\begin{equation*}
H_b\cdot \eta = 0,
\end{equation*}

\noindent
where $\cdot$ denotes Clifford multiplication. 
\end{definition}

\noindent
As an immediate consequence of Lemma \ref{lemma:constrainedspinoreven}, we obtain the following characterization of spinorial curvings $(\eta,b)$ in terms of the Hermitian square of $\eta$.

\begin{cor}
A pair $(\eta,b)$ is a spinorial curving on $(\cC,M,g)$ if and only if:
\begin{equation}
\label{eq:Hbspinorial}
H_b \diamond \widehat{\alpha} = 0,
\end{equation}

\noindent
where $\widehat{\alpha}$ is the Hermitian square of $\eta$ as determined in Proposition \ref{prop:HermitiansquareLorentz6d}.
\end{cor}

\begin{remark}
In the previous corollary we could have equivalently used the complex-bilinear square of $\eta$, instead of its Hermitian square, which we consider to be more convenient for our purposes.
\end{remark}

\noindent
From the previous corollary, we arrive at the following equivalent characterization of spinorial curvings.

\begin{lemma}
\label{lemma:spinorialcurving}
A pair $(\eta,b)$ is a spinorial curving on $(\cC,M,g)$ if and only if:
\begin{gather*}
H_b(u^\sharp) - \mu \ast (H_b \wedge u) + i( u\wedge H_b\triangle_2 \omega - H_b(u^\sharp) \triangle_1 \omega ) = 0,\\
H_b\wedge u + \mu \ast H_b(u^\sharp) + i (H_b (u^\sharp)\wedge \omega - u \wedge H_b\triangle_1\omega) = 0,\\
u\wedge \omega \wedge H_b = 0 \, , \qquad \escal{H_b(u^\sharp),\omega} = 0,
\end{gather*}

\noindent
where $u\in \Omega^1(M)$ and $\omega \in \Omega^2(M)$ are determined by the Hermitian square $\widehat{\alpha} = u \diamond (1 + i \omega - \mu \nu)$ of $\eta$.
\end{lemma}

\begin{proof}
Equation \eqref{eq:Hbspinorial} is equivalent to:
\begin{equation*}
H_b \diamond \widehat{\alpha} = H_b\diamond u \diamond (1 + i\omega - \mu \nu) = 0.
\end{equation*}

\noindent
We compute:
\begin{align*}
H_b\diamond u\diamond(1-\mu\nu)&=H_b\wedge u+H_b(u^\sharp)-\mu*(H_b\wedge u)+\mu *H_b(u^\sharp),\\
H_b\diamond u\diamond\omega&=H_b\diamond(u\wedge\omega)=H_b\wedge u\wedge\omega+H_b\triangle_1(u\wedge\omega)-H_b\triangle_2(u\wedge\omega)-\escal{H_b,u\wedge\omega}\\
&=H_b(u^\sharp)\wedge\omega-H_b(u^\sharp)\triangle_1\omega-\escal{H_b(u^\sharp),\omega}-u\wedge(H_b\wedge\omega+H_b\triangle_1\omega-H_b\triangle_2\omega),
\end{align*}

\noindent
where we have used:
\begin{align*}
H_b\triangle_1(u\wedge\omega)&=H_b(u^\sharp)\wedge\omega-u\wedge(H_b\triangle_1\omega),\\
H_b\triangle_2(u\wedge\omega)&=H_b(u^\sharp)\triangle_1\omega-u\wedge(H_b\triangle_2\omega),\\
\escal{H_b,u\wedge\omega}&=\escal{H_b(u^\sharp),\omega}.
\end{align*}

Isolating by degree, we obtain the stated equations.
\end{proof}

\noindent
In order to shed light on the conditions contained in the previous lemma, we choose a one-form $v\in \Omega^1(M)$ conjugate to $u$ and we choose a representative for $\omega$ satisfying $\omega(v^\sharp)=0$. Note that the existence of $v$ is guaranteed by the orientability and time-orientability of $(M,g)$. The pair $(u,v)$ defines a natural orthogonal splitting of the tangent bundle of $(M,g)$:
\begin{align*}
TM &= \langle \mathbb{R} u^\sharp \rangle \oplus \langle \mathbb{R} v^\sharp \rangle \oplus (\mathrm{Ker}(u)\cap \mathrm{Ker}(v))\\
&= \langle \mathbb{R} u^\sharp \rangle \oplus \langle \mathbb{R} v^\sharp \rangle \oplus (\langle \mathbb{R} u^\sharp \rangle \oplus \langle \mathbb{R} v^\sharp \rangle)^{\perp_g}\\
&= \langle \mathbb{R} u^\sharp \rangle \oplus \langle \mathbb{R} v^\sharp \rangle \oplus V_{uv},
\end{align*}

\noindent
and similarly for the cotangent bundle $T^{\ast}M$. For every curving $b\in \Omega^2(Y , \mathbb{C})$, we split $H_b \in \Omega^3(M ,\C)$ accordingly:
\begin{equation}
\label{eq:splitHb}
H_b = u\wedge v \wedge \beta + u \wedge \chi_u + v\wedge \chi_v + H_b^{\perp}
\end{equation}

\noindent
for uniquely determined complex-valued sections:
\begin{equation*}
\beta \in \Gamma(V_{uv}^{\ast} , \mathbb{C}) \, , \qquad \chi_u , \chi_v \in \Gamma(\wedge^2 V_{uv}^{\ast} ,\mathbb{C})\, , \qquad H_b^{\perp} \in \Gamma(\wedge^3 V_{uv}^{\ast} , \mathbb{C}).
\end{equation*}

\noindent
Hence:
\begin{equation*}
H_b(u^{\sharp}) = - u \wedge \beta + \chi_v.
\end{equation*}

\noindent
For further reference, we compute the following Hodge duals using Equation \eqref{eq:nunuuv}:
\begin{align}
\ast(u \wedge \chi_u) &= (-1)^{\vert \chi_u \vert + 1} u\wedge \ast_{uv} \chi_u = - u\wedge \ast_{uv} \chi_u, \nonumber\\
\ast(v \wedge \chi_v) &= (-1)^{\vert \chi_v \vert } v\wedge \ast_{uv} \chi_v = v \wedge \ast_{uv} \chi_v, \label{eq:Hodgeidentities}\\
\ast H_b^{\perp} &= u\wedge v \wedge \ast_{uv} H_b^{\perp}, \nonumber
\end{align}

\noindent
where $\ast_{uv}$ denotes the Hodge dual of the induced metric and orientation on $V_{uv}$. In the following, we denote by $\triangle^{\perp}_k$ the bilinear operator $\triangle_k$ defined in Equation \eqref{eq:bilinearDelta} restricted to elements in $\wedge V_{uv}^{\ast}$ with the induced Euclidean metric. A direct computation proves the following lemma.

\begin{lemma}
\label{lemma:HDeltaomega}
The following formulas hold:
\begin{align*}
H_b \triangle_1 \omega &= u\wedge v \wedge \omega(\beta^{\sharp}) - u \wedge \chi_u \triangle^{\perp}_1 \omega -  v \wedge \chi_v \triangle^{\perp}_1 \omega + H_b^{\perp} \triangle^{\perp}_1 \omega,\\
H_b \triangle_2 \omega &=  \langle \chi_u , \omega\rangle u +  \langle \chi_v , \omega \rangle v + H_b^{\perp} \triangle^{\perp}_2 \omega,
\end{align*}

\noindent
where $H_b \in \Omega^3(M ,\C)$ is given as in Equation \eqref{eq:splitHb}.
\end{lemma}

\noindent
We refine now Lemma \ref{lemma:spinorialcurving} in terms of $\beta$, $\chi_u$, $\chi_v$ and $H_b^{\perp}$ in the splitting of $H_b$ given in Equation \eqref{eq:splitHb}.

\begin{lemma}
\label{lemma:spinorialcurvingII}
A pair $(\eta,b)$ is a spinorial curving on $(\cC,M,g)$ if and only if:
\begin{equation*}
\beta = \mu*_{uv}H_b^\perp + iH_b^\perp\triangle_2^\perp\omega - i\omega(\beta^\sharp)\, , \qquad  \chi_v  \wedge \omega = 0\, , \qquad  \mu \ast_{uv} \chi_v = \chi_v - i \chi_v \triangle^{\perp}_1 \omega,
\end{equation*}

\noindent
where $H_b \in \Omega^3(M ,\C)$ is given as in Equation \eqref{eq:splitHb}.
\end{lemma}

\begin{proof}
A pair $(\eta,b)$ is a spinorial curving if and only if all equations in Lemma \ref{lemma:spinorialcurving} are satisfied. Plugging \eqref{eq:splitHb} into the last line of the equations in Lemma \ref{lemma:spinorialcurving}, we obtain:
\begin{equation*}
\chi_v  \wedge \omega = 0 \, , \qquad  \langle \chi_v , \omega \rangle = 0,
\end{equation*}

\noindent
which are equivalent equations since $\omega$ satisfies $\ast_{uv} \omega = -\mu \omega$. On the other hand, plugging \eqref{eq:splitHb} into the second line of the equations in Lemma \ref{lemma:spinorialcurving}, and using the identities \eqref{eq:Hodgeidentities} together with the previous equations, we obtain the following equivalent conditions:
\begin{equation}
\label{eq:ast_uv_chi_v}
\mu \ast_{uv} \chi_v = \chi_v  - i \chi_v \triangle_1^{\perp} \omega\, , \qquad \mu \ast_{uv}\beta + H_b^\perp + i\beta\wedge\omega + i H_b^\perp\triangle_1^\perp\omega = 0.
\end{equation}

\noindent
Finally, plugging \eqref{eq:splitHb} into the first line of the equations in Lemma \ref{lemma:spinorialcurving}, and using the identities \eqref{eq:Hodgeidentities} together with the previous equations, we obtain:
\begin{equation}
\label{eq:ast_uv_chi_vII}
\beta = \mu*_{uv}H_b^\perp + iH_b^\perp\triangle_2^\perp\omega - i\omega(\beta^\sharp),
\end{equation}

\noindent
where we have used the identities $\ast\chi_v=u\wedge v\wedge \ast_{uv}\chi_v$, $\ast (u\wedge\beta)=u\wedge \ast_{uv}\beta$, as well as Equation \eqref{eq:ast_uv_chi_v}. Taking the Hodge dual of the second equation in \eqref{eq:ast_uv_chi_v}, we conclude:
\begin{equation*}
\beta = \mu\ast_{uv}H_b^\perp + i \mu \ast_{uv} (H_b^\perp\triangle_1^\perp\omega) - i\omega(\beta^\sharp).
\end{equation*}

\noindent
Using the identity $\ast_{uv} (H_b^\perp\triangle_1^\perp\omega) = \mu H_b^\perp\triangle_2^\perp\omega$ we recover Equation \eqref{eq:ast_uv_chi_vII}. Hence, the equations in \eqref{eq:ast_uv_chi_v}, or, equivalently, the first equation in \eqref{eq:ast_uv_chi_v} together with Equation \eqref{eq:ast_uv_chi_vII}, give the if and only if conditions for $(\eta,b)$ to be a spinorial curving on $(\cC,M,g)$.
\end{proof}

\begin{thm}
\label{thm:spinorialcurving}
A pair $(\eta,b)$ is a spinorial curving on $(\cC,M,g)$ if and only if there exists an isotropic one-form $v\in \Omega^1(M)$ conjugate to the Dirac current $u$ of $\eta$ such that the following conditions are satisfied:
\begin{gather*}
H_b(u^{\sharp},v^{\sharp}) = -\mu \ast_{uv} H_b^\perp - iH_b^\perp\triangle_2^\perp\omega - i\omega(H_b(u^{\sharp},v^{\sharp})^\sharp),\\
(H_b(u^\sharp,v^\sharp)\wedge u+H_b(u^\sharp))\wedge\omega=0,\\
\mu \ast_{uv} (H_b(u^\sharp,v^\sharp)\wedge u+H_b(u^\sharp)) = H_b(u^\sharp,v^\sharp)\wedge u+H_b(u^\sharp) - i (H_b(u^\sharp,v^\sharp)\wedge u+H_b(u^\sharp)) \triangle^{\perp}_1 \omega,
\end{gather*}

\noindent
where $\ast_{uv}$ denotes the Hodge dual induced by $h$ on the orthogonal complement $V_{uv}$ to $u$ and $v$ and the Hermitian square of $\eta$ is given by $\widehat{\alpha} = u + i u\wedge \omega - \mu \ast u$. 
\end{thm}
  
\begin{proof}
The pair $(\eta , b)$ is a spinorial curving if and only if the conditions of Lemma \ref{lemma:spinorialcurvingII} are satisfied. The first equation in Lemma \ref{lemma:spinorialcurvingII} is equivalent to the first equation in the statement after noticing that $\beta = -H_b(u^{\sharp},v^{\sharp})$. The second and third equation in the statement follow by substituting $\chi_v=H_b(u^\sharp,v^\sharp)\wedge u+H_b(u^\sharp)$ in Lemma \ref{lemma:spinorialcurvingII}.
\end{proof}

\noindent
As we can see from the above theorem, the equations that an arbitrary complex three-form $H_b$ must satisfy so that the pair $(\eta,b)$ is a spinorial curving are rather involved. If we assume that $H_b$ satisfies the appropriate duality condition, the condition of being spinorial follows automatically.

\begin{cor}\label{cor:b_curving_if_H_dual}
Let $b$ be a curving whose curvature satisfies $\ast H_b = \mu H_b$. Then, the pair $(\eta,b)$ is a spinorial curving on $(\cC,M,g)$ for every chiral irreducible complex spinor $\eta$ of chirality $\mu\in \mathbb{Z}_2$ on $(M,g)$.
\end{cor}

\begin{proof}
The condition $\ast H_b = \mu H_b$ is equivalent to: $$*_{uv}\chi_u=-\mu\chi_u,\qquad *_{uv}\chi_v=\mu\chi_v,\qquad H_b(u^\sharp,v^\sharp)=-\mu*_{uv}H_b^\perp.$$

The condition $*_{uv}\chi_v=\mu\chi_v$, together with $\ast_{uv} \omega = -\mu \omega$, implies that $\chi_v\diamond\omega=0$, which in turns implies $\chi_v\wedge\omega=0$ and $\chi_v\triangle_1^\perp\omega=0$. Hence, the second and third equations in Theorem \ref{thm:spinorialcurving} automatically hold. Similarly, the condition $H_b(u^\sharp,v^\sharp)=-\mu*_{uv}H_b^\perp$ implies $H_b^\perp\triangle_2^\perp\omega=-\omega(H_b(u^\sharp,v^\sharp)^\sharp)$, thus the first equation in Theorem \ref{thm:spinorialcurving} is automatically satisfied.
\end{proof}

\begin{remark}
Corollary \ref{cor:b_curving_if_H_dual} gives the six-dimensional Lorentzian analog of one of the directions of the well-known correspondence between spinorial connections on a Riemannian four-manifold and (anti-)self-dual instantons. 
\end{remark}

\noindent
Theorem \ref{thm:spinorialcurving} has been established for general $\mathbb{C}^{\ast}$-bundle gerbes, whose associated curvature is in general a complex three-form. The case of $\U(1)$-bundle gerbes, whose curvature is purely imaginary, follows as a corollary.

\begin{cor}
\label{cor:U(1)spinorialb}
A pair $(\eta,b)$ is a spinorial curving on a $\U(1)$-bundle gerbe $(\cC,M,g)$ if and only if there exists an isotropic one-form $v\in \Omega^1(M)$ conjugate to the Dirac current $u$ of $\eta$ such that the following conditions are satisfied:
\begin{equation*}
H_b(u^\sharp,v^\sharp)=-\mu*_{uv}H_b^\perp,\qquad *_{uv}(H_b(u^\sharp,v^\sharp)\wedge u+H_b(u^\sharp))=\mu(H_b(u^\sharp,v^\sharp)\wedge u+H_b(u^\sharp)),
\end{equation*}

\noindent
where $\ast_{uv}$ denotes the Hodge dual induced by $h$ on the orthogonal complement $V_{uv}$ to $u$ and $v$. 
\end{cor}

\begin{example}
Consider $(M,g) = (\mathbb{R}^2 \oplus \mathbb{R}^4 , \delta_{1,1}\oplus h)$ as the direct product of two-dimensional Minkowski space $(\mathbb{R}^2 , \delta_{1,1})$ and $\mathbb{R}^4$ equipped with a Riemannian metric $h$. We consider a trivialized $\U(1)$-bundle gerbe on $\mathbb{R}^6$ with trivial connective structure. Then, we can consider curvings on such bundle gerbe simply as two-forms on $\mathbb{R}^6$, which we split as follows in terms of the splitting $\mathbb{R}^6 = \mathbb{R}^2 \oplus \mathbb{R}^4$ with Cartesian coordinates $(x_u,x_v,x_1,\hdots , x_4)$ in which $\delta_{1,1} = \dd x_u \odot \dd x_v$. We write:
\begin{equation*}
b = f\, \dd x_u \wedge \dd x_v + \dd x_u \wedge a_u + \dd x_v \wedge a_v + b^{\perp}
\end{equation*}

\noindent
for a function $f\colon \mathbb{R}^6\to \mathbb{R}$, a pair of one-forms $a_u , a_v\colon \mathbb{R}^6 \to (\mathbb{R}^4)^{\ast}$ and a two-form $b^{\perp} \colon \mathbb{R}^6 \to \wedge^2 (\mathbb{R}^4)^{\ast}$. The curvature of $b$ is readily found to be:
\begin{eqnarray*}
& H_b = \dd b = \dd x_u \wedge \dd x_v \wedge \dd^{\perp}f - \dd x_u \wedge \dd x_v \wedge \partial_{x_v} a_u - \dd x_u   \wedge \dd^{\perp} a_u \\
& - \dd x_v \wedge \dd x_u \wedge \partial_{x_u} a_v - \dd x_v \wedge \dd^{\perp} a_v + \dd x_u \wedge \partial_{x_u} b^{\perp} + \dd x_v \wedge \partial_{x_v} b^{\perp} + \dd^{\perp} b^{\perp} \\
& = \dd x_u \wedge \dd x_v \wedge (\dd^{\perp}f + \partial_{x_u} a_v - \partial_{x_v} a_u) + \dd x_u   \wedge (\partial_{x_u} b^{\perp} - \dd^{\perp} a_u) + \dd x_v   \wedge (\partial_{x_v} b^{\perp} - \dd^{\perp} a_v) + \dd^{\perp} b^{\perp},
\end{eqnarray*}

\noindent
where $\dd^{\perp}$ denotes the exterior derivative on $\mathbb{R}^4$. Take the Hermitian square of $\eta$ to be given by: $$\widehat{\alpha} = \dd x_u + i \dd x_u \wedge \omega -\mu \ast \dd x_u,$$ where $\omega$ is a symplectic form on $\mathbb{R}^4$ satisfying $\ast_h \omega = -\mu \omega$, where $\ast_h$ denotes the Hodge dual on $(\mathbb{R}^4 , h)$. Hence, by Corollary \ref{cor:U(1)spinorialb}, such $b$ is a spinorial curving if and only if:
\begin{equation*}
\mu \ast_h \dd^{\perp} b^{\perp} = \dd^{\perp}f + \partial_{x_u} a_v - \partial_{x_v} a_u \, ,\qquad \ast_h (\partial_{x_v} b^{\perp} - \dd^{\perp} a_v) = \mu (\partial_{x_v} b^{\perp} - \dd^{\perp} a_v).
\end{equation*}

\noindent
Hence, taking $a_v$ as closed, $b^{\perp}$ independent of $x_v$, and setting:
\begin{equation*}
a_u = \int_{0}^{x_u}  (\dd^{\perp}f + \partial_{x_u} a_v )\, \dd s  - \mu \, x_u \ast_h \dd^{\perp} b^{\perp} + a_o
\end{equation*}

\noindent
for a one-form $a_o$, we obtain an example of spinorial curving on $(\mathbb{R}^2 \oplus \mathbb{R}^4 , \delta_{1,1}\oplus h)$.
\end{example}


\appendix



\section{The Kähler-Atiyah model for the Clifford algebra}
\label{app:KA}


Let $V$ be an oriented $d$-dimensional real vector space equipped with a non-degenerate metric $h$ of signature $(p,q)$ and let $(V^*,h^*)$ be the quadratic space dual to $(V,h)$, where $h^*$ denotes the metric dual to $h$. Let $\mathrm{Cl}(V^*,h^*)$ be the real Clifford algebra of the quadratic vector space $(V^*,h^*)$, viewed as a $\Z_2$-graded associative algebra with decomposition: $$\mathrm{Cl}(V^*,h^*)=\mathrm{Cl}^{\mathrm{ev}}(V^*,h^*)\oplus\mathrm{Cl}^{\mathrm{odd}}(V^*,h^*).$$

In our conventions the Clifford algebra satisfies: $$\theta^2=h^*(\theta,\theta),\quad\theta\in V^*.$$

We identify the real Clifford algebra $\mathrm{Cl}(V^*,h^*)$ with the \emph{Kähler-Atiyah algebra} of $(V^*,h^*)$, which we denote by $(\wedge V^*,\diamond)$ (see \cite{Chev54,Chev55}). The map $\diamond\colon\wedge V^*\times\wedge V^*\to\wedge V^*$ denotes the \emph{geometric product} determined by $h$. This is given by the linear and associative extension of the following expression: $$\theta\diamond\alpha=\theta\wedge\alpha+\iota_{\theta^\sharp}\alpha,\quad\theta\in V^*,\,\alpha\in\wedge V^*,$$ where $\theta^\sharp\in V$ denotes the $h$-dual vector of the one-form $\theta$.\medskip

In order to do computations with the geometric product it is convenient to introduce the \emph{generalized products} of $(V^*,h^*)$. These are the bilinear operators:
\begin{equation}
\label{eq:bilinearDelta}
\triangle_k\colon\wedge^i V^*\times\wedge^j V^*\to\wedge^{i + j - 2k}V^*,
\end{equation}

\noindent
where $k=0,\ldots,d$, defined through the expansion: 
\begin{equation*}
\alpha\diamond\beta=\sum_{k=0}^d(-1)^{\binom{k+1}{2}+jk}\alpha\triangle_k\beta,\quad\alpha\in\wedge^j V^*,\,\beta\in\wedge V^*.
\end{equation*}

\noindent
Choosing a basis $\{e_1,\ldots,e_d\}$ of $V$ we can express the generalized products as: 
\begin{equation*}
    \alpha\triangle_k\beta=\tfrac{1}{k!}h^{i_1j_1}\cdots h^{i_kj_k}(\iota_{e_{i_1}}\ldots\iota_{e_{i_k}}\alpha)\wedge(\iota_{e_{j_1}}\ldots\iota_{e_{j_k}}\beta).
\end{equation*}

\noindent
Below, we collect some useful properties of the generalized products that we will use in the computations.

\begin{prop}[{\cite{LB13,LBC13,LBC16}}]\label{prop:appendix_properties}
    Let $\alpha\in\wedge^aV^*$ and $\beta\in\wedge^bV^*$. Then: \begin{itemize}
        \item $\alpha\triangle_k\beta=0$ if $k>a$ or $k>b$.
        \item $\alpha\triangle_k\beta=(-1)^{(a-k)(b-k)}\beta\triangle_k\alpha$. In particular $\alpha\triangle_k\alpha=0$ if $a-k$ is odd.
        \item $\alpha\triangle_0\beta=\alpha\wedge\beta$ and if $b=a$ then $\alpha\triangle_a\beta=\escal{\alpha,\beta}$.
        \item $\alpha\triangle_a(*\beta)=*(\beta\wedge\alpha)$ if $a+b\leq d$.
    \end{itemize}
\end{prop}

As an associative and unital algebra, the Kähler-Atiyah algebra $(\wedge V^*,\diamond)$ is isomorphic to the Clifford algebra $\mathrm{Cl}(V^*,h^*)$ through the $h$-dependent \emph{Chevalley-Riesz isomorphism} (see \cite{CLS21,LBC13,LBC16}), which we denote by: \begin{equation}\label{eq:ChevRiesz_iso}
    \Psi_0\colon(\wedge V^*,\diamond)\to\mathrm{Cl}(V^*,h^*).
\end{equation}

We denote by $\pi$ the \emph{signature automorphism} of the Kähler-Atiyah algebra, which is defined as the unique unital algebra automorphism which acts as minus the identity on $V^*\subset\wedge V^*$, and by $\tau$ the \emph{reversion anti-automorphism}, defined as the unique unital algebra anti-automorphism which acts as the identity on $V^*$. For $\alpha\in\wedge^aV^*$ we have: \begin{equation*}
    \pi(\alpha)=(-1)^a\alpha,\quad\tau(\alpha)=(-1)^{\binom{a}{2}}\alpha.
\end{equation*}

The pseudo-Riemannian volume form $\nu\in\wedge^dV^*$ satisfies the following properties:

\begin{lemma}\label{lemma:product_volume_form}
Let $(V,h)$ be a quadratic vector space of signature $(p,q)$ and dimension $d$. Then the following identities hold for all $\alpha\in\wedge V^*$: \begin{equation}\label{eq:diamond_volume}
\alpha\diamond\nu=*\tau(\alpha)\quad\text{and}\quad\nu\diamond\alpha=*(\pi^{d-1}\circ\tau)(\alpha).
\end{equation}
\end{lemma}

\begin{proof}
On one hand we have $*\alpha=\tau(\alpha)\diamond\nu$, thus $\alpha\diamond\nu=\tau(\tau(\alpha))\diamond\nu=*\tau(\alpha)$ since $\tau^2=\Id$. On the other hand we have that $\nu\diamond\alpha=\pi^{d-1}(\alpha)\diamond\nu$, thus $\nu\diamond\alpha=\pi^{d-1}(\alpha)\diamond\nu=*(\tau\circ\pi^{d-1})(\alpha)=*(\pi^{d-1}\circ\tau)(\alpha)$ since $\tau\circ\pi=\pi\circ\tau$.
\end{proof}

\begin{remark}
The second equation in \eqref{eq:diamond_volume} implies that the left-multiplication with the volume form $\nu$ depends on the parity of the dimension $d$ rather than on the signature $(p,q)$ of the metric $h$. For $d$ even we obtain the same equations as in \cite[Lemma 3.24]{CLS21}, whereas for $d$ odd, we obtain the same equations as in \cite[Equation 4]{Sha24}.
\end{remark}


\bibliographystyle{myamsplain}
\bibliography{biblio}

\end{document}